%% file: MGFF_CrossingProba.tex
\pdfoutput=1
\documentclass[11pt]{article}
\usepackage{authblk}
\usepackage[toc,page]{appendix}
\usepackage{color}
\usepackage{helvet}         % selects\textbf{\textbf{•}} Helvetica as sans-serif font
\usepackage{courier}        % selects Courier as typewriter font
\usepackage{type1cm}        % activate if the above 3 fonts are
\usepackage{makeidx}         % allows index generation
\usepackage{graphicx}        % standard LaTeX graphics tool
\usepackage{multicol}        % used for the two-column index
\usepackage[bottom]{footmisc}% places footnotes at page bottom
\usepackage{amsmath}
\usepackage{amssymb}
\usepackage{bbm}
\usepackage{amsthm}
\usepackage{subcaption}
\usepackage{sidecap}
\usepackage{comment}
\usepackage[font=small]{caption}
\usepackage{pdflscape}
\newtheorem{theorem}{Theorem}
\newtheorem{corollary}[theorem]{Corollary}

\newtheorem{lemma}[theorem]{Lemma}

\newtheorem{proposition}[theorem]{Proposition}

\usepackage[top=2cm, bottom=2cm, left=2cm, right=2cm]{geometry}
\numberwithin{theorem}{section}
\numberwithin{figure}{section}
\numberwithin{equation}{section}

\DeclareMathOperator{\SLE}{SLE}

\DeclareMathOperator{\GFF}{GFF}

\DeclareMathOperator{\diam}{diam}

\DeclareMathOperator{\LP}{LP}
\DeclareMathOperator{\DP}{DP}
\DeclareMathOperator{\Pair}{PP}
\DeclareMathOperator{\mgff}{mGFF}

\newcommand*\samethanks[1][\value{footnote}]{\footnotemark[#1]}

\begin{document}

\title{Scaling Limits of Crossing Probabilities in Metric Graph GFF}

\author{Mingchang Liu\thanks{Supported by Chinese Thousand Talents Plan for Young Professionals and Beijing Natural Science Foundation (Z180003).}}
\affil{Tsinghua University, China\\Email: liumc18@mails.tsinghua.edu.cn}

\author{Hao Wu\samethanks}
\affil{Tsinghua University, China\\Email: hao.wu.proba@gmail.com.}

\date{}
\maketitle
\vspace{-1cm}
\begin{center}
\begin{minipage}{0.85\textwidth}
\abstract{We consider metric graph Gaussian free field (GFF) defined on polygons of $\delta\mathbb{Z}^2$ with alternating boundary data. The crossing probabilities for level-set percolation of metric graph GFF have scaling limits. When the boundary data is well-chosen, the scaling limits of crossing probabilities can be explicitly constructed as ``fusion" of multiple SLE$_4$ pure partition functions. 
\smallbreak
 }
\end{minipage}
\end{center}

\tableofcontents
\newpage

\input{tex/macro}

\section{Introduction}
\input{tex/intro}

%%%%%%%%%%
\section{Preliminaries}
\label{sec::pre}
\input{tex/pre}
%%%%%%%%%%
\section{Partition functions for multiple SLEs}
\label{sec::partitionfunctions}
\input{tex/partitionfunctions}

%%%%%%%%%%%
\section{Connection probabilities for level lines in GFF}
\label{sec::continuumGFF}
\input{tex/gfflevellines}
%%%%%%%%%%%
\section{Metric graph GFF and first passage sets}
\label{sec::mgfffps}
\input{tex/mgff}

%{\small
%\bibliographystyle{alpha}
%\bibliography{bibliography}}

\end{document}

%% file: tex/macro.tex
\newcommand{\eps}{\epsilon}
\newcommand{\ov}{\overline}
\newcommand{\U}{\mathbb{U}}
\newcommand{\T}{\mathbb{T}}
\newcommand{\HH}{\mathbb{H}}
\newcommand{\LA}{\mathcal{A}}
\newcommand{\LB}{\mathcal{B}}
\newcommand{\LC}{\mathcal{C}}
\newcommand{\LD}{\mathcal{D}}
\newcommand{\LF}{\mathcal{F}}
\newcommand{\LK}{\mathcal{K}}
\newcommand{\LE}{\mathcal{E}}
\newcommand{\LG}{\mathcal{G}}
\newcommand{\LL}{\mathcal{L}}
\newcommand{\LM}{\mathcal{M}}
\newcommand{\LO}{\mathcal{O}}
\newcommand{\LQ}{\mathcal{Q}}
\newcommand{\CP}{\mathcal{P}}
\newcommand{\LT}{\mathcal{T}}
\newcommand{\LS}{\mathcal{S}}
\newcommand{\LU}{\mathcal{U}}
\newcommand{\LV}{\mathcal{V}}
\newcommand{\PartF}{\mathcal{Z}}
\newcommand{\LH}{\mathcal{H}}
\newcommand{\R}{\mathbb{R}}
\newcommand{\C}{\mathbb{C}}
\newcommand{\N}{\mathbb{N}}
\newcommand{\Z}{\mathbb{Z}}
\newcommand{\E}{\mathbb{E}}
\newcommand{\PP}{\mathbb{P}}
\newcommand{\QQ}{\mathbb{Q}}
\newcommand{\A}{\mathbb{A}}
\newcommand{\one}{\mathbbm{1}}
\newcommand{\bn}{\mathbf{n}}
\newcommand{\MR}{MR}
\newcommand{\cond}{\,\big|\,}
\newcommand{\la}{\langle}
\newcommand{\ra}{\rangle}
\newcommand{\tree}{\Upsilon}
\global\long\def\chamber{\mathfrak{X}}
\newcommand{\KWle}{\overset{()}{\longleftarrow}}
\newcommand{\ABP}[2]{\protect\rotatebox[origin=c]{180}{$#1\A$}}
\newcommand{\AB}{{\mathpalette\ABP\relax}}

%% file: tex/intro.tex
This article concerns crossing probability of level-set percolation of Gaussian free field (GFF) on the square lattice $\Z^2$. For $L>0$, consider the rectangle 
$R_L=\{z: 0<\Re{z}< L, 0< \Im{z}< 1\}$.
Let $y_1, y_2, y_3, y_4$ be its four corners, listed in counterclockwise order with $y_2=0$. For $\delta>0$, let $V_{\delta}=R_L\cap \delta\Z^2$ and let $y_1^{\delta}, y_2^{\delta}, y_3^{\delta}, y_4^{\delta}$ be its four corners, listed in counterclockwise order such that $y_2^{\delta}$ is closest to $y_2$. For two vertices $u, v\in\partial V_{\delta}$, we denote by $(uv)$ the arc of $\partial V_{\delta}$ from $u$ to $v$ in counterclockwise order. Let $\Gamma^{\delta}$ be a discrete GFF (see Section~\ref{subsec::dGFFmGFF}) on $V_{\delta}$ with alternating boundary data: 
\begin{equation*}
\mu \text{ on }(y_1^{\delta}y_2^{\delta})\cup(y_3^{\delta}y_4^{\delta}),\quad -\mu\text{ on }(y_2^{\delta}y_3^{\delta})\cup(y_4^{\delta}y_1^{\delta}),
\end{equation*}
where $\mu>0$ is a positive constant. Let $\tilde{\Gamma}^{\delta}$ be the GFF on the metric graph $\tilde{V}_{\delta}$ (see Section~\ref{subsec::dGFFmGFF}) with the same boundary data.  We are interested in the event that there exists a path in $V_{\delta}$ (resp. $\tilde{V}_{\delta}$) from $(y_1^{\delta}y_2^{\delta})$ to $(y_3^{\delta}y_4^{\delta})$ such that $\Gamma^{\delta}$ (resp. $\tilde{\Gamma}^{\delta}$) is non-negative on this path. We denote this event by 
\[\left\{(y_1^{\delta}y_2^{\delta})\overset{\Gamma^{\delta}\ge 0}{\longleftrightarrow}(y_3^{\delta}y_4^{\delta})\right\}\quad\text{and}\quad \left\{(y_1^{\delta}y_2^{\delta})\overset{\tilde{\Gamma}^{\delta}\ge 0}{\longleftrightarrow}(y_3^{\delta}y_4^{\delta})\right\}\]
for $\Gamma^{\delta}$ and $\tilde{\Gamma}^{\delta}$ respectively. Although, both discrete GFF $\Gamma^{\delta}$ and metric graph GFF $\tilde{\Gamma}^{\delta}$ converge as distributions to the continuum GFF as $\delta\to 0$, the probabilities for such crossing events have distinct scaling limits, as proved in~\cite[Theorem~1.2]{DingWirthWuGFF}. It is then natural to ask whether we are able to give explicit formula for scaling limits of such crossing probabilities.  

The answer to this question relies on Schramm-Sheffield's famous work on level lines of GFF. 
We call $(\Omega; x, y)$ a Dobrushin domain if $\Omega\nsubseteq\C$ is non-empty simply connected and $x, y$ are distinct boundary points. 
In~\cite{SchrammSheffieldDiscreteGFF}, the authors prove that there exists $\lambda=\lambda(\Z^2)>0$ such that the zero level line of discrete GFF on Dobrushin domains of $\delta\Z^2$ with boundary data $\pm \lambda$ converges in distribution to Schramm-Loewner Evolution ($\SLE_4$, see Section~\ref{subsec::sle}). Based on this result, one is able to show that~\cite[Theorem~1.3]{DingWirthWuGFF}, when $\mu=\lambda$,
\begin{equation}\label{eqn::crossingproba_dgff}
\lim_{\delta\to 0}\PP\left[(y_1^{\delta}y_2^{\delta})\overset{\Gamma^{\delta}\ge 0}{\longleftrightarrow}(y_3^{\delta}y_4^{\delta})\right]=q,
\end{equation}  
where $q$ is the cross-ratio of the rectangle: let $\varphi$ be any conformal map from $R_L$ onto the upper-half plane $\HH$ with $\varphi(y_1)<\varphi(y_2)<\varphi(y_3)<\varphi(y_4)$, then 
\begin{equation}\label{eqn::crossratio}
q=\frac{(\varphi(y_2)-\varphi(y_1))(\varphi(y_4)-\varphi(y_3))}{(\varphi(y_3)-\varphi(y_1))(\varphi(y_4)-\varphi(y_2))}.
\end{equation}
This gives the answer to the case of discrete GFF. The authors in~\cite{DingWirthWuGFF} derive~\eqref{eqn::crossingproba_dgff} by showing that the scaling limit of the crossing probability in discrete GFF is the same as the one for continuum GFF whose crossing probability is calculated in~\cite[Theorem~1.4]{PeltolaWuGlobalMultipleSLEs}. 

The goal of this article is to derive explicit formula for scaling limits of crossing probability in metric graph GFF. We will show that, when $\mu=2\lambda$, 
\begin{equation}\label{eqn::crossingproba_mgff}
\lim_{\delta\to 0}\PP\left[(y_1^{\delta}y_2^{\delta})\overset{\tilde{\Gamma}^{\delta}\ge 0}{\longleftrightarrow}(y_3^{\delta}y_4^{\delta})\right]=q^4,
\end{equation}
where $q$ is the cross-ratio of the rectangle as in~\eqref{eqn::crossratio}. Our method relies on the following two ingredients: a). Aru-Lupu-Sep\'{u}lveda's work on the convergence of first passage sets of metric graph GFF~\cite{AruLupuSepulvedaFPSGFFCVGISO}, and b). a generalization of Peltola and the second author's work~\cite{PeltolaWuGlobalMultipleSLEs} on crossing probabilities in continuum GFF. 

In fact, we are able to give answer in a more general setting: we can calculate the scaling limits of crossing probabilities for the metric graph GFF with alternating boundary data on a polygon with $2N$ marked points on the boundary. To state our main result, we first introduce some notations about planar link patterns.

\medbreak
For $p\in\Z_{>0}$, we call $(\Omega; x_1, \ldots, x_{p})$ a polygon if $\Omega\nsubseteq\C$ is non-empty simply connected and $x_1, \ldots, x_{p}$ are $p$ boundary points in counterclockwise order lying on locally connected boundary segments. 
We first introduce planar pair partitions. Suppose $p=2N$ is even and suppose there are $N$ non-intersecting simple curves in $\Omega$ connecting the $2N$ boundary points pairwise. These $N$ curves form a planar pair partition that we denote by $\alpha=\{\{a_1, b_1\}, \ldots \{a_N, b_N\}\}$ with $\{a_1, b_1, \ldots, a_N, b_N\}=\{1,2, \ldots, 2N\}$. We call the pairs $\{a, b\}$ in $\alpha$ links. We denote by $\Pair_N$ the set of planar pair partitions with $2N$ points and set $\Pair=\bigsqcup_{N\ge 0}\Pair_N$. 

Next, we introduce general planar link patterns. The planar pair partitions  then arise as a special case. Suppose $(\Omega; x_1, \ldots, x_{p})$ is a polygon. Fix a multiindex $\varsigma=(s_1, \ldots, s_p)\in\Z_{>0}^p$ such that $\sum_{i=1}^p s_i$ is even and denote by 
\[\ell=\frac{1}{2}\sum_{i=1}^p s_i\in\Z_{>0}. \]
Suppose there are $\ell$ simple curves in $\Omega$ connecting the $p$ boundary points pairwise such that they do not intersect except at their common end points. These $\ell$ curves form a planar link pattern that we call planar $\ell$-link patterns of $p$ points. Precisely, we call planar $\ell$-link patterns of $p$ points with index valences $\varsigma=(s_1, \ldots, s_p)$ as collections $\omega=\{\{a_1, b_1\}, \ldots, \{a_{\ell}, b_{\ell}\}\}$ of $\ell$-links $\{a, b\}$ which connect a pair of distinct indices $a, b\in\{1, 2, \ldots, p\}$ such that, for any $i\in\{1, 2, \ldots, p\}$, the index $i$ is an endpoint of exactly $s_i$ links and that none of the links intersect except at their common endpoints. We denote the collection of $\ell$-link patterns of $p$ points with index valences $\varsigma$ by $\LP_{\varsigma}$. 
With such definition, when $p=2N$ is even, the planar $N$-link pattern of $2N$ points with index valences $\varsigma=(1, \ldots, 1)$ is a planar pair partition and $\LP_{(1,\ldots, 1)}=\Pair_N$. 

\medbreak
In this article, we are interested in planar $2N$-link patterns of $2N$ points with index valences $\varsigma=(2, \ldots, 2)$, see Figure~\ref{fig::mgfflinkpatterns} for $N=2$. With the above definition, the collection of such planar link patterns is denoted by $\LP_{(2,\ldots, 2)}$ where the index has length $2N$. 
 
\begin{figure}[ht!]
\begin{center}
\includegraphics[width=0.24\textwidth]{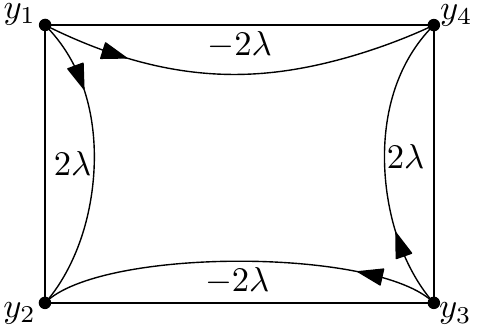}$\quad$
\includegraphics[width=0.24\textwidth]{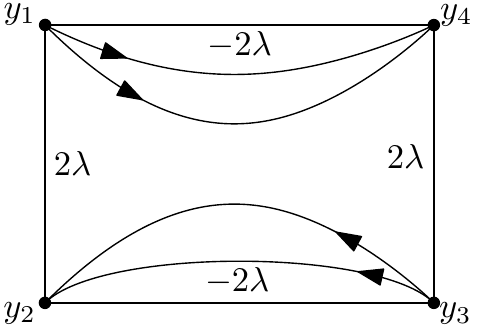}$\quad$
\includegraphics[width=0.24\textwidth]{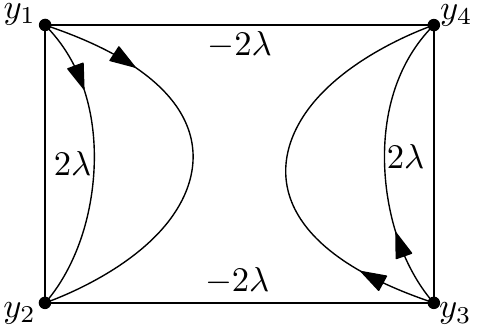}
\end{center}
\caption{\label{fig::mgfflinkpatterns} Consider metric graph GFF in rectangle with alternating boundary data when $\mu=2\lambda$. Consider the positive first passage sets attached to $(y_1^{\delta}y_2^{\delta})$ and to $(y_3^{\delta}y_4^{\delta})$ and consider the negative first passage sets attached to $(y_2^{\delta}y_3^{\delta})$ and to $(y_4^{\delta}y_1^{\delta})$. Their frontier form a planar 4-link pattern of 4 points with  index valences $\varsigma=(2,2,2,2)$. There are three possibilities as indicated in the figure. In the right panel, there is negative vertical crossing of the metric graph GFF. In the middle panel, there is positive horizontal crossing.  In the left panel, there is neither positive horizontal crossing nor negative vertical crossing. }
\end{figure}

\medbreak

Fix a polygon $(\Omega; y_1, \ldots, y_{2N})$ such that $\Omega\subset [-C, C]^{2}$ for some $C>0$. 
Suppose $\{\left(\Omega^{\delta};y_{1}^{\delta},\ldots,y_{2N}^{\delta}\right)\}_{\delta>0}$ are polygons such that $\Omega^{\delta}\subset [-C, C]^{2}$ for all $\delta>0$. Suppose $\left(\Omega^{\delta};y_{1}^{\delta},\ldots,y_{2N}^{\delta}\right)$ converges to $\left(\Omega;y_{1},\ldots,y_{2N}\right)$ as $\delta\to 0$ in the following sense: 
\begin{align}\label{eqn::topology}
[-C, C]^{2}\setminus\Omega^{\delta}\, \text{converges to } [-C, C]^{2}\setminus\Omega\, \text{in Hausdorff metric and } y_{i}^{\delta}\to y_{i}\, \text{for each } 1\le i\le 2N. 
\end{align}
Consider metric graph GFF $\tilde{\Gamma}^{\delta}$ in $(\Omega^{\delta}; y_1^{\delta}, \ldots, y_{2N}^{\delta})$ with alternating boundary data:
\begin{equation}\label{eqn::mgff_boundaryconditions}
2\lambda \text{ on }(y_{2j-1}^{\delta}y_{2j}^{\delta}), \quad \text{and}\quad -2\lambda\text{ on }(y_{2j}^{\delta}y_{2j+1}^{\delta}), \quad\text{for }j\in\{1,\ldots, N\},
\end{equation}
where $y_{2N+1}=y_1$ by convention. 
Consider positive first passage set (see Section~\ref{subsec::fps}) of $\tilde{\Gamma}^{\delta}$ attached to the boundary segments $(y_{2j-1}^{\delta}y_{2j}^{\delta})$, and negative first passage set attached to the boundary segments $(y_{2j}^{\delta}y_{2j+1}^{\delta})$, for $j\in\{1, \ldots, N\}$. The frontier of these first passage sets is a collection of $2N$ curves connecting the $2N$ boundary points so that their end points form a planar $2N$-link pattern of $2N$ points with index valences $\varsigma=(2, 2, \ldots, 2)$. See Figure~\ref{fig::mgfflinkpatterns}. We denote the link pattern by $\LA^{\delta}$. Our main result is the following. 

\begin{theorem}\label{thm::main}
Fix $N\ge 1$ and the index valences $\varsigma=(2,\ldots, 2)$ of length $2N$. 
Consider the frontier of first passage sets of metric graph GFF in $\Omega^{\delta}$ with alternating boundary data~\eqref{eqn::mgff_boundaryconditions}. 
The frontier is a collection of $2N$ curves connecting the $2N$ boundary points whose end points form a planar link pattern $\LA^{\delta}\in\LP_{\varsigma}$. 
We have
\begin{equation*}
\lim_{\delta\to 0}\PP[\LA^{\delta}=\hat{\alpha}]=\LM_{\omega, \tau(\hat{\alpha})}\frac{\PartF_{\hat{\alpha}}(\Omega; y_1, \ldots, y_{2N})}{\PartF_{\mgff}^{(N)}(\Omega; y_1, \ldots, y_{2N})}, \quad \text{for all }\hat{\alpha}\in\LP_{\varsigma}, 
\end{equation*}
where the coefficient $\LM_{\omega, \tau(\hat{\alpha})}$ is given by Lemma~\ref{lem::thmmaincoefficient}, the function $\PartF_{\hat{\alpha}}$ is  given by Proposition~\ref{prop::purepartitionfusionall}, and 
\[\PartF_{\mgff}^{(N)}(\Omega; y_1, \ldots, y_{2N})=\sum_{\hat{\alpha}\in\LP_{\varsigma}}\LM_{\omega, \tau(\hat{\alpha})}\PartF_{\hat{\alpha}}(\Omega; y_1, \ldots, y_{2N}). \]
\end{theorem}

The conclusion in~\eqref{eqn::crossingproba_mgff} is the special case of Theorem~\ref{thm::main} when $N=2$, see Corollary~\ref{cor::crossingproba_mgff}. 
The definition for $\LM_{\omega, \tau(\hat{\alpha})}$ and $\PartF_{\hat{\alpha}}$ is quite involved, and we omit it from the introduction. Nevertheless, let us mention in the introduction the formula for $\PartF_{\mgff}^{(N)}$ and nice properties that  $\PartF_{\hat{\alpha}}$ enjoys. When $\Omega=\HH$ and $y_1<\cdots<y_{2N}$, we write 
\[\PartF_{\hat{\alpha}}(y_1, \ldots, y_{2N})=\PartF_{\hat{\alpha}}(\HH; y_1, \ldots, y_{2N}),\quad \PartF_{\mgff}^{(N)}(y_1, \ldots, y_{2N})=\PartF_{\mgff}^{(N)}(\HH; y_1, \ldots, y_{2N}).\]
Then, we have
\begin{equation}\label{eqn::thmmainZtotal}
\PartF_{\mgff}^{(N)}(y_1, \ldots, y_{2N})=\prod_{1\le i<j\le 2N}(y_j-y_i)^{2(-1)^{j-i}}. 
\end{equation}

\begin{proposition}\label{prop::mGFFpdecov}
Fix $N\ge 1$ and the index valences $\varsigma=(2,\ldots, 2)$ of length $2N$. 
For any $\hat{\alpha}\in\LP_{\varsigma}$, the function $\PartF_{\hat{\alpha}}: \chamber_{2N}\to \R_{>0}$ given by Proposition~\ref{prop::purepartitionfusionall} satisfies the following PDE system and conformal covariance with $\kappa=4$ and $H=1$. 
\begin{itemize}
\item Partial differential equations of third order: for all $j\in\{1, \ldots, 2N\}$, 
\begin{align}\label{eqn::mGFFpde3rd}
&\left[\frac{\partial^3}{\partial y_j^3}-\frac{16}{\kappa}\LL_{-2}^{(j)}\frac{\partial}{\partial y_j}+\frac{8(8-\kappa)}{\kappa^2}\LL_{-3}^{(j)}\right]
\PartF_{\hat{\alpha}}(y_1, \ldots, y_{2N})=0, \\
&\text{where }\LL_{-2}^{(j)}:=\sum_{i\neq j}\left(\frac{H}{(y_i-y_j)^2}-\frac{1}{y_i-y_j}\frac{\partial}{\partial y_i}\right),\quad 
\LL_{-3}^{(j)}:=\sum_{i\neq j}\left(\frac{2H}{(y_i-y_j)^3}-\frac{1}{(y_i-y_j)^2}\frac{\partial}{\partial y_i}\right). \notag
\end{align}
\item M\"{o}bius covariance: for all M\"{o}bius maps $\varphi$ of $\HH$ such that $\varphi(y_1)<\cdots <\varphi(y_{2N})$, 
\begin{equation}\label{eqn::mGFFcov}
\PartF_{\hat{\alpha}}(y_1, \ldots, y_{2N})=\prod_{i=1}^{2N}\varphi'(y_i)^H \times \PartF_{\hat{\alpha}}(\varphi(y_1), \ldots, \varphi(y_{2N})). 
\end{equation}
\end{itemize} 
\end{proposition}

\medbreak
\noindent\textbf{Outline.}
In Section~\ref{sec::pre}, we will give preliminaries on planar link patterns and SLEs. 
In Section~\ref{sec::partitionfunctions}, we will introduce multiple SLE partition functions and prove a preliminary result about ``fusion" of partition functions---Proposition~\ref{prop::purepartitionFusion}. 
In Section~\ref{sec::continuumGFF}, we will introduce continuum GFF and a result on connection probabilities---Theorem~\ref{thm::GFFconnectionproba}. 
In Section~\ref{sec::mgfffps}, we will introduce metric graph GFF and combine the results from preceding sections to complete the proof of Theorem~\ref{thm::main} and Proposition~\ref{prop::mGFFpdecov}. 
The functions $\{\PartF_{\hat{\alpha}}: \hat{\alpha}\in\LP_{\varsigma}\}$ in Theorem~\ref{thm::main} are ``fusion" of SLE pure partition functions, see Proposition~\ref{prop::purepartitionfusionall}. 
The scaling limit of the crossing probabilities in Theorem~\ref{thm::main} can be interpreted as certain connection probabilities in continuum GFF which can be explicitly constructed from Theorem~\ref{thm::GFFconnectionproba} and Proposition~\ref{prop::purepartitionfusionall}.
Then, we will check the PDE in Proposition~\ref{prop::mGFFpdecov} from Proposition~\ref{prop::purepartitionFusion}. Note that the third order PDE and the conformal covariance in Proposition~\ref{prop::mGFFpdecov} are not surprising. SLE partition functions can be understood as certain correlation functions in terms of conformal field theory. Then the third order PDE and the conformal covariance can be obtained by specific fusion channel, see~\cite{PeltolaBasisPDE, PeltolaCFTSLE}.

\medbreak
\noindent\textbf{Acknowledgment.}
We acknowledge Jian Ding, Titus Lupu, and Mateo Wirth  for helpful discussion on GFF. We acknowledge Eveliina Peltola for stimulating discussion about partition functions for multiple SLEs.

%% file: tex/pre.tex
\subsection{Planar pair partitions and Dyck paths}
\label{subsec::pairpartitionDyckpath}
In this section, we will give a one-to-one correspondence between planar pair partitions and Dyck paths. 
A Dyck path is a walk on $\Z_{\ge 0}$ with steps of length one, starting and ending at zero: $\alpha : \{ 0,1, \ldots,2N \} \to \Z_{\ge 0}$ such that $\alpha(0) = \alpha(2N) = 0,$ and $| \alpha(k) - \alpha(k-1) | = 1$ for all $k\in \{1, 2, \ldots, 2N\}$. 
For $N \ge 1$, we denote the set of all Dyck paths of $2N$ steps by
$\DP_N$. There is a natural partial order on Dyck paths: 
\begin{equation}\label{eqn::Dychpathspartialorder}
\alpha\preceq\beta\quad\text{if and only if}\quad\alpha(k)\le\beta(k),\ \text{for all}\ k\in\{0,1,\ldots, 2N\}.
\end{equation}
We set $\DP=\bigsqcup_{N\ge 0}\DP_N$. 

\medbreak
To each planar pair partition $\alpha \in \Pair_N$, we write it as  
\begin{align} \label{eqn::LPtoDP_order}
\alpha = \{ \{a_1, b_1\}, \ldots, \{a_N, b_N\} \} ,
\quad \text{where } a_1 < a_2 < \cdots < a_N \text{ and }
a_j < b_j , \text{ for all } j \in \{ 1, \ldots, N \} .
\end{align}
We associate it with a Dyck path, also denoted by 
$\alpha \in \DP_N$, as follows. We set $\alpha(0) = 0$ and, for all $k \in \{1, \ldots, 2N\}$, we set
\begin{align} \label{eqn::LPtoDP} 
\alpha(k) = \begin{cases}
\alpha(k-1) + 1 , & \text{if } k\in\{a_1, a_2, \ldots, a_N\}, \\
\alpha(k-1) - 1 , & \text{if } k\in\{b_1, b_2, \ldots, b_N\} .
\end{cases}
\end{align}
One may check, this defines a Dyck path $\alpha \in \DP_N$. Conversely, for any Dyck path 
$\alpha : \{ 0,1, \ldots,2N \} \to \Z_{\ge 0}$, we associate a planar pair partition $\alpha$
by giving to each up-step (i.e., step away from zero) an index $a_r$, for $r = 1,2,\ldots, N$, and 
to each down-step (i.e., step towards zero) an index $b_s$, for $s = 1,2,\ldots, N$, and setting 
$\alpha := \{ \{a_1, b_1\}, \ldots, \{a_N, b_N\} \}$. 
These two mappings define a bijection between $\Pair_N$ and $\DP_N$. We thus identify the elements $\alpha$ of these two sets and use the indistinguishable notation
$\alpha \in \Pair_N$ and $\alpha \in \DP_N$ for both.

\medbreak
For a Dyck path $\alpha\in\DP_N$, we say that $\alpha$ has a local maximum at $j$ if $\alpha(j)-\alpha(j-1)=1$ and $\alpha(j+1)-\alpha(j)=-1$, and we denote $\wedge^j\in\alpha$; we say that $\alpha$ has a local minimum at $j$ if $\alpha(j)-\alpha(j-1)=-1$ and $\alpha(j+1)-\alpha(j)=1$, and we denote $\vee_j\in\alpha$; we say that $\alpha$ has a slope at $j$ if otherwise, and we denote $\times_j\in\alpha$. 
We say that $\alpha$ has a local extremum at $j$ if $\alpha$ has a local minimum or maximum at $j$, and we denote $\lozenge_j\in\alpha$. 

If a planar pair partition $\alpha \in \Pair_N$ has a link $\{j, j+1\} \in \alpha$, then $\wedge^j\in\alpha$. 
Let $\alpha/\{j, j+1\}$ denote the planar pair partition by removing from $\alpha$ the link $\{j, j+1\}$ and relabelling the remaining indices by $1, 2, \ldots, 2N-2$. In terms of Dyck path, we denote this operation by $\alpha/\wedge^j\in \DP_{N-1}$. We define operation $\alpha/\vee_j\in\DP_{N-1}$ analogously when $\alpha$ has a local minimum at $j$. When $\alpha$ has a local extremum at $j$, we denote such operation by $\alpha/\lozenge_j$. If $\alpha$ has a local minimum at $j$, we associate $\alpha$ with another Dyck path by converting the local minimum at $j$ to local maximum, and denote this operation by $\alpha\uparrow\lozenge_{j}$. 

\subsection{From planar link pattern to planar pair partition}
\label{subsec::linkpatterntopairpartition}
Fix an index valences $\varsigma=(s_1, \ldots, s_p)\in\Z_{>0}^p$ such that $\sum_{i=1}^p s_i$ is even and we denote this even number by $2\ell$. Recall that $\LP_{\varsigma}$ is the collection of all planar $\ell$-link patterns of $p$ points with index valences $\varsigma$. We define a natural map which associates to each planar link pattern a planar pair partition. This map, denoted by 
\[\tau: \LP_{\varsigma}\to \Pair_{\ell}, \qquad \hat{\alpha}\mapsto \tau(\hat{\alpha}),\]
is defined as following: 
in $\hat{\alpha}$, for each $j\in\{1, 2, \ldots, p\}$, we split the $j$th point to $s_j$ distinct points and attach the $s_j$ links of $\hat{\alpha}$ ending there to these new $s_j$ points so that each of them has valence one. See Figure~\ref{fig::linkpatterntopairpartition}. 

\begin{figure}[ht!]
\begin{center}
\includegraphics[width=0.5\textwidth]{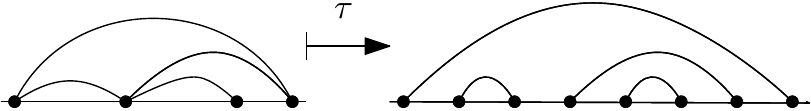}
\end{center}
\caption{\label{fig::linkpatterntopairpartition} In this figure, we have a planar link pattern with index valences $\varsigma=(2,3,1,2)$. It is associated to a planar pair partition by splitting the four points into eight points according to the valences and attaching the corresponding links.}
\end{figure}

\subsection{Loewner chain and SLE}
\label{subsec::sle}
We call a compact subset $K$ of $\overline{\HH}$ an $\HH$-hull if $\HH\setminus K$ 
is simply connected. Riemann's mapping theorem asserts that there exists a unique conformal map 
$g_K$ from $\HH\setminus K$ onto $\HH$ with the property that $\lim_{z\to\infty}|g_K(z)-z|=0$.
We say that $g_K$ is normalized at $\infty$.
%and we call $a(K):=\lim_{z\to \infty} z(g_t(z)-z)$ the \textit{half-plane capacity} of $K$. 

Consider families of conformal maps $(g_{t}, t\ge 0)$ 
obtained by solving the Loewner equation: for each $z\in\mathbb{H}$,
\begin{align*}
\partial_{t}{g}_{t}(z)=\frac{2}{g_{t}(z)-W_{t}}, \qquad \qquad  g_{0}(z)=z,
\end{align*}
where $(W_t, t\ge 0)$ is a real-valued continuous function, which we call the driving function. 
Let $T_z$ be the swallowing time of $z$ defined as $\sup\{t\ge 0 \colon \inf_{s\in[0,t]}|g_{s}(z)-W_{s}|>0\}$.
Denote $K_{t}:=\overline{\{z\in\mathbb{H}: T_{z}\le t\}}$.
Then, $g_{t}$ is the unique conformal map from $H_{t}:=\mathbb{H}\setminus K_{t}$ onto $\mathbb{H}$ normalized at $\infty$. 
The collection of $\HH$-hulls $(K_{t}, t\ge 0)$ associated with such maps is called a Loewner chain.

\medbreak
Fix $\kappa > 0$. The Schramm-Loewner Evolution $\SLE_{\kappa}$ in $\HH$ from $0$ to $\infty$ 
is the random Loewner chain $(K_{t}, t\ge 0)$ driven by $W_t=\sqrt{\kappa}B_t$, where $(B_t, t\ge 0)$ is 
the standard one-dimensional 
Brownian motion. Rohde-Schramm prove
in~\cite{RohdeSchrammSLEBasicProperty} 
that $(K_{t}, t\ge 0)$ is almost surely generated by a continuous transient curve, i.e., there almost 
surely exists a continuous curve $\eta$ such that for each $t\ge 0$, $H_{t}$ is the unbounded connected 
component of $\HH \setminus \eta[0,t]$ and $\lim_{t\to\infty}|\eta(t)|=\infty$. 
This random curve is called 
the $\SLE_\kappa$ trace in $\HH$ from $0$ to $\infty$.
When $\kappa\in (0,4]$, 
the $\SLE_{\kappa}$ curves are simple; when $\kappa\in (4,8)$, they have self-touchings;  when $\kappa \ge 8$, they are space-filling. 
In this article, we focus on $\kappa=4$ as $\SLE_4$ is the level line of Gaussian free field, see Section~\ref{sec::continuumGFF}. 

%% file: tex/partitionfunctions.tex
In this section, we use the following real parameters: 
\begin{align*}
\kappa\in (0,8),\qquad h = \frac{6-\kappa}{2\kappa}, \qquad H=\frac{8-\kappa}{\kappa}. 
\end{align*}

A multiple $\SLE_\kappa$ partition function is a positive smooth function $
\PartF \colon \chamber_{2N} \to \R_{>0}$
defined on the configuration space
$\chamber_{2N} :=\; \{ (x_{1},\ldots,x_{2N}) \in \R^{2N} \colon x_{1} < \cdots < x_{2N} \} $
satisfying the following two properties: 
\begin{itemize}
\item Partial differential equations of second order (PDE): for all $j\in\{1, \ldots, 2N\}$, 
\begin{align}\label{eqn::multipleSLEsPDE}
\left[ \frac{\kappa}{2}\frac{\partial^2}{\partial x_j^2} + \sum_{i\neq j}\left(\frac{2}{x_{i}-x_{j}}\frac{\partial}{\partial x_i} - 
\frac{2h}{(x_{i}-x_{j})^{2}}\right) \right]
\PartF(x_1,\ldots,x_{2N}) =  0 .
\end{align}
\item M\"obius covariance (COV): 
For all M\"obius maps 
$\varphi$ of $\HH$ 
such that $\varphi(x_{1}) < \cdots < \varphi(x_{2N})$, 
\begin{align}\label{eqn::multipleSLEsCOV}
\PartF(x_{1},\ldots,x_{2N}) = 
\prod_{i=1}^{2N} \varphi'(x_{i})^{h} 
\times \PartF(\varphi(x_{1}),\ldots,\varphi(x_{2N})) .
\end{align}
\end{itemize}

In the seires~\cite{FloresKlebanPDE1, FloresKlebanPDE2, FloresKlebanPDE3, FloresKlebanPDE4}, the authors investigate the solution space of the above PDE system. Let $\LS_N$ be the collection of smooth functions $F: \chamber_{2N}\to \C$ satisfying PDE~\eqref{eqn::multipleSLEsPDE}, COV~\eqref{eqn::multipleSLEsCOV}, and the following power law bound: there exist constants $C>0$ and $p>0$ such that, for all $(x_1, \ldots, x_{2N})\in\chamber_{2N}$, we have
\begin{equation}\label{eqn::multipleSLEPLB}
|F(x_1, \ldots, x_{2N})|\le C\prod_{1\le i<j\le 2N}(x_j-x_i)^{\mu_{ij}(p)},\quad\text{where}\quad \mu_{ij}(p)=\begin{cases}
p,&\text{if }|x_j-x_i|>1,\\
-p,&\text{if }|x_j-x_i|\le 1.
\end{cases}
\end{equation} 
They prove that the solution space $\LS_N$ has dimension $C_N:=\#\Pair_N=\frac{1}{N+1}\binom{2N}{N}$ which is the $N$-th Catalan number. We next introduce pure partition functions which give a basis for $\LS_N$. 

\medbreak
The pure partition functions $\PartF_{\alpha} \colon \chamber_{2N} \to \R_{>0}$ are indexed by
planar pair partitions $\alpha \in \Pair_N$. They are positive solutions to PDE~\eqref{eqn::multipleSLEsPDE} and COV~\eqref{eqn::multipleSLEsCOV} and the following  boundary conditions: 
\begin{itemize}
\item Asymptotics (ASY): 
For all $\alpha \in \Pair_N$ and for all $j \in \{1, \ldots, 2N-1 \}$ and $\xi \in (x_{j-1}, x_{j+2})$, we have
\begin{align}\label{eqn::multipleSLEsASY}
\lim_{x_j , x_{j+1} \to \xi} 
\frac{\PartF_\alpha(x_1 , \ldots , x_{2N})}{(x_{j+1} - x_j)^{-2h}} 
=\begin{cases}
0, \quad &
    \text{if } \{j, j+1\} \notin \alpha, \\
\PartF_{\alpha /\{j, j+1\}}(x_{1},\ldots,x_{j-1},x_{j+2},\ldots,x_{2N}), &
    \text{if } \{j, j+1\} \in \alpha ,
\end{cases} 
\end{align}
where $\alpha/\{j,j+1\}\in\Pair_{N-1}$ denotes the link pattern obtained from $\alpha$ by removing the link $\{j, j+1\}$ and relabelling the remaining indices by $1,2, \ldots, 2N-2$, as defined in Section~\ref{subsec::pairpartitionDyckpath}. 
\end{itemize}

It is proved in~\cite[Theorem~1.7]{WuHyperSLE} that, for $\kappa\in (0,6]$, there exists a unique collection $\{\PartF_{\alpha}: \alpha\in\Pair\}$ of smooth functions satisfying the normalization $\PartF_{\emptyset}=1$ and PDE~\eqref{eqn::multipleSLEsPDE}, COV~\eqref{eqn::multipleSLEsCOV}, ASY~\eqref{eqn::multipleSLEsASY}, and the power law bound~\eqref{eqn::multipleSLEPLB}. 
Furthermore, they also satisfy the following refined upper bound~\cite[Theorem~1.7]{WuHyperSLE}: 
For all $\alpha=\{\{a_1, b_1\}, \ldots, \{a_N, b_N\}\}\in\Pair_N$, 
\begin{equation}\label{eqn::multipleSLEPLBoptimal}
0<\PartF_{\alpha}(x_1, \ldots, x_{2N})\le \prod_{i=1}^N |x_{b_i}-x_{a_i}|^{-2h}. 
\end{equation}
The collection $\{\PartF_{\alpha}: \alpha\in\Pair_N\}$ forms a basis for the $C_N$-dimensional solution space $\LS_N$, see~\cite[Proposition~4.5]{PeltolaWuGlobalMultipleSLEs}. See \cite{KytolaPeltolaPurePartitionFunctions, PeltolaWuGlobalMultipleSLEs} for earlier works on pure partition functions. 

\medbreak
The values $\kappa=2$ and $\kappa=4$ are special for pure partition functions: in these two cases, there are known explicit expressions for pure partition functions, see~\cite{KarrilaKytolaPeltolaCorrelationsLERWUST} for $\kappa=2$ and~\cite{PeltolaWuGlobalMultipleSLEs} for $\kappa=4$. In this article, we are interested in the case $\kappa=4$. In this case, pure partition functions are given by conformal block functions which we will introduce in Section~\ref{subsec::conformalblocks}. This collection of functions gives another basis for the solution space $\LS_N$ when $\kappa=4$. In the rest of the article, we assume $\kappa=4$. 

\medbreak
In~\eqref{eqn::multipleSLEsASY}, we see that, if $\{j, j+1\}\in\alpha$, we normalize the function $\PartF_{\alpha}$ by $(x_{j+1}-x_j)^{-2h}$ and we obtain the limiting function $\PartF_{\alpha/\{j,j+1\}}$. The goal of this section is to investigate the correct normalization of $\PartF_{\alpha}$ when $\{j, j+1\}\not\in\alpha$ and to analyze the limiting function. 

\begin{proposition}\label{prop::purepartitionFusion}
Fix $\kappa=4$. For $\alpha\in\Pair_N$ and for $j\in\{1, 2, \ldots, 2N-1\}$, we assume $\{j, j+1\}\not\in\alpha$. For all $\xi\in (x_{j-1}, x_{j+2})$, the following limit exists: 
\begin{equation}\label{eqn::purepartitionFusion}
\PartF_{\alpha/\amalg_j}(x_1, \ldots, x_{j-1}, \xi, x_{j+2}, \ldots, x_{2N}):=\lim_{x_j, x_{j+1}\to \xi} \frac{\PartF_{\alpha}(x_1, \ldots, x_{2N})}{(x_{j+1}-x_j)^{2/\kappa}}. 
\end{equation}
Furthermore, the limiting function $\PartF_{\alpha/\amalg_j}$ satisfies the following system of $(2N-1)$ PDEs \footnote{Note that, the operators $\LL_{-2}^{(j)}$ and $\LL_{-3}^{(j)}$ in Proposition~\ref{prop::purepartitionFusion} are distinct from the ones in Proposition~\ref{prop::mGFFpdecov}. In fact, this kind of operators also depends on the index valences of planar link patterns. To simplify notations, we omit the dependence. } and the conformal covariance with $\kappa=4$. 
\begin{itemize}
\item Partial differential equations of second order (PDE): for $n\in\{1, \ldots, 2N\}\setminus\{j, j+1\}$, we have
\begin{align}\label{eqn::pdefusion2nd}
&\left[\frac{\partial^2}{\partial x_n^2}-\frac{4}{\kappa}\LL_{-2}^{(n)}\right]\PartF_{\alpha/\amalg_j}(x_1, \ldots, x_{j}, x_{j+2}, \ldots, x_{2N})=0, \\
&\text{where } \LL_{-2}^{(n)}=\sum_{\substack{1\le i\le 2N, \\ i\neq j, j+1, n}}\left(\frac{h}{(x_i-x_n)^2}-\frac{1}{x_i-x_n}\frac{\partial}{\partial x_i}\right)+\left(\frac{H}{(x_j-x_n)^2}-\frac{1}{x_j-x_n}\frac{\partial}{\partial x_j}\right). \notag
\end{align}
\item Partial differential equation of third order (PDE): 
\begin{align}\label{eqn::pdefusion3rd}
&\left[\frac{\partial^3}{\partial x_j^3}-\frac{16}{\kappa}\LL_{-2}^{(j)}\frac{\partial}{\partial x_j}+\frac{8(8-\kappa)}{\kappa^2}\LL_{-3}^{(j)}\right]\PartF_{\alpha/\amalg_j}(x_1, \ldots, x_j, x_{j+2}, \ldots, x_{2N})=0, \\
&\text{where }\LL_{-2}^{(j)}=\sum_{\substack{1\le i\le 2N, \\ i\neq j, j+1}}\left(\frac{h}{(x_i-x_j)^2}-\frac{1}{(x_i-x_j)}\frac{\partial}{\partial x_i}\right), \quad 
\LL_{-3}^{(j)}=\sum_{\substack{1\le i\le 2N, \\ i\neq j, j+1}}\left(\frac{2h}{(x_i-x_j)^3}-\frac{1}{(x_i-x_j)^2}\frac{\partial}{\partial x_i}\right). \notag
\end{align}
\item M\"{o}bius covariance (COV): For all M\"{o}bius maps $\varphi$ of $\HH$ such that $\varphi(x_1)<\cdots<\varphi(x_{2N})$, 
\begin{equation}\label{eqn::partitionFusionCOV}
\PartF_{\alpha/\amalg_j}(x_1, \ldots, x_{j}, x_{j+2}, \ldots, x_{2N})=\prod_{i}\varphi'(x_i)^{\Delta_i}\times \PartF_{\alpha/\amalg_j}(\varphi(x_1), \ldots, \varphi(x_j), \varphi(x_{j+2}), \ldots, \varphi(x_{2N})),
\end{equation}
where $\Delta_i=h$ for $i\in\{1, \ldots, j-1, j+2, \ldots, 2N\}$ and $\Delta_j=H$. 
\end{itemize}
\end{proposition}

The PDEs and the conformal covariance in Proposition~\ref{prop::purepartitionFusion} are not surprising: they come as specific fusion channel of correlation functions in terms of conformal field theory. In fact, Peltola proves in~\cite{PeltolaBasisPDE} a more general conclusion for $\kappa\in (0,8)\setminus\QQ$. Our results indicate that a similar conclusion as in~\cite{PeltolaBasisPDE} also holds for $\kappa=4$. Our method is straight forward but is specific for $\kappa=4$, as our proof uses the explicit formulae for $\SLE_4$ partition functions constructed in~\cite{PeltolaWuGlobalMultipleSLEs}. 

\subsection{Conformal block functions}\label{subsec::conformalblocks}
For $\alpha=\{ \{a_1, b_1\}, \ldots, \{a_N, b_N\} \} \in \DP_N$ ordered as in~\eqref{eqn::LPtoDP_order}, we define
conformal block function $\LU_\alpha : \chamber_{2N} \to \R_{>0}$
as follows:
\begin{align} \label{eqn::conformalblock_def}
\LU_\alpha(x_1, \ldots, x_{2N}) 
:= & \prod_{1\le i<j\le 2N} (x_j - x_i)^{\frac{1}{2} \vartheta_{\alpha}(i,j)}, \\
\text{where } \quad \vartheta_{\alpha}(i,j) := & 
\begin{cases}
+1 , & \text{if } i,j \in \{a_1, a_2, \ldots, a_N\} ,
\text{ or } i,j \in \{b_1, b_2, \ldots, b_N\} ,\\
-1 , & \text{otherwise.}
\end{cases}
\nonumber
\end{align}
The function $\LU_{\alpha}$ satisfies the second order PDEs~\eqref{eqn::multipleSLEsPDE}, see~\cite[Lemma~6.4]{PeltolaWuGlobalMultipleSLEs}. We will show that the function $\LU_{\alpha}^4$ satisfies the third order PDEs~\eqref{eqn::mGFFpde3rd}. The purpose will be clear in Section~\ref{subsec::mgffasy}. Note that, we will use the following basic facts about $\vartheta_{\alpha}$ through calculation without notice: for distinct $i, s, t\in\{1, 2, \ldots, 2N\}$, we have
\[\vartheta_{\alpha}(t,s)^2=1,\quad \vartheta_{\alpha}(t,i)\vartheta_{\alpha}(s,i)=\vartheta_{\alpha}(t,s). \]
\begin{lemma}\label{lem::conformalblock3rdpde}
For any $\alpha\in\DP_N$, the function $\LU_{\alpha}^4$ satisfies the third order PDEs~\eqref{eqn::mGFFpde3rd} with $\kappa=4$. 
\end{lemma}

\begin{proof}
It suffices to check PDE~\eqref{eqn::mGFFpde3rd} with $j=1$. 
Note that 
\[\LL_{-2}^{(1)}=\sum_{i\neq 1}\left(\frac{1}{(x_i-x_1)^2}-\frac{1}{x_i-x_1}\frac{\partial}{\partial x_i}\right),\quad \LL_{-3}^{(1)}=\sum_{i\neq 1}\left(\frac{2}{(x_i-x_1)^3}-\frac{1}{(x_i-x_1)^2}\frac{\partial}{\partial x_i}\right). \]
We write $\boldsymbol{x}=(x_1, \ldots, x_{2N})\in\chamber_{2N}$ and set
%We have
%\[\frac{\frac{\partial}{\partial x_i} \LU_{\alpha}^4(\boldsymbol{x})}{\LU_{\alpha}^4(\boldsymbol{x})}=\sum_{s\neq i}\frac{2\vartheta_{\alpha}(i,s)}{x_i-x_s}.\]
\[\tau(\boldsymbol{x})=\sum_{2\le i\le 2N}\frac{\vartheta_{\alpha}(i,1)}{x_i-x_1}.\]
Then we have
\begin{align*}
\frac{\LL_{-3}^{(1)}\LU_{\alpha}^4(\boldsymbol{x})}{\LU_{\alpha}^4(\boldsymbol{x})}=&\sum_{2\le i\le 2N}\frac{2-2\vartheta_{\alpha}(i,1)}{(x_i-x_1)^3}
+\sum_{2\le t<s\le 2N}\frac{2\vartheta_{\alpha}(s,t)(x_t-x_1+x_s-x_1)}{(x_t-x_1)^2(x_s-x_1)^2}. \\
\frac{\LL_{-2}^{(1)}\frac{\partial}{\partial x_1}\LU_{\alpha}^4(\boldsymbol{x})}{\LU_{\alpha}^4(\boldsymbol{x})}=&\sum_{2\le i\le 2N}\frac{-2\vartheta_{\alpha}(i,1)}{(x_i-x_1)^3}+4\tau(\boldsymbol{x})\sum_{2\le i\le 2N}\frac{\vartheta_{\alpha}(i,1)}{(x_i-x_1)^2}-2\tau(\boldsymbol{x})^3.\\
\frac{\frac{\partial^3}{\partial x_1^3}\LU_{\alpha}^4(\boldsymbol{x})}{\LU_{\alpha}^4(\boldsymbol{x})}=&\sum_{2\le i\le 2N}\frac{12-12\vartheta_{\alpha}(i,1)}{(x_i-x_1)^3}
+\sum_{2\le t<s<n\le 2N}\frac{-48\vartheta_{\alpha}(t,1)\vartheta_{\alpha}(s,1)\vartheta_{\alpha}(n,1)}{(x_t-x_1)(x_s-x_1)(x_n-x_1)}\\
&+\sum_{2\le t<s\le 2N}\frac{-12(2-\vartheta_{\alpha}(s,1))\vartheta_{\alpha}(t,1)(x_t-x_1)-12(2-\vartheta_{\alpha}(t,1))\vartheta_{\alpha}(s,1)(x_s-x_1)}{(x_t-x_1)^2(x_s-x_1)^2}.
\end{align*}
Therefore,
\begin{align*}
&\frac{\frac{\partial^3}{\partial x_1^3}\LU_{\alpha}^4(\boldsymbol{x})}{\LU_{\alpha}^4(\boldsymbol{x})}-4\frac{\LL_{-2}^{(1)}\frac{\partial}{\partial x_1}\LU_{\alpha}^4(\boldsymbol{x})}{\LU_{\alpha}^4(\boldsymbol{x})}+2\frac{\LL_{-3}^{(1)}\LU_{\alpha}^4(\boldsymbol{x})}{\LU_{\alpha}^4(\boldsymbol{x})}\\
=&\sum_{2\le i\le 2N}\frac{16-8\vartheta_{\alpha}(i,1)}{(x_i-x_1)^3}+\sum_{2\le t<s<n\le 2N}\frac{-48\vartheta_{\alpha}(t,1)\vartheta_{\alpha}(s,1)\vartheta_{\alpha}(n,1)}{(x_t-x_1)(x_s-x_1)(x_n-x_1)}\\
&+\sum_{2\le t<s\le 2N}\frac{-24\vartheta_{\alpha}(t,1)(x_t-x_1)-24\vartheta_{\alpha}(s,1)(x_s-x_1)+16\vartheta_{\alpha}(t,1)\vartheta_{\alpha}(s,1)(x_t-x_1+x_s-x_1)}{(x_t-x_1)^2(x_s-x_1)^2}\\
&-16\tau(\boldsymbol{x})\sum_{2\le i\le 2N}\frac{\vartheta_{\alpha}(i,1)}{(x_i-x_1)^2}+8\tau(\boldsymbol{x})^3\\
=&\sum_{2\le i\le 2N}\frac{16}{(x_i-x_1)^3}+\sum_{2\le t<s\le 2N}\frac{16\vartheta_{\alpha}(t,1)\vartheta_{\alpha}(s,1)(x_t-x_1+x_s-x_1)}{(x_t-x_1)^2(x_s-x_1)^2}-16\tau(\boldsymbol{x})\sum_{2\le i\le 2N}\frac{\vartheta_{\alpha}(i,1)}{(x_i-x_1)^2}=0.
\end{align*}
This completes the proof.
\end{proof}

\medbreak
Next, we give the relation between the two collections $\{\PartF_{\alpha}: \alpha\in\Pair_N\}$ and $\{\LU_{\alpha}: \alpha\in\DP_N\}$. 
Both of them form a basis for the solution space $\LS_N$, and they are related by a linear transformation. 
To give the transformation, we introduce a binary relation $\KWle$. Let $\alpha = \{ \{a_1, b_1\}, \ldots, \{a_N, b_N\}  \} \in \Pair_N$ be ordered as 
in~\eqref{eqn::LPtoDP_order}. Let $\beta \in \Pair_N$.
Then, $\alpha \KWle \beta$ if and only if there exists a permutation 
$\sigma \in \mathfrak{S}_N$ such that 
\[ \beta = \{ \{a_1, b_{\sigma(1)}\}, \ldots, \{a_N, b_{\sigma(N)}\} \} . \]
Note that the right-hand side in the above expression may not be ordered as in~\eqref{eqn::LPtoDP_order}. 
We denote by $\LM=(\LM_{\alpha, \beta})$ the $C_N\times C_N$ incidence matrix  of this relation: 
\begin{equation}\label{eqn::KWleincidencematrix}
\LM_{\alpha, \beta}=\one\{\alpha \KWle \beta\}.
\end{equation}
We collect some properties of $\LM$ in the following lemma. Recall from Section~\ref{subsec::pairpartitionDyckpath} that each planar pair partition $\alpha\in\Pair_N$ is associated with a Dyck path which we also denote by $\alpha\in\DP_N$. 
\begin{lemma} \label{lem::KWleinverse}
The matrix $\LM$ is invertible and we denote its inverse by $\LM^{-1}=(\LM^{-1}_{\alpha,\beta})$. 
The entry $\LM^{-1}_{\alpha, \beta}$ is non-zero if and only if $\alpha\preceq\beta$ as in~\eqref{eqn::Dychpathspartialorder}.
Furthermore, we have the following properties of $\LM^{-1}$. Suppose $\alpha, \beta\in\DP_N$. 
\begin{itemize}
\item  Suppose $\wedge^j\not\in\alpha$ and $\vee_j\in\beta$. Then $\alpha\preceq\beta$ if and only if $\alpha\preceq\beta\uparrow\lozenge_j$.  
\item Suppose $\wedge^j\not\in\alpha$, $\vee_j\in\beta$ and $\alpha\preceq\beta$. Then $\LM^{-1}_{\alpha,\beta}=-\LM^{-1}_{\alpha, \beta\uparrow\lozenge_j}$. 
\end{itemize}
\end{lemma}
\begin{proof}
See \cite[Proposition~2.9 and Lemma~2.10]{PeltolaWuGlobalMultipleSLEs}. 
\end{proof}

Now, we are ready to state the linear transformation between the two collections $\{\PartF_{\alpha}: \alpha\in\Pair_N\}$ and $\{\LU_{\alpha}: \alpha\in\DP_N\}$: (see~\cite[Theorem~1.5]{PeltolaWuGlobalMultipleSLEs})
\begin{equation}\label{eqn::purepartitionvsconformalblock}
\LU_{\alpha}(x_1, \ldots, x_{2N})=\sum_{\beta\in\Pair_N}\LM_{\alpha, \beta}\PartF_{\beta}(x_1, \ldots, x_{2N}), \quad
 \PartF_{\alpha}(x_1, \ldots, x_{2N})=\sum_{\beta\in\DP_N}\LM_{\alpha, \beta}^{-1}\LU_{\beta}(x_1, \ldots, x_{2N}).
\end{equation}

\subsection{Asymptotics of partition functions}
In this section, we will analyze the asymptotics of pure partition functions and conformal block functions as $x_j, x_{j+1}\to \xi$. 

\begin{lemma} \label{lem::conformalblock_ASY}
The collection $\{\LU_{\alpha}: \alpha\in \DP\}$ of conformal block functions satisfy the following asymptotic property: 
for any $j \in \{1, \ldots, 2N-1 \}$ and 
$x_1 < x_2 < \cdots < x_{j-1} < \xi < x_{j+2} < \cdots < x_{2N}$,
\begin{align}
\label{eqn::conformalblockASYrefined}
\lim_{\substack{\tilde{x}_j , \tilde{x}_{j+1} \to \xi, \\ \tilde{x}_i\to x_i \text{ for } i \neq j, j+1}} 
\frac{\LU_\alpha(\tilde{x}_1 , \ldots , \tilde{x}_{2N})}{(\tilde{x}_{j+1} - \tilde{x}_j)^{-1/2}} 
&=\begin{cases}
\LU_{ \alpha/\wedge^j} (x_1 , \ldots, x_{j-1}, x_{j+2}, \ldots, x_{2N}) , & \text{if } \wedge^j \in \alpha , \\
\LU_{ \alpha/\vee_j} (x_1 , \ldots, x_{j-1}, x_{j+2}, \ldots, x_{2N}) , & \text{if } \vee_j \in \alpha ,
\end{cases}\\
\lim_{\substack{\tilde{x}_j , \tilde{x}_{j+1} \to \xi, \\ \tilde{x}_i\to x_i \text{ for } i \neq j, j+1}} 
\frac{\LU_\alpha(\tilde{x}_1 , \ldots , \tilde{x}_{2N})}{(\tilde{x}_{j+1} - \tilde{x}_j)^{1/2}} 
&= \LU_{\alpha/\times_j}(x_1, \ldots, x_{j-1}, \xi, x_{j+2}, \ldots, x_{2N}), \quad \text{if }\times_j\in\alpha, 
\label{eqn::conformalblockFusionrefined}
\end{align}
where
\begin{equation}\label{eqn::conformalblockFusiondef}
\LU_{\alpha/\times_j}(x_1, \ldots, x_{j-1}, \xi, x_{j+2}, \ldots, x_{2N}):=\prod_{\substack{1\le t<s\le 2N\\ t,s\neq j,j+1}}(x_{s}-x_{t})^{\frac{1}{2}\vartheta_{\alpha}(t,s)}\prod_{\substack{1\le i\le 2N\\ i\neq j,j+1}}|x_{i}-\xi|^{\vartheta_{\alpha}(i,j)}.
\end{equation}
\end{lemma}
\begin{proof}
The asymptotics in~\eqref{eqn::conformalblockASYrefined} is proved in~\cite[Lemma~6.6]{PeltolaWuGlobalMultipleSLEs}. It remains to show~\eqref{eqn::conformalblockFusionrefined}. By definition, 
\[\frac{\LU_\alpha(\tilde{x}_1 , \ldots , \tilde{x}_{2N})}{(\tilde{x}_{j+1} - \tilde{x}_j)^{1/2}}=\prod_{\substack{1\le t<s\le 2N\\ t,s\neq j,j+1}}(\tilde x_{s}-\tilde x_{t})^{\frac{1}{2}\vartheta_{\alpha}(t,s)}\prod_{\substack{1\le i\le 2N\\ i\neq j,j+1}}|\tilde x_{i}-\tilde x_{j}|^{\frac{1}{2}\vartheta_{\alpha}(i,j)}\prod_{\substack{1\le i\le 2N\\ i\neq j,j+1}}|\tilde x_{i}-\tilde x_{j+1}|^{\frac{1}{2}\vartheta_{\alpha}(i,j+1)}.\]
Since $\times_{j}\in\alpha$, we have $\vartheta_{\alpha}(i,j)=\vartheta_{\alpha}(i,j+1)$. By taking limit, we obtain~\eqref{eqn::conformalblockFusionrefined}.
\end{proof}

\begin{lemma} \label{lem::purepartition_ASY}
The collection $\{\PartF_{\alpha}: \alpha\in \Pair\}$ of pure partition functions satisfy the following asymptotic property: 
for any $j \in \{1, \ldots, 2N-1 \}$ and 
$x_1 < x_2 < \cdots < x_{j-1} < \xi < x_{j+2} < \cdots < x_{2N}$,
\begin{align}
\label{eqn::purepartitionASYrefined}
&\lim_{\substack{\tilde{x}_j , \tilde{x}_{j+1} \to \xi, \\ \tilde{x}_i\to x_i \text{ for } i \neq j, j+1}} 
\frac{\PartF_\alpha(\tilde{x}_1 , \ldots , \tilde{x}_{2N})}{(\tilde{x}_{j+1} - \tilde{x}_j)^{-1/2}} 
=\PartF_{\alpha/\wedge_j}(x_{1},\ldots,x_{j-1},x_{j+2},\ldots,x_{2N}) 
  \quad  &&\text{if } \{j, j+1\} \in \alpha;\\
\label{eqn::purepartitionFusionrefined}
&\lim_{\substack{\tilde{x}_j , \tilde{x}_{j+1} \to \xi, \\ \tilde{x}_i\to x_i \text{ for } i \neq j, j+1}} 
\frac{\PartF_\alpha(\tilde{x}_1 , \ldots , \tilde{x}_{2N})}{(\tilde{x}_{j+1} - \tilde{x}_j)^{1/2}}=\PartF_{\alpha/\amalg_j}(x_1, \ldots, x_{j-1}, \xi, x_{j+2}, \ldots, x_{2N}) \quad  &&\text{if } \{j, j+1\} \not\in \alpha, 
\end{align}
where
\begin{align}\label{eqn::purepartitionFusiondef}
&\PartF_{\alpha/\amalg_j}:=\sum_{\vee_{j}\in\beta}\LM_{\alpha, \beta}^{-1} \LV_{\beta/\vee_j}
+\sum_{\times_{j}\in\beta}\LM_{\alpha, \beta}^{-1} \LU_{\beta/\times_j},\\
&\LV_{\beta/\vee_j}(x_1, \ldots, x_{j-1}, \xi, x_{j+2}, \ldots, x_{2N}):=\prod_{\substack{1\le t<s\le 2N\\ t,s\neq j,j+1}}(x_{s}-x_{t})^{\frac{1}{2}\vartheta_{\beta}(t,s)}\sum_{\substack{1\le i\le 2N\\ i\neq j,j+1}}\frac{\vartheta_{\beta}(i,j)}{x_{i}-\xi}. \notag
\end{align}
\end{lemma}
\begin{proof}
The asymptotics in~\eqref{eqn::purepartitionASYrefined} is proved in~\cite[Lemma~6.7]{PeltolaWuGlobalMultipleSLEs}. It remains to show~\eqref{eqn::purepartitionFusionrefined}. In the following, we assume $\{j,j+1\}\not\in\alpha$. 
From Lemma~\ref{lem::KWleinverse} and~\eqref{eqn::purepartitionvsconformalblock}, we have
\begin{align}\label{eqn::purepartitionFusionrefinedaux0}
 \PartF_{\alpha}(\tilde x_1, \ldots, \tilde x_{2N}) %&=\sum_{\beta\in\DP_N}\LM_{\alpha, \beta}^{-1}\LU_{\beta}(\tilde x_1, \ldots, \tilde x_{2N})\notag\\ 
 &=\sum_{\substack{\alpha\preceq\beta\\ \vee_{j}\in\beta}}\LM_{\alpha, \beta}^{-1}\LU_{\beta}(\tilde x_1, \ldots, \tilde x_{2N})+\sum_{\substack{\alpha\preceq\beta\\ \wedge_{j}\in\beta}}\LM_{\alpha, \beta}^{-1}\LU_{\beta}(\tilde x_1, \ldots, \tilde x_{2N})+\sum_{\substack{\alpha\preceq\beta\\ \times_{j}\in\beta}}\LM_{\alpha, \beta}^{-1}\LU_{\beta}(\tilde x_1, \ldots, \tilde x_{2N}) 
% &=\sum_{\substack{\alpha\preceq\beta\\ \vee_{j}\in\beta}}\LM_{\alpha, \beta}^{-1}(\LU_{\beta}-\LU_{\beta\uparrow\lozenge})(\tilde x_1, \ldots, \tilde x_{2N})+\sum_{\substack{\alpha\preceq\beta\\ \times_{j}\in\beta}}\LM_{\alpha, \beta}^{-1}\LU_{\beta}(\tilde x_1, \ldots, \tilde x_{2N}).
\end{align}
From Lemma~\ref{lem::KWleinverse}, for every $\beta\in\DP_{N}$ with $\vee_{j}\in\beta$ , we have $\alpha\preceq\beta$ if and only if $\alpha\preceq\beta\uparrow\lozenge_{j}$. In such case, we have further that $\LM_{\alpha, \beta}^{-1}=-\LM_{\alpha, \beta\uparrow\lozenge_{j}}^{-1}$. For the first two sums in the right hand side of~\eqref{eqn::purepartitionFusionrefinedaux0}, we have
\begin{align*}
\sum_{\substack{\alpha\preceq\beta\\ \vee_{j}\in\beta}}\LM_{\alpha, \beta}^{-1}\LU_{\beta}(\tilde x_1, \ldots, \tilde x_{2N})+\sum_{\substack{\alpha\preceq\beta\\ \wedge_{j}\in\beta}}\LM_{\alpha, \beta}^{-1}\LU_{\beta}(\tilde x_1, \ldots, \tilde x_{2N})=\sum_{\substack{\alpha\preceq\beta\\ \vee_{j}\in\beta}}\LM_{\alpha, \beta}^{-1}(\LU_{\beta}-\LU_{\beta\uparrow\lozenge_j})(\tilde x_1, \ldots, \tilde x_{2N})
\end{align*}
Fix $\beta$ such that $\alpha\preceq\beta$ and $\vee_j\in\beta$, we have
\begin{align*}
&(\LU_{\beta}-\LU_{\beta\uparrow\lozenge_j})(\tilde x_1, \ldots, \tilde x_{2N})\\
&=(\tilde x_{j+1}-\tilde x_{j})^{-\frac{1}{2}}\left(\prod_{i\neq j,j+1}\left(\frac{\tilde x_{i}-\tilde x_{j}}{\tilde x_{i}-\tilde x_{j+1}}\right)^{\frac{1}{2}\vartheta_{\beta}(i,j)}-\prod_{i\neq j,j+1}\left(\frac{\tilde x_{i}-\tilde x_{j+1}}{\tilde x_{i}-\tilde x_{j}}\right)^{\frac{1}{2}\vartheta_{\beta}(i,j)}\right)\prod_{\substack{1\le t<s\le 2N\\ t,s\neq j,j+1}}(\tilde x_{s}-\tilde x_{t})^{\frac{1}{2}\vartheta_{\beta}(t,s)}.
\end{align*}
Dividing by $(\tilde{x}_{j+1}-\tilde{x}_j)^{1/2}$, we have
\begin{equation}\label{eqn::purepartitionFusionrefinedaux1}
\lim_{\substack{\tilde{x}_j , \tilde{x}_{j+1} \to \xi, \\ \tilde{x}_i\to x_i \text{ for } i \neq j, j+1}} \frac{(\LU_{\beta}-\LU_{\beta\uparrow\lozenge_j})(\tilde x_1, \ldots, \tilde x_{2N})}{(\tilde x_{j+1}-\tilde x_{j})^{\frac{1}{2}}}=\prod_{\substack{1\le t<s\le 2N\\ t,s\neq j,j+1}}(x_{s}-x_{t})^{\frac{1}{2}\vartheta_{\beta}(t,s)}\sum_{\substack{1\le i\le 2N\\ i\neq j,j+1}}\frac{\vartheta_{\beta}(i,j)}{x_{i}-\xi}.
\end{equation}

For the third sum in the right hand side of~\eqref{eqn::purepartitionFusionrefinedaux0}, by~\eqref{eqn::conformalblockFusionrefined}, we have
\begin{equation}\label{eqn::purepartitionFusionrefinedaux2}
\lim_{\substack{\tilde{x}_j , \tilde{x}_{j+1} \to \xi, \\ \tilde{x}_i\to x_i \text{ for } i \neq j, j+1}} 
\frac{\LU_\beta(\tilde{x}_1 , \ldots , \tilde{x}_{2N})}{(\tilde{x}_{j+1} - \tilde{x}_j)^{1/2}} 
= \LU_{\beta/\times_j}(x_1, \ldots, x_{j-1}, \xi, x_{j+2}, \ldots, x_{2N}).
\end{equation}
Plugging~\eqref{eqn::purepartitionFusionrefinedaux1} and~\eqref{eqn::purepartitionFusionrefinedaux2} into~\eqref{eqn::purepartitionFusionrefinedaux0}, we obtain~\eqref{eqn::purepartitionFusionrefined}. 
\end{proof}

Note that, we use the notation $\alpha/\amalg_j$ in~\eqref{eqn::purepartitionFusiondef}. It can be understood as a general link pattern. For a planar pair partition $\alpha\in\Pair_N$, suppose $\times_j\in\alpha$ or $\vee_j\in\alpha$, we define $\alpha/\amalg_j$ to be the $N$-link pattern of $(2N-1)$ points with index valences $\varsigma=(1, \ldots, 1, 2, 1, \ldots, 1)$ obtained from $\alpha$ by merging the points $j$ and $j+1$ and relabelling the remaining $(2N-1)$ indices so that they are the first $(2N-1)$ integers.   

\subsection{Fusion of partition functions}

In this section, we will show that the functions defined in~\eqref{eqn::conformalblockFusiondef} and~\eqref{eqn::purepartitionFusiondef} satisfy the system of $(2N-1)$ PDEs in~\eqref{eqn::pdefusion2nd} and~\eqref{eqn::pdefusion3rd}, and complete the proof of Proposition~\ref{prop::purepartitionFusion}.

\begin{lemma}\label{lem::conformalblockFusion2nd}
The function $\LU_{\alpha/\times_j}$ defined in~\eqref{eqn::conformalblockFusiondef} satisfies the second order PDE~\eqref{eqn::pdefusion2nd} with $\kappa=4$ for $n\in\{1, \ldots, 2N\}\setminus\{j, j+1\}$. 
\end{lemma}
\begin{proof}
Without loss of generality, we assume $j=1$. Note that $h=1/4$ and $H=1$ when $\kappa=4$. 
The second order PDE~\eqref{eqn::pdefusion2nd} becomes the following: for $n\in\{3, 4, \ldots, 2N\}$, 
\begin{align}\label{eqn::pdefusion2ndbis}
&\left[\frac{\partial^2}{\partial x_n^2}-\LL_{-2}^{(n)}\right]F(x_1, x_3, \ldots, x_{2N})=0, \\
&\text{where } \LL_{-2}^{(n)}=\sum_{\substack{3\le i\le 2N, \\ i\neq n}}\left(\frac{\frac{1}{4}}{(x_i-x_n)^2}-\frac{1}{x_i-x_n}\frac{\partial}{\partial x_i}\right)+\left(\frac{1}{(x_1-x_n)^2}-\frac{1}{x_1-x_n}\frac{\partial}{\partial x_1}\right). \notag
\end{align}
The function in~\eqref{eqn::conformalblockFusiondef} with $j=1$ becomes
\begin{equation}\label{eqn::conformalblockFusiondefbis}
\LU_{\alpha/\times_1}(x_1, x_3, \ldots, x_{2N}):=\prod_{3\le t<s\le 2N}(x_{s}-x_{t})^{\frac{1}{2}\vartheta_{\alpha}(t,s)}\prod_{3\le i\le 2N}(x_{i}-x_1)^{\vartheta_{\alpha}(i,1)}.
\end{equation}
It suffices to show that the function in~\eqref{eqn::conformalblockFusiondefbis} solves the second order PDE~\eqref{eqn::pdefusion2ndbis}. 

We write $\boldsymbol{x}=(x_1, x_3, \ldots, x_{2N})$. 
We have, for $i\in\{3, 4, \ldots, 2N\}$, 
\begin{align*}
\frac{\frac{\partial}{\partial x_i}\LU_{\alpha/\times_1}(\boldsymbol{x})}{\LU_{\alpha/\times_1}(\boldsymbol{x})}=\sum_{\substack{3\le s\le 2N, \\ s\neq i}}\frac{\frac{1}{2}\vartheta_{\alpha}(s,i)}{x_i-x_s}+\frac{\vartheta_{\alpha}(i, 1)}{x_i-x_1}; \qquad \frac{\frac{\partial}{\partial x_1}\LU_{\alpha/\times_1}(\boldsymbol{x})}{\LU_{\alpha/\times_1}(\boldsymbol{x})}=\sum_{3\le s\le 2N}\frac{-\vartheta_{\alpha}(s, 1)}{x_s-x_1}. 
\end{align*}
Then, we have
\begin{align*}
\frac{\LL_{-2}^{(n)}\LU_{\alpha/\times_1}(\boldsymbol{x})}{\LU_{\alpha/\times_1}(\boldsymbol{x})}=
&\sum_{\substack{3\le i\le 2N, \\ i\neq n}}\frac{\frac{1}{4}-\frac{1}{2}\vartheta_{\alpha}(n,i)}{(x_n-x_i)^2}+\frac{1-\vartheta_{\alpha}(n,1)}{(x_n-x_1)^2}\\
&+\sum_{\substack{3\le t<s\le 2N,\\ t, s\neq n}}\frac{\frac{1}{2}\vartheta_{\alpha}(t,s)}{(x_n-x_t)(x_n-x_s)}+\sum_{\substack{3\le i\le 2N, \\ i\neq n}}\frac{\vartheta_{\alpha}(i, 1)}{(x_n-x_i)(x_n-x_1)}. \\
\frac{\frac{\partial^2}{\partial x_n^2}\LU_{\alpha/\times_1}(\boldsymbol{x})}{\LU_{\alpha/\times_1}(\boldsymbol{x})}=
&\sum_{\substack{3\le i\le 2N, \\ i\neq n}}\frac{\frac{1}{4}-\frac{1}{2}\vartheta_{\alpha}(n,i)}{(x_n-x_i)^2}
+\frac{1-\vartheta_{\alpha}(n,1)}{(x_n-x_1)^2}\\
&+\sum_{\substack{3\le t<s\le 2N,\\ t, s\neq n}}\frac{\frac{1}{2}\vartheta_{\alpha}(s,t)}{(x_n-x_t)(x_n-x_s)}+\sum_{\substack{3\le i\le 2N, \\ i\neq n}}\frac{\vartheta_{\alpha}(i,1)}{(x_n-x_i)(x_n-x_1)}. 
\end{align*}
%Comparing the two expressions in the right hand side, we notice that 
%\begin{align*}
%\frac{1}{2}\vartheta_{\alpha}(n,i)(\frac{1}{2}\vartheta_{\alpha}(n,i)-1)&=\frac{1}{4}-\frac{1}{2}\vartheta_{\alpha}(n,i); \\
%\vartheta_{\alpha}(n,1)(\vartheta_{\alpha}(n,1)-1)&=1-\vartheta_{\alpha}(n,1);\\
%\vartheta_{\alpha}(n,t)\vartheta_{\alpha}(n,s)&=\vartheta_{\alpha}(t,s). 
%\end{align*}
These give $\left[\frac{\partial^2}{\partial x_n^2}-\LL_{-2}^{(n)}\right]\LU_{\alpha/\times_1}=0$ as desired. 
\end{proof}

\begin{lemma}\label{lem::conformalblockFusion3rd}
The function $\LU_{\alpha/\times_j}$ defined in~\eqref{eqn::conformalblockFusiondef} satisfies the third order PDE~\eqref{eqn::pdefusion3rd} with $\kappa=4$. 
\end{lemma}
\begin{proof}
Without loss of generality, we assume $j=1$. The third order PDE~\eqref{eqn::pdefusion3rd} becomes the following: 
\begin{align}\label{eqn::pdefusion3rdbis}
&\left[\frac{\partial^3}{\partial x_1^3}-4\LL_{-2}^{(1)}\frac{\partial}{\partial x_1}+2\LL_{-3}^{(1)}\right]F(x_1, x_3, \ldots, x_{2N})=0, \\
&\text{where }\LL_{-2}^{(1)}=\sum_{3\le i\le 2N}\left(\frac{\frac{1}{4}}{(x_i-x_1)^2}-\frac{1}{(x_i-x_1)}\frac{\partial}{\partial x_i}\right), \quad 
\LL_{-3}^{(1)}=\sum_{3\le i\le 2N}\left(\frac{\frac{1}{2}}{(x_i-x_1)^3}-\frac{1}{(x_i-x_1)^2}\frac{\partial}{\partial x_i}\right). \notag
\end{align}
It suffices to show that the function in~\eqref{eqn::conformalblockFusiondefbis} solves the third order PDE~\eqref{eqn::pdefusion3rdbis}. 

We write $\boldsymbol{x}=(x_1, x_3, \ldots, x_{2N})$. We have
\begin{align*}
\frac{2\LL_{-3}^{(1)}\LU_{\alpha/\times_1}(\boldsymbol{x})}{\LU_{\alpha/\times_1}(\boldsymbol{x})}=&\sum_{3\le i\le 2N}\frac{1-2\vartheta_{\alpha}(i,1)}{(x_i-x_1)^3}\\
&+\sum_{3\le t<s\le 2N}\frac{\vartheta_{\alpha}(s,t)(x_t-x_1+x_s-x_1)}{(x_t-x_1)^2(x_s-x_1)^2}.\\
\frac{4\LL_{-2}^{(1)}\frac{\partial}{\partial x_1}\LU_{\alpha/\times_1}(\boldsymbol{x})}{\LU_{\alpha/\times_1}(\boldsymbol{x})}
=&\sum_{3\le i\le 2N}\frac{4-5\vartheta_{\alpha}(i,1)}{(x_i-x_1)^3}\\
&+\sum_{3\le t<s\le 2N}\frac{4\vartheta_{\alpha}(s,t)(x_t-x_1+x_s-x_1)-3\vartheta_{\alpha}(t,1)(x_t-x_1)-3\vartheta_{\alpha}(s,1)(x_s-x_1)}{(x_t-x_1)^2(x_s-x_1)^2}\\
&+\sum_{3\le t<s<n\le 2N}\frac{(-6)\vartheta_{\alpha}(t,1)\vartheta_{\alpha}(s,1)\vartheta_{\alpha}(n,1)}{(x_t-x_1)(x_s-x_1)(x_n-x_1)}. \\
\frac{\frac{\partial^3}{\partial x_1^3}\LU_{\alpha/\times_1}(\boldsymbol{x})}{\LU_{\alpha/\times_1}(\boldsymbol{x})}
=&\sum_{3\le i\le 2N}\frac{3-3\vartheta_{\alpha}(i,1)}{(x_i-x_1)^3}\\
&+\sum_{3\le t<s\le 2N}\frac{3(\vartheta_{\alpha}(s,t)-\vartheta_{\alpha}(t,1))(x_t-x_1)+3(\vartheta_{\alpha}(s,t)-\vartheta_{\alpha}(s,1))(x_s-x_1)}{(x_t-x_1)^2(x_s-x_1)^2}
\\
&+\sum_{3\le t<s<n\le 2N}\frac{(-6)\vartheta_{\alpha}(t,1)\vartheta_{\alpha}(s,1)\vartheta_{\alpha}(n,1)}{(x_t-x_1)(x_s-x_1)(x_n-x_1)}.
\end{align*}
These give $\left[\frac{\partial^3}{\partial x_1^3}-4\LL_{-2}^{(1)}\frac{\partial}{\partial x_1}+2\LL_{-3}^{(1)}\right]\LU_{\alpha/\times_1}=0$ as desired.  
\end{proof}

\begin{lemma}\label{lem::purepartitionFusionpde2nd}
The function $\PartF_{\alpha/\amalg_j}$ defined in~\eqref{eqn::purepartitionFusiondef} satisfies the second order PDE~\eqref{eqn::pdefusion2nd} with $\kappa=4$ for $n\in\{1, \ldots, 2N\}\setminus\{j, j+1\}$. 
\end{lemma}
\begin{proof}
Without loss of generality, we assume $j=1$. The function in~\eqref{eqn::purepartitionFusiondef} with $j=1$ becomes $\PartF_{\alpha/\amalg_1}=\sum_{\beta: \vee_1\in\beta}\LM_{\alpha, \beta}^{-1}\LV_{\beta/\vee_1}+\sum_{\beta: \times_1\in\beta}\LM_{\alpha, \beta}^{-1}\LU_{\beta/\times_1}$
where
\begin{equation}\label{eqn::purepartitionFusiondefbis}
\LV_{\beta/\vee_1}(x_1, x_3, \ldots, x_{2N}):=\prod_{3\le t<s\le 2N}(x_s-x_t)^{\frac{1}{2}\vartheta_{\beta}(t,s)}\sum_{3\le i\le 2N}\frac{\vartheta_{\beta}(i,1)}{x_i-x_1}. 
\end{equation}
We will show that $\PartF_{\alpha/\amalg_1}$ satisfies the second order PDE~\eqref{eqn::pdefusion2ndbis}. 
From Lemma~\ref{lem::conformalblockFusion2nd}, the function $\LU_{\beta/\times_1}$ satisfies PDE~\eqref{eqn::pdefusion2ndbis}. It suffcies to show that the function $\LV_{\beta/\vee_1}$ in~\eqref{eqn::purepartitionFusiondefbis} satisfies PDE~\eqref{eqn::pdefusion2ndbis}. 

We write $\boldsymbol{x}=(x_1, x_3, \ldots, x_{2N})$ and set
\begin{equation}\label{eqn::functionsigmaaux}
\sigma(\boldsymbol{x})=\sum_{3\le i\le 2N}\frac{\vartheta_{\beta}(i,1)}{x_i-x_1}. 
\end{equation}
We have, for $i\in\{3, 4, \ldots, 2N\}$, 
\begin{align*}
\frac{\frac{\partial}{\partial x_i}\LV_{\beta/\vee_1}(\boldsymbol{x})}{\LV_{\beta/\vee_1}(\boldsymbol{x})}
=\sum_{\substack{3\le s\le 2N, \\ s\neq i}}\frac{\frac{1}{2}\vartheta_{\beta}(i, s)}{x_i-x_s}
+\frac{-\vartheta_{\beta}(i,1)}{\sigma(\boldsymbol{x})(x_i-x_1)^2},\qquad 
\frac{\frac{\partial}{\partial x_1}\LV_{\beta/\vee_1}(\boldsymbol{x})}{\LV_{\beta/\vee_1}(\boldsymbol{x})}
=\sum_{3\le s\le 2N}\frac{\vartheta_{\beta}(s,1)}{\sigma(\boldsymbol{x})(x_s-x_1)^2}. 
\end{align*}
Then we obtain
\begin{align*}
\frac{\LL_{-2}^{(n)}\LV_{\beta/\vee_1}(\boldsymbol{x})}{\LV_{\beta/\vee_1}(\boldsymbol{x})}
=&\sum_{\substack{3\le i\le 2N,\\ i\neq n}}\frac{\frac{1}{4}-\frac{1}{2}\vartheta_{\beta}(i,n)}{(x_n-x_i)^2}+\frac{1}{(x_n-x_1)^2}+\frac{\vartheta_{\beta}(n, 1)}{\sigma(\boldsymbol{x})(x_n-x_1)^3}\\
&+\sum_{\substack{3\le t<s\le 2N,\\ t, s\neq n}}\frac{\frac{1}{2}\vartheta_{\beta}(t,s)}{(x_n-x_s)(x_n-x_t)}+\sum_{\substack{3\le i\le 2N,\\ i\neq n}}\frac{-\vartheta_{\beta}(i,1)}{\sigma(\boldsymbol{x})(x_n-x_i)(x_i-x_1)(x_n-x_1)}. \\
\frac{\frac{\partial^2}{\partial x_n^2}\LV_{\beta/\vee_1}(\boldsymbol{x})}{\LV_{\beta/\vee_1}(\boldsymbol{x})}
=&\sum_{\substack{3\le i\le 2N,\\ i\neq n}}\frac{\frac{1}{4}-\frac{1}{2}\vartheta_{\beta}(i,n)}{(x_n-x_i)^2}+\frac{2\vartheta_{\beta}(n, 1)}{\sigma(\boldsymbol{x})(x_n-x_1)^3}\\
&+\sum_{\substack{3\le t<s\le 2N,\\ t, s\neq n}}\frac{\frac{1}{2}\vartheta_{\beta}(t,s)}{(x_n-x_s)(x_n-x_t)}+\sum_{\substack{3\le i\le 2N,\\ i\neq n}}\frac{-\vartheta_{\beta}(i,1)}{\sigma(\boldsymbol{x})(x_n-x_i)(x_n-x_1)^2}. 
\end{align*}
Taking the difference, we have
\begin{align*}
\frac{\frac{\partial^2}{\partial x_n^2}\LV_{\beta/\vee_1}(\boldsymbol{x})}{\LV_{\beta/\vee_1}(\boldsymbol{x})}-\frac{\LL_{-2}^{(n)}\LV_{\beta/\vee_1}(\boldsymbol{x})}{\LV_{\beta/\vee_1}(\boldsymbol{x})}
=&\frac{-1}{(x_n-x_1)^2}+\frac{\vartheta_{\beta}(n, 1)}{\sigma(\boldsymbol{x})(x_n-x_1)^3}+\sum_{\substack{3\le i\le 2N,\\ i\neq n}}\frac{-\vartheta_{\beta}(i,1)}{\sigma(\boldsymbol{x})(x_n-x_i)(x_n-x_1)^2}\\
&+\sum_{\substack{3\le i\le 2N,\\ i\neq n}}\frac{\vartheta_{\beta}(i,1)}{\sigma(\boldsymbol{x})(x_n-x_i)(x_i-x_1)(x_n-x_1)}. \\
=&\frac{-1}{(x_n-x_1)^2}+\frac{\vartheta_{\beta}(n, 1)}{\sigma(\boldsymbol{x})(x_n-x_1)^3}+\sum_{\substack{3\le i\le 2N,\\ i\neq n}}\frac{\vartheta_{\beta}(i,1)}{\sigma(\boldsymbol{x})(x_i-x_1)(x_n-x_1)^2}\\
=&\frac{-1}{(x_n-x_1)^2}+\frac{\vartheta_{\beta}(n, 1)}{\sigma(\boldsymbol{x})(x_n-x_1)^3}
+\frac{\sigma(\boldsymbol{x})-\frac{\vartheta_{\beta}(n, 1)}{x_n-x_1}}{\sigma(\boldsymbol{x})(x_n-x_1)^2}=0.
\end{align*}
This completes the proof. 
\end{proof}
\begin{lemma}\label{lem::purepartitionFusionpde3rd}
The function $\PartF_{\alpha/\amalg_j}$ defined in~\eqref{eqn::purepartitionFusiondef} satisfies the third order PDE~\eqref{eqn::pdefusion3rd} with $\kappa=4$. 
\end{lemma}
\begin{proof}
Without loss of generality, we assume $j=1$. From Lemma~\ref{lem::conformalblockFusion3rd}, the function $\LU_{\alpha/\times_1}$ satisfies PDE~\eqref{eqn::pdefusion3rdbis}. It suffices to show that the function $\LV_{\beta/\vee_1}$ in~\eqref{eqn::purepartitionFusiondefbis} satisfies the third order PDE~\eqref{eqn::pdefusion3rdbis}.
We write $\boldsymbol{x}=(x_1, x_3, \ldots, x_{2N})$ and set $\sigma(\boldsymbol{x})$ as in~\eqref{eqn::functionsigmaaux}. 
Then we have
\begin{align*}
\frac{2\LL_{-3}^{(1)}\LV_{\beta/\vee_1}(\boldsymbol{x})}{\LV_{\beta/\vee_1}(\boldsymbol{x})}
=&\sum_{3\le i\le 2N}\frac{1}{(x_i-x_1)^3}+\sum_{3\le i\le 2N}\frac{2\vartheta_{\beta}(i,1)}{\sigma(\boldsymbol{x})(x_i-x_1)^4}+\sum_{3\le t<s\le 2N}\frac{\vartheta_{\beta}(s,t)(x_t-x_1+x_s-x_1)}{(x_t-x_1)^2(x_s-x_1)^2}.\\
\frac{4\LL_{-2}^{(1)}\frac{\partial}{\partial x_1}\LV_{\beta/\vee_1}(\boldsymbol{x})}{\LV_{\beta/\vee_1}(\boldsymbol{x})}
=&\sum_{3\le i\le 2N}\frac{8\vartheta_{\beta}(i,1)}{\sigma(\boldsymbol{x})(x_i-x_1)^4}+\sigma(\boldsymbol{x})\frac{\partial \sigma(\boldsymbol{x})}{\partial x_1}. \\
\frac{\frac{\partial^3}{\partial x_1^3}\LV_{\beta/\vee_1}(\boldsymbol{x})}{\LV_{\beta/\vee_1}(\boldsymbol{x})}
=&\sum_{3\le i\le 2N}\frac{6\vartheta_{\beta}(i,1)}{\sigma(\boldsymbol{x})(x_i-x_1)^4}.
\end{align*}
Therefore,
\begin{align*}
&\frac{\frac{\partial^3}{\partial x_1^3}\LV_{\beta/\vee_1}(\boldsymbol{x})}{\LV_{\beta/\vee_1}(\boldsymbol{x})}-\frac{4\LL_{-2}^{(1)}\frac{\partial}{\partial x_1}\LV_{\beta/\vee_1}(\boldsymbol{x})}{\LV_{\beta/\vee_1}(\boldsymbol{x})}+\frac{2\LL_{-3}^{(1)}\LV_{\beta/\vee_1}(\boldsymbol{x})}{\LV_{\beta/\vee_1}(\boldsymbol{x})}\\
&=\sum_{3\le i\le 2N}\frac{1}{(x_i-x_1)^3}
+\sum_{3\le t<s\le 2N}\frac{\vartheta_{\beta}(s,t)(x_t-x_1+x_s-x_1)}{(x_t-x_1)^2(x_s-x_1)^2}-\sigma(\boldsymbol{x})\frac{\partial \sigma(\boldsymbol{x})}{\partial x_1}\\
&=\left(\sum_{3\le t\le 2N}\frac{\vartheta_{\beta}(t,1)}{x_t-x_1}\right)\left(\sum_{3\le s\le 2N}\frac{\vartheta_{\beta}(s,1)}{(x_s-x_1)^2}\right)-\sigma(\boldsymbol{x})\frac{\partial \sigma(\boldsymbol{x})}{\partial x_1}=0.
\end{align*}
This completes the proof. 
\end{proof}

\begin{proof}[Proof of Proposition~\ref{prop::purepartitionFusion}]
The existence of the limit~\eqref{eqn::purepartitionFusion} is a consequence of~\eqref{eqn::purepartitionFusionrefined}. 
The limiting function satisfies the PDE system due to Lemmas~\ref{lem::purepartitionFusionpde2nd} and~\ref{lem::purepartitionFusionpde3rd}. 
COV~\eqref{eqn::partitionFusionCOV} is a consequence of COV~\eqref{eqn::multipleSLEsCOV} and the existence of the limit~\eqref{eqn::purepartitionFusion}. 
\end{proof}

We end this section by a discussion on the solutions to the system of $(2N-1)$ PDEs in Proposition~\ref{prop::purepartitionFusion}. Consider the solution space of smooth functions $F:\chamber_{2N-1}\to \C$ satisfying PDEs~\eqref{eqn::pdefusion2nd} and~\eqref{eqn::pdefusion3rd}, COV~\eqref{eqn::partitionFusionCOV}, and a mild power bound. We believe that this solution space would have dimension $C_N-C_{N-1}$. Furthermore, we believe that the collection 
\[\{\LU_{\beta/\times_j}: \beta\in\DP_N\text{ with } \times_j\in\beta\}\cup\{\LV_{\beta/\vee_j}: \beta\in\DP_N\text{ with } \vee_j\in\beta\}\] 
gives a basis for this solution space, and that the collection 
\[\{\PartF_{\alpha/\amalg_j}: \alpha\in\Pair_N\text{ with }\wedge_j\not\in\alpha\}\] 
gives another basis for the solution space.

%% file: tex/gfflevellines.tex
In this section, we first introduce continuum GFF and level lines in Section~\ref{subsec::gfflevellines}. Then we state the main conclusion of the section---Theorem~\ref{thm::GFFconnectionproba}---in Section~\ref{subsec::gffconnectionproba}. This theorem gives the connection probabilities for level lines of GFF in polygons with boundary data given by Dyck paths. The proof of Theorem~\ref{thm::GFFconnectionproba} involves several technical lemmas which we prove in Section~\ref{subsec::technicallemmas}. 

\subsection{Continuum GFF and level lines}
\label{subsec::gfflevellines}
In this section, we introduce the Gaussian free field and its level lines. 
We refer to the literature~\cite{SheffieldGFFMath, SchrammSheffieldContinuumGFF, MillerSheffieldIG1, WangWuLevellinesGFFI} for details. 
Let $\Omega\subsetneq \C$ be a non-empty domain.
We denote by $H_s(\Omega)$ the space of real-valued smooth functions which are compactly supported in $\Omega$. 
We equip the space with Dirichlet inner product
\[ (f,g)_{\nabla} := \frac{1}{2\pi} \int_\Omega \nabla f(z) \cdot \nabla g(z) d^2 z . \]
We denote by $H(\Omega)$ the Hilbert space completion of $H_s(\Omega)$
with respect to the Dirichlet inner product. A (zero-boundary) Gaussian free field (GFF) $\Gamma$ is an $H(\Omega)$-indexed linear space of random variables, denoted by $(\Gamma, f)_{\nabla}$ for each $f\in H(\Omega)$, such that the map $f\mapsto (\Gamma, f)_{\nabla}$ is linear and each $(\Gamma, f)_{\nabla}$ is a centered Gaussian with variance $(f, f)_{\nabla}$. 
In general, for any harmonic function $u$ on $\Omega$, we define the GFF with boundary data $u$ by  
$\Gamma +u$ where $\Gamma$ is the zero-boundary GFF on $\Omega$. 

\medbreak
Next, we introduce $\SLE$ with force points.  
We set
\[\underline{y}^L=(y^{L, l}<\cdots<y^{L, 1}\le 0)\quad \text{and}\quad\underline{y}^R=(0\le y^{R, 1}<\cdots<y^{R, r}),\] and 
\[\underline{\rho}^L=(\rho^{L, l},\cdots, \rho^{L, 1})\quad \text{and} \quad\underline{\rho}^R=(\rho^{R, 1},\cdots, \rho^{R, r}),\] 
where $\rho^{q, i}\in\R$, for $q\in\{L, R\}$ and $i \in \Z_{>0}$.
An $\SLE_{\kappa}(\underline{\rho}^L;\underline{\rho}^R)$ process with force points 
$(\underline{y}^L;\underline{y}^R)$ is the Loewner evolution driven by $W_t$ that solves
the following system of integrated SDEs: 
\begin{align}
\label{eqn::SLEkapparho}
\begin{split}
W_t =&\sqrt{\kappa} B_t + \int_0^t\frac{\rho^{L, i} ds}{W_s-V_s^{L, i}} + \sum_{i=1}^r 
\int_0^t\frac{\rho^{R, i} ds}{W_s-V_s^{R, i}} , \\
V^{q, i}_t =& y^{q, i} +  \int_0^t\frac{2ds}{V^{q, i}_s-W_s} , 
\quad \text{for }q\in\{L, R\}  \text{ and } i \in \Z_{>0},
\end{split}
\end{align}
where $B_t$ is the one-dimensional Brownian motion. Note that the process $V_t^{q, i}$ is 
the evolution of the point $y^{q, i}$, and we may write $g_t(y^{q, i})$ for $V_t^{q, i}$. 
We define the continuation threshold
of the $\SLE_{\kappa}(\underline{\rho}^L;\underline{\rho}^R)$ 
to be the infimum of the time $t$ for which 
\[ \text{either} \quad \sum_{i : V^{L, i}_t=W_t} \rho^{L, i}\le -2 , 
\quad \text{or} \quad \sum_{i : V^{R, i}_t=W_t} \rho^{R, i}\le -2 . \]
By~\cite{MillerSheffieldIG1}, 
the $\SLE_{\kappa}(\underline{\rho}^L;\underline{\rho}^R)$
process is well-defined up to the continuation threshold, and it is almost surely
generated by a continuous curve up to and including the continuation threshold.

\medbreak

Now, we are ready to introduce level lines of GFF.  
Let $K= (K_t, t\ge 0)$ be an $\SLE_{4}(\underline{\rho}^L;\underline{\rho}^R)$ process with force points 
$(\underline{y}^L;\underline{y}^R)$, with $W,V^{q, i}$ solving the SDE system~\eqref{eqn::SLEkapparho} with $\kappa=4$. 
Let $(g_t, t\ge 0)$ be the corresponding family of conformal maps and set $f_t := g_t - W_t$. 
Let $u_t^0$ be the harmonic function on $\HH$ with boundary data 
\begin{align*}
\begin{cases}
-\lambda(1+\sum_{i=0}^j\rho^{L, i}) , \quad & \text{if } x \in (  f_t(y^{L, j+1}), f_t(y^{L, j})) , \\
+\lambda(1+\sum_{i=0}^j \rho^{R, i}) , \quad & \text{if } x \in ( f_t(y^{R, j}), f_t(y^{R, j+1})),
\end{cases}
\end{align*}
where $\lambda=\pi/2$ and 
$\rho^{L, 0}=\rho^{R, 0}=0$, $y^{L, 0}=0^-$, $\; y^{L, l+1}=-\infty$,$y^{R, 0}=0^+$, and $y^{R, r+1}=\infty$ by convention. 
Define $u_t(z) := u_t^0(f_t(z))$.
By~\cite{DubedatSLEFreefield, SchrammSheffieldContinuumGFF}, 
there exists a coupling $(\Gamma, K)$, where $\Gamma$ is a zero-boundary GFF on $\HH$, 
such that the following is true.  Let $\tau$ be any $\eta$-stopping time before 
the continuation threshold. Then, the conditional law of $\Gamma+u_0$ restricted to $\HH\setminus K_{\tau}$ given 
$K_{\tau}$ is the same as the law of $\Gamma'\circ f_{\tau}+u_{\tau}$ where $\Gamma'$ is a zero-boundary GFF. 
Furthermore, in this coupling, the process 
$K$ is almost surely determined by $\Gamma$. We refer to the $\SLE_{4}(\underline{\rho}^L;\underline{\rho}^R)$ 
in this coupling as the level line of the field $\Gamma+u_0$.  
In general, for $a\in\R$, the level line of $\Gamma+u_0$ with height $a$ is the level line of $\Gamma+u_0-a$.

\subsection{Connection probabilities}
\label{subsec::gffconnectionproba}
For $\alpha\in\Pair_N$, recall from Section~\ref{subsec::pairpartitionDyckpath} that $\alpha$ also denotes the corresponding Dyck path in $\DP_N$. Let $u_{\alpha}$ be the harmonic function on $\HH$ with the following boundary data: ($x_0=-\infty$ and $x_{2N+1}=\infty$ by convention)
\begin{equation}\label{eqn::gffboudnarydata}
2\lambda(\alpha(k)-1)\text{ on }(x_k, x_{k+1}),\quad\text{for all }k\in\{0,1,2,\ldots, N\}. 
\end{equation}
With such choice, we see that $u_{\alpha}$ has boundary data $-2\lambda$ on $(-\infty, x_1)\cup (x_{2N}, \infty)$, and has boundary data $0$ on $(x_1, x_2)\cup(x_{2N-1}, x_{2N})$. Define 
\begin{equation}\label{eqn::levellineheight}
\LH_{\alpha}(k)=\lambda(\alpha(k-1)+\alpha(k)-2), \quad\text{for all }k\in\{1, ,2 \ldots, 2N\}. 
\end{equation}

We write $\alpha=\{\{a_1, b_1\}, \ldots, \{a_N, b_N\}\}$ as ordered in~\eqref{eqn::LPtoDP_order}. Suppose $\Gamma$ is zero-boundary GFF on $\HH$, and consider level lines of $\Gamma+u_{\alpha}$. Let $\eta_{a_i}$ be the level line of $\Gamma+u_{\alpha}$ starting from $x_{a_i}$ with height $\LH_{\alpha}(a_i)$.  With such choice, the boundary data to the left side of $\eta_{a_i}$ is $2\lambda(\alpha(k-1)-1)$ and the boundary data to the right side of $\eta_{a_i}$ is $2\lambda(\alpha(k)-1)$. 
Then the $N$ curves $\{\eta_{a_1}, \eta_{a_2}, \ldots, \eta_{a_N}\}$ are non-intersecting simple curves and their end points form a planar pair partition of the $2N$ boundary points. We denote this planar pair partition by $\LA=\LA(\eta_{a_1}, \ldots, \eta_{a_N})\in\Pair_N$. See Figures~\ref{fig::gffboundarydata8}--\ref{fig::gfflevellines8}. The goal of this section is to derive the probabilities for $\PP[\LA=\beta]$. 

\begin{theorem}\label{thm::GFFconnectionproba}
Fix $\alpha\in\Pair_N$. Let $\Gamma+u_{\alpha}$ be the GFF on $\HH$ with boundary data given by~\eqref{eqn::gffboudnarydata}. Consider the planar pair partition $\LA$ formed by its level lines described as above. Then we have 
 \begin{equation}
 \PP[\LA=\beta]=\LM_{\alpha, \beta}\frac{\PartF_{\beta}(x_1, \ldots, x_{2N})}{\LU_{\alpha}(x_1, \ldots, x_{2N})}, \quad \text{for all }\beta\in \Pair_{N},
\end{equation}  
where $\{\PartF_{\beta}: \beta\in\Pair_{N}\}$ are pure partition functions for multiple $\SLE_4$, $\{\LU_{\alpha}: \alpha\in\DP_N\}$ are conformal block functions defined in~\eqref{eqn::conformalblock_def}, and $\{\LM_{\gamma, \beta}: \gamma, \beta\in\Pair_{N}\}$ is the incidence matrix defined through~\eqref{eqn::KWleincidencematrix}. 
\end{theorem}

Theorem~\ref{thm::GFFconnectionproba} is a generalization of~\cite[Theorem~1.4]{PeltolaWuGlobalMultipleSLEs} where the authors derive the connection probabilities for $\alpha=\{\{1,2\}, \{3, 4\}, \ldots, \{2N-1, 2N\}\}$.

%\begin{figure}[ht!]
%\begin{center}
%\includegraphics[width=0.18\textwidth]{figures/gffboundarydataDP6}$\quad$
%\includegraphics[width=0.24\textwidth]{figures/gffboundarydata6}$\quad$
%\includegraphics[width=0.24\textwidth]{figures/gffboundarydata6linkpattern1}$\quad$
%\includegraphics[width=0.24\textwidth]{figures/gffboundarydata6linkpattern2}
%\end{center}
%\caption{\label{fig::gffboundarydata6} Illustration of the boundary data with $N=3$: the planar pair partition for boudnary data is $\alpha=\{\{1,4\}, \{2, 3\}, \{5, 6\}\}$ as ordered in~\eqref{eqn::LPtoDP_order}. The curve $\eta_i$ is the level line starting from $x_i$ with height $-\lambda$ for $i=1,5$; the curve $\eta_2$ is the level line starting from $x_2$ with height $\lambda$. The three curves $\eta_1, \eta_2, \eta_5$ connect the six boundary points. Their end points give a planar pair partition, and there are two possibilities as indicated in the figure. From left to right, the two planar pair partitions are $\beta_1=\{\{1,6\}, \{2,3\}, \{5, 4\}\}$ and $\beta_2=\{\{1, 4\}, \{2, 3\}, \{5, 6\}\}$. Note that $\alpha\KWle\beta_i$ for $i=1,2$.}
%\end{figure}

\begin{figure}[ht!]
\begin{center}
\includegraphics[width=0.24\textwidth]{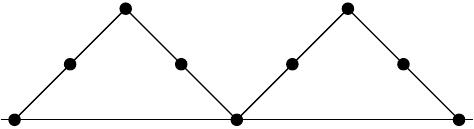}$\quad$
\includegraphics[width=0.24\textwidth]{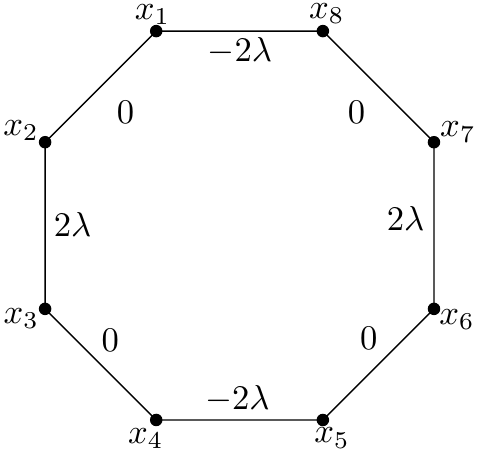}
\end{center}
\caption{\label{fig::gffboundarydata8} Illustration of the boundary data with $N=4$: the planar pair partition for boundary data is $\alpha=\{\{1, 4\}, \{2, 3\}, \{5, 8\}, \{6, 7\}\}$ as ordered in~\eqref{eqn::LPtoDP_order}. }
\end{figure}

 \begin{figure}[ht!]
\begin{center}
\includegraphics[width=0.24\textwidth]{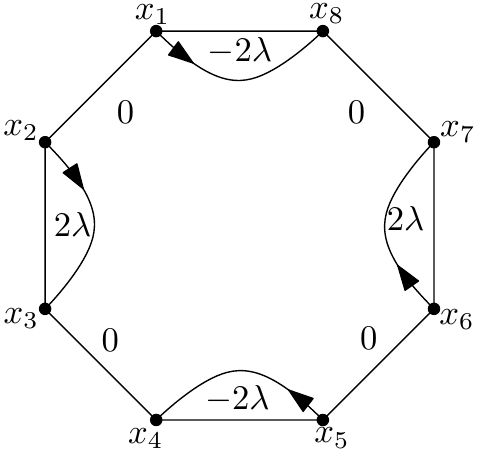}$\quad$
\includegraphics[width=0.24\textwidth]{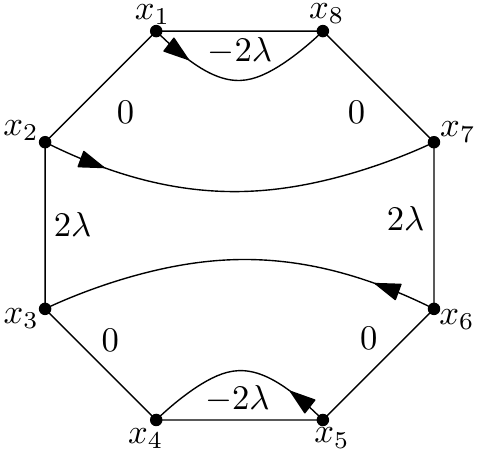}$\quad$
\includegraphics[width=0.24\textwidth]{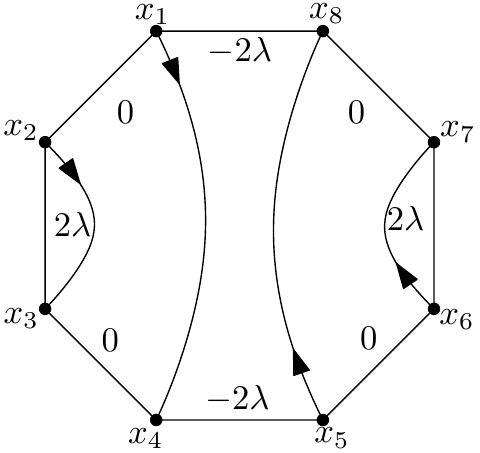}
\end{center}
\caption{\label{fig::gfflevellines8} Figure~\ref{fig::gffboundarydata8} continued: the curve $\eta_i$ is the level line starting from $x_i$ with height $-\lambda$ for $i=1,5$;  the curve $\eta_i$ is the level line starting from $x_i$ with height $\lambda$ for $i=2,6$. The four curves $\eta_1, \eta_2, \eta_5, \eta_6$ connect the eight boundary points. Their end points give a planar pair partition, and there are three possibilities as indicated in the figure. From left to right, the three planar pair partitions are $\beta_1=\{\{1, 8\}, \{2, 3\}, \{5, 4\}, \{6, 7\}\}$, $\beta_2=\{\{1, 8\}, \{2, 7\}, \{5, 4\}, \{6, 3\}\}$, $\beta_3=\{\{1, 4\}, \{2, 3\}, \{5, 8\}, \{6, 7\}\}$. Note that $\alpha\KWle\beta_i$ for $i=1,2,3$. }
\end{figure} 
 
 \begin{lemma}\label{lem::gffmart}
Let $\eta=\eta_1$ be the level line of $\Gamma+u_{\alpha}$ starting from $x_1$ with height $-\lambda$, let $(W_t, t\ge 0)$ be the driving function, and $(g_t, t\ge 0)$ be the corresponding conformal maps, and $T$ be the continuation threshold. For a smooth function $F: \chamber_{2N}\to \R$, the process
 \[M_t:=\frac{F(W_t, g_t(x_2), \ldots, g_t(x_{2N}))}{\LU_{\alpha}(W_t, g_t(x_2), \ldots, g_t(x_{2N}))}\]
 is a local martingale if and only if $F$ satisfies PDE~\eqref{eqn::multipleSLEsPDE} with $\kappa=4$ and $i=1$. 
 \end{lemma}
 \begin{proof}
The level line of $\Gamma+u_{\alpha}$ starting from $x_1$ with height $-\lambda$ is the $\SLE_{4}(\rho_{2},\ldots,\rho_{2N})$ process with force points $(x_{2},\ldots,x_{2N})$ and $\rho_{i}=2(\alpha(i)-\alpha(i-1))$.
Recalling from~\eqref{eqn::SLEkapparho}, its driving function $W_{t}$ satisfies the following intergrated SDEs up to the continuation threshold $T$: 
\begin{align*}
W_t =2 B_t +x_{1}+ \sum_{i=2}^{2N} 
\int_0^t\frac{\rho_{i} ds}{W_s-g_s(x_{i})} , \qquad 
g_t(x_{i}) = x_{i} +  \int_0^t\frac{2ds}{g_s(x_{i})-W_s} , 
\qquad \text{for } 2\le i\le 2N.
\end{align*}
We denote $\boldsymbol{Y}=(W_{t},g_{t}(x_{2}),\ldots,g_{t}(x_{2N}))$ and $X_{i1}=g_{t}(x_{i})-W_{t}$ for $2\le i\le 2N$. 
In this proof, we write $\partial_i$ for $\frac{\partial}{\partial x_i}$ as there is no ambiguity.
We denote the differential operator in~\eqref{eqn::multipleSLEsPDE} with $\kappa=4$ and $i=1$ by 
\[\LD^{(1)}:=2\partial^{2}_{1}+\sum_{i=2}^{2N}\left(\frac{2\partial_{i}}{x_{i}-x_{1}}-\frac{1}{2(x_{i}-x_{1})^{2}}\right).\]
By It$\hat{\rm{o}}$'s formula, we have
\begin{align*}
dF(\boldsymbol{Y})&=2\partial_{1}F(\boldsymbol{Y}) dB_{t}+\left(2\partial_{1}^{2}+\sum_{i=2}^{2N}\left(\frac{2\partial_{i}}{X_{i1}}-\frac{\rho_{i}\partial_{1}}{X_{i1}}\right)\right)F(\boldsymbol{Y})dt,\\
 &=2\partial_{1}F(\boldsymbol{Y}) dB_{t}+\left(\LD^{(1)}+\sum_{i=2}^{2N}\left(\frac{1}{2X_{i1}^{2}}-\frac{\rho_{i}\partial_{1}}{X_{i1}}\right)\right)F(\boldsymbol{Y})dt.
\end{align*}
We also have
\begin{align*}
\frac{d\LU_{\alpha}(\boldsymbol{Y})}{\LU_{\alpha}(\boldsymbol{Y})}=-\sum_{i=2}^{2N}\frac{\vartheta_{\alpha}(1,i)}{X_{i1}}dB_{t}+\left(\sum_{i=2}^{2N}\frac{1+\vartheta_{\alpha}(1,i)\rho_{i}}{2X_{i1}^{2}}+\sum_{2\le i\neq j\le 2N}\frac{\frac{1}{2}(\vartheta_{\alpha}(1,j)\rho_{i}+\vartheta_{\alpha}(1,i)\rho_{j})}{2X_{i1}X_{j1}}\right)dt.
\end{align*}
By definition, we have $\vartheta_{\alpha}(1,i)\rho_{i}=2$ for $2\le i\le 2N$ and $\vartheta_{\alpha}(1,j)\rho_{i}+\vartheta_{\alpha}(1,i)\rho_{j}=4\vartheta_{\alpha}(i,j)$ for $i\neq j$. 
Thus 
\[\frac{d\LU_{\alpha}(\boldsymbol{Y})}{\LU_{\alpha}(\boldsymbol{Y})}=-\sum_{i=2}^{2N}\frac{\vartheta_{\alpha}(1,i)}{X_{i1}}dB_{t}+\left(\sum_{i=2}^{2N}\frac{3}{2X_{i1}^{2}}+\sum_{2\le i\neq j\le 2N}\frac{\vartheta_{\alpha}(i, j)}{X_{i1}X_{j1}}\right)dt.\]
Therefore, we have 
\begin{align*}
\frac{dM_{t}}{M_{t}}&=\frac{dF(\boldsymbol{Y})}{F(\boldsymbol{Y})}-\frac{d\LU_{\alpha}(\boldsymbol{Y})}{\LU_{\alpha}(\boldsymbol{Y})}+4\left(\frac{\partial_{1}\LU_{\alpha}(\boldsymbol{Y})}{\LU_{\alpha}(\boldsymbol{Y})}\right)^{2}dt-4\left(\frac{\partial_{1}\LU_{\alpha}(\boldsymbol{Y})}{\LU_{\alpha}(\boldsymbol{Y})}\right)\left(\frac{\partial_{1}F(\boldsymbol{Y})}{F(\boldsymbol{Y})}\right)dt
\\ &=\left(\frac{2\partial_{1}F(\boldsymbol{Y})}{F(\boldsymbol{Y})}-\frac{2\partial_{1}\LU_{\alpha}(\boldsymbol{Y})}{\LU_{\alpha}(\boldsymbol{Y})}\right)dB_{t}+\frac{\LD^{(1)}F(\boldsymbol{Y})}{F(\boldsymbol{Y})}dt.
\end{align*}
Thus $M_{t}$ is a local martingale if and only if $F$ satisfies PDE~\eqref{eqn::multipleSLEsPDE} with $\kappa=4$ and $i=1$. 
 \end{proof}
 
 \begin{proof}[Proof of Theorem~\ref{thm::GFFconnectionproba}] We prove by induction on $N$.  We write $\alpha=\{\{a_1, b_1\}, \ldots, \{a_N, b_N\}\}$ as ordered in~\eqref{eqn::LPtoDP_order}. It suffices to show the conclusion for $\beta\in\Pair_N$ such that $\alpha\KWle \beta$. 
Because the summation of probabilities with such $\beta$'s equals one, and this implies that the probabilities for other planar pair partitions are zero.  
Fix $\beta\in\Pair_N$ such that $\alpha\KWle\beta$. There exists $j\in\{1, ,2 \ldots, 2N-1\}$ such that $\{j, j+1\}\in\beta$. In this case, we have $\wedge^j\in\beta$ and $\lozenge_j\in\alpha$. If $\wedge^j\in\alpha$, we let $\eta=\eta_j$ be the level line of $\Gamma+u_{\alpha}$ starting from $x_j$ with height $\LH_{\alpha}(j)$. If $\vee_j\in\alpha$, we let $\eta=\eta_{j+1}$ be the level line of $\Gamma+u_{\alpha}$ starting from $x_{j+1}$ with height $\LH_{\alpha}(j+1)$. The second case can be proved in a similar way as the first case. So we only give proof for the first case: we may assume $\wedge^j\in\alpha$. 
Let $\eta=\eta_j$ be the level line of $\Gamma+u_{\alpha}$ starting from $x_j$ with height $\LH_{\alpha}(j)$. Let $(W_t, t\ge 0)$ be the driving function, and $(g_t, t\ge 0)$ be the corresponding conformal maps, and $T$ be the continuation threshold. 

Define
 \[M_t:=\frac{\PartF_{\beta}(g_t(x_1), \ldots, g_t(x_{j-1}), W_t, g_t(x_{j+1}), g_t(x_{2N}))}{\LU_{\alpha}(g_t(x_1), \ldots, g_t(x_{j-1}), W_t, g_t(x_{j+1}), g_t(x_{2N}))}. \]
From a similar calculation as in Lemma~\ref{lem::gffmart}, this is a local martingale. From~\eqref{eqn::purepartitionvsconformalblock}, this is a bounded martingale. Optional stopping theorem gives $M_0=\E[M_T]$. 
We will analyze the behavior of the process as $t\to T$. 
Consider the level line $\eta$, it will terminate at a point $x_{n}$ such that $\alpha(n-1)=\alpha(j)$ and $\alpha(n)=\alpha(j-1)$. 

If $\eta(T)=x_{j+1}$, from~\eqref{eqn::conformalblockASYrefined} and~\eqref{eqn::purepartitionASYrefined}, we have, as $t\to T$, almost surely, 
\begin{align*}
M_t&=\frac{(g_t(x_{j+1})-W_t)^{1/2}\PartF_{\beta}(g_t(x_1), \ldots, g_t(x_{j-1}), W_t, g_t(x_{j+1}), g_t(x_{2N}))}{(g_t(x_{j+1})-W_t)^{1/2}\LU_{\alpha}(g_t(x_1), \ldots, g_t(x_{j-1}), W_t, g_t(x_{j+1}), g_t(x_{2N}))}\\
&\to \frac{\PartF_{\beta/\wedge_j}(g_T(x_1), \ldots, g_T(x_{j-1}), g_T(x_{j+2}), \ldots, g_T(x_{2N}))}{\LU_{\alpha/\lozenge_j}(g_T(x_1), \ldots, g_T(x_{j-1}), g_T(x_{j+2}), \ldots, g_T(x_{2N}))}. 
\end{align*}
If $\eta(T)=x_{n}$ with $n\neq j+1$, from Lemma~\ref{lem::valuemart} and~\eqref{eqn::multipleSLEPLBoptimal}, we have $\lim_{t\to T}M_t=0$ almost surely. 
In summary, we have 
\begin{align*}
M_0=\E[M_T]=\E\left[\one_{\{\eta(T)=x_{j+1}\}}\frac{\PartF_{\beta/\wedge_j}(g_T(x_1), \ldots, g_T(x_{j-1}), g_T(x_{j+2}), \ldots, g_T(x_{2N}))}{\LU_{\alpha/\lozenge_j}(g_T(x_1), \ldots, g_T(x_{j-1}), g_T(x_{j+2}), \ldots, g_T(x_{2N}))}\right]. 
\end{align*}
By induction hypothesis, we have
\[\PP[\LA=\beta\cond \eta[0,T]]=\frac{\PartF_{\beta/\wedge_j}(g_T(x_1), \ldots, g_T(x_{j-1}), g_T(x_{j+2}), \ldots, g_T(x_{2N}))}{\LU_{\alpha/\lozenge_j}(g_T(x_1), \ldots, g_T(x_{j-1}), g_T(x_{j+2}), \ldots, g_T(x_{2N}))}. \]
Therefore, $M_0=\PP[\LA=\beta]$ as desired. 
\end{proof}

\subsection{Technical lemmas}
\label{subsec::technicallemmas}
The following three lemmas are technical. Lemmas~\ref{lem::valuemartaux} and~\ref{lem::crossbound} are needed in the proof of Lemma~\ref{lem::valuemart} which is essential in the proof of Theorem~\ref{thm::GFFconnectionproba}.

\begin{lemma}\label{lem::valuemartaux}
Let $x_{1}<x_{2}<x_{3}<x_{4}$. Suppose $\eta$ is a continuous simple curve in $\HH$ starting from $x_1$ and terminating at $x_4$ at time $T$. Assume $\eta$ hits $\R$ only at its two end points. 
Let $(W_t, 0\le t\le T)$ be its driving function and $(g_t, 0\le t\le T)$ be the corresponding family of conformal maps. Then 
\begin{align*}
\lim_{t\to T}\frac{(g_{t}(x_{3})-g_{t}(x_{2}))(g_{t}(x_{4})-W_{t})}{(g_{t}(x_{3})-W_{t})(g_{t}(x_{4})-g_{t}(x_{2}))}=0.
\end{align*}
\end{lemma}
\begin{proof}
See~\cite[Lemma~B.2]{PeltolaWuGlobalMultipleSLEs}.
\end{proof}
\begin{lemma}\label{lem::crossbound}
Let $x_{0}<x_{1}<x_{2}<x_{3}<x_{4}$. Suppose $\eta$ is a continuous simple curve in $\HH$ starting from $x_0$ and terminating at $x_4$ at time $T$. Assume $\eta$ hits $\R$ only at its two end points. 
Let $(W_t, 0\le t\le T)$ be its driving function and $(g_t, 0\le t\le T)$ be the corresponding family of conformal maps. Then there exist $C_{1}, C_{2}>0$, which depend on $\eta[0,T]$, such that for all $t\in [0,T]$, 
\begin{align*}
C_{1}\le\left|\frac{(g_{t}(x_{2})-g_{t}(x_{1}))(g_{t}(x_{3})-W_{t})}{(g_{t}(x_{2})-W_{t})(g_{t}(x_{3})-g_{t}(x_{1}))}\right|\le C_{2}.
\end{align*}
\end{lemma}
\begin{proof}
To prove the conclusion, we will show the following two estimates: First, we will show that there exist $C_{1}, C_{2}>0$, which only depend on $\eta[0,T]$, such that for all $t\in [0,T]$, 
\begin{align}\label{eqn::bounded}
C_{1}\le\frac{g_{t}(x_{2})-g_{t}(x_{1})}{g_{t}(x_{3})-g_{t}(x_{1})}\le C_{2}. 
\end{align}
Second, we will show 
\begin{align}\label{eqn::1}
\lim_{t\to T}\frac{g_{t}(x_{3})-W_{t}}{g_{t}(x_{2})-W_{t}}=1+\lim_{t\to T}\frac{g_{t}(x_{3})-g_{t}(x_{2})}{g_{t}(x_{2})-W_{t}}=1.
\end{align}

In this proof, we use $\asymp$ to simplify notations: for two functions $f$ and $g$, the notation $f\asymp g$ means that there exists a constant $C>0$ which only depends on $\eta[0,T]$ such that $C^{-1}\le f/g\le C$.

We first show~\eqref{eqn::bounded}. 
Note that for an interval $[a,b]$, we have $b-a=\lim_{y\to\infty}\pi y\PP^{iy}\left[\text{BM hits }\partial\HH\, \text{in }[a,b]\right]$, where $\text{BM}$ is the Brownian motion starts from $iy$. By conformal invariance of the Brownian motion, we have $b-a=\lim_{y\to\infty}\pi y\PP^{g_{t}^{-1}(iy)}\left[\text{BM hits } \partial\left(\HH\setminus\eta[0,t]\right) \text{in }g_{t}^{-1}([a,b])\right]$. 

We choose $\delta_{0}$ small enough, such that the $\delta_{0}$-neighborhood of the interval $\left[x_{1},x_{3}\right]$ does not intersect $\eta[0,T]$. We denote the boundary of this neighborhood in $\HH$ by $\gamma$, this is a simple curve. For the Brownian motion starting  from $g_{t}^{-1}(iy)$, let $\tau$ be the first time the Brownian motion hits $\gamma$. Consider the connected component $V$ of $\HH\setminus\eta[0,T]$ which contains $x_{1}$  on its boundary and choose a point $z\in V$. Suppose $\U$ is the unit disk, and $\phi_{t}:\HH\setminus\eta[0,t]\to \U$ is the conformal map with $\phi_{t}(z)=0$, $\phi_{t}'(z)>0$. Suppose $\phi_{T}: V\to \U$ is the conformal map with the same normalization. Then, for any compact set $K\subset\overline V$ which does not intersect $\eta[0,T]$, the conformal map $\phi_{t}$ converges to $\phi_{T}$ uniformly on $K$.

Note that
\[\PP^{g_{t}^{-1}(iy)}\left[\text{BM}\,\text{hits } \partial\left(\HH\setminus\eta[0,t]\right) \text{in  } [x_{1},x_{2}]\right]=\PP^{g_{t}^{-1}(iy)}\left[\one_{\{\tau<\infty\}}\PP^{B_{\tau}}\left[\text{BM}\, \text{hits } \partial\left(\HH\setminus\eta[0,t]\right) \text{in } [x_{1},x_{2}]\right]\right].\]
We will compare $\PP^{B_{\tau}}\left[\text{BM}\, \text{hits } \partial\left(\HH\setminus\eta[0,t]\right) \text{in } [x_{1},x_{2}]\right]$ and $\PP^{B_{\tau}}\left[\text{BM}\, \text{hits } \partial\left(\HH\setminus\eta[0,t]\right) \text{in } [x_{1},x_{3}]\right]$, in fact we can replace $B_{\tau}$ by a deterministic point on $\gamma$. For every $w\in\gamma$, we have 
\[\PP^{w}\left[\text{BM}\, \text{hits } \partial\left(\HH\setminus\eta[0,t]\right) \text{in } [x_{1},x_{2}]\right]=\PP^{\phi_{t}(w)}\left[\text{BM}\, \text{hits } \partial\U\ \text{in } [\phi_{t}(x_{1}),\phi_{t}(x_{2})]\right],\]
where $[\phi_{t}(x_{1}),\phi_{t}(x_{2})]$ is the conformal image of $[x_{1},x_{2}]$. By direct computation, the right hand-side equals 
\[\frac{1}{2\pi}\left(\arg \frac{\phi_{t}(x_{2})-\phi_{t}(w)}{1-\overline{\phi_{t}(w)}\phi_{t}(x_{2})}-\arg \frac{\phi_{t}(x_{1})-\phi_{t}(w)}{1-\overline{\phi_{t}(w)}\phi_{t}(x_{1})}\right),\]
where $\arg$ is the argument principal which takes value in $[0,2\pi)$. Note that there exists $\epsilon_{0}>0$ such that
 \[\frac{1}{2\pi}\left(\arg \frac{\phi_{t}(x_{2})-\phi_{t}(w)}{1-\overline{\phi_{t}(w)}\phi_{t}(x_{2})}-\arg \frac{\phi_{t}(x_{1})-\phi_{t}(w)}{1-\overline{\phi_{t}(w)}\phi_{t}(x_{1})}\right)\le 1-\epsilon_{0},\]
because $\gamma$ is bounded away from $[x_{1},x_{3}]$. Thus,
\begin{align*}
\PP^{\phi_{t}(w)}\left[\text{BM}\ \text{hits}\ \partial\U\ \text{in}\ [\phi_{t}(x_{1}),\phi_{t}(x_{2})]\right]&\asymp \left|\frac{\phi_{t}(x_{2})-\phi_{t}(w)}{1-\overline{\phi_{t}(w)}\phi_{t}(x_{2})}-\frac{\phi_{t}(x_{1})-\phi_{t}(w)}{1-\overline{\phi_{t}(w)}\phi_{t}(x_{1})}\right|\\ &= \frac{(1-|\phi_{t}(w)|^{2})|\phi_{t}(x_{2})-\phi_{t}(x_{1})|}{|1-\overline{\phi_{t}(w)}\phi_{t}(x_{2})||1-\overline{\phi_{t}(w)}\phi_{t}(x_{1})|}.
\end{align*}
Similarly, we have
\begin{align*}
\PP^{\phi_{t}(w)}\left[\text{BM}\, \text{hits}\ \partial\U\ \text{in}\ [\phi_{t}(x_{1}),\phi_{t}(x_{3})]\right]\asymp \frac{(1-|\phi_{t}(w)|^{2})|\phi_{t}(x_{3})-\phi_{t}(x_{1})|}{|1-\overline{\phi_{t}(w)}\phi_{t}(x_{3})||1-\overline{\phi_{t}(w)}\phi_{t}(x_{1})|}.
\end{align*}
Therefore,
\[\frac{\PP^{\phi_{t}(w)}\left[\text{BM}\, \text{hits } \partial\U\, \text{in } [\phi_{t}(x_{1}),\phi_{t}(x_{2})]\right]}{\PP^{\phi_{t}(w)}\left[\text{BM}\, \text{hits } \partial\U\, \text{in } [\phi_{t}(x_{1}),\phi_{t}(x_{3})]\right]}\asymp \frac{|\phi_{t}(x_{2})-\phi_{t}(x_{1})|}{|\phi_{t}(x_{3})-\phi_{t}(x_{1})|}\frac{|\phi_{t}(x_{3})-\phi_{t}(w)|}{|\phi_{t}(x_{2})-\phi_{t}(w)|}\asymp 1.\]
The last $\asymp$ is because of the uniform convergence of $\phi_{t}$. 
Thus, we have
\[\PP^{B_{\tau}}\left[\text{BM}\, \text{hits } \partial\left(\HH\setminus\eta[0,t]\right) \text{in } [x_{1},x_{2}]\right]\asymp\PP^{B_{\tau}}\left[\text{BM}\, \text{hits } \partial\left(\HH\setminus\eta[0,t]\right) \text{in } [x_{1},x_{3}]\right].\]
This implies~\eqref{eqn::bounded}.

Next, we show~\eqref{eqn::1}. 
Consider the Brownian motion starting from $g_{t}^{-1}(iy)$. We define $C(x_{4},\delta):=\{z\in\HH:d(z,x_{4})=\delta\}$. Let $\tau_{\delta}$ be the first time that it hits the connected component of half circle $C(x_{4},\delta)\cap\HH\setminus\eta[0,t]$ which contains $x_{4}-\delta$  on its boundary and we denote this connected component by $C_{\delta}$.
Then we have 
\begin{align*}
&\PP^{g_{t}^{-1}(iy)}\left[\text{BM}\, \text{hits } \partial\left(\HH\setminus\eta[0,t]\right) \text{in the right side of } \eta[0,t]\cup[x_{0},x_{2}]\right]\\&\ge\PP^{g_{t}^{-1}(iy)}\left[\one_{\{\tau_{\delta}<\infty\}}\PP^{B_{\tau_{\delta}}}\left[\text{BM}\, \text{hits } \partial\left(\HH\setminus\eta[0,t] \right) \text{in the right side of } \eta[0,t]\cup[x_{0},x_{2}]\right]\right],
\end{align*}
and 
\begin{align*}
\PP^{g_{t}^{-1}(iy)}\left[\text{BM hits } \partial\left(\HH\setminus\eta[0,t]\right) \text{in } [x_{2},x_{3}]\right]=\PP^{g_{t}^{-1}(iy)}\left[\one_{\{\tau_{\delta}<\infty\}}\PP^{B_{\tau_{\delta}}}\left[\text{BM}\, \text{hits } \partial\left(\HH\setminus\eta[0,t]\right) \text{in } [x_{2},x_{3}]\right]\right].
\end{align*}
By conformal invariance of the Brownian motion, we have
\begin{align*}
\PP^{B_{\tau_{\delta}}}\left[\text{BM hits } \partial\left(\HH\setminus\eta[0,t]\right) \text{in the right side of } \eta[0,t]\cup[x_{0},x_{2}]\right]=\PP^{\phi_{t}(B_{\tau_{\delta}})}\left[\text{BM hits } \partial\U\, \text{in } [\phi_{t}(\eta(t)),\phi_{t}(x_{2})]\right],
\end{align*}
where $[\phi_{t}(\eta(t)),\phi_{t}(x_{2})]$ is the conformal image of the right side of $\eta[0,t]\cup[x_{0}, x_{2}]$. Moreover
\begin{align*}
\PP^{B_{\tau_{\delta}}}\left[\text{BM hits } \partial\left(\HH\setminus\eta[0,t]\right) \text{in } [x_{2},x_{3}]\right]=\PP^{\phi_{t}(B_{\tau_{\delta}})}\left[\text{BM hits } \partial\U\, \text{in } [\phi_{t}(x_{2}),\phi_{t}(x_{3})]\right],
\end{align*}
where $[\phi_{t}(x_{2}),\phi_{t}(x_{3})]$ is the conformal image of $[x_{2},x_{3}]$. We replace $B_{\tau_{\delta}}$ by a deterministic point on $C_{\delta}$ for the same reason as in the proof of~\eqref{eqn::bounded}. For every $w\in C_{\delta}$, by Beurling estimate and conformal invariance, there exists $C>0$ such that 
\[\PP^{\phi_{t}(w)}\left[\text{BM hits } \partial\U\, \text{in } [\phi_{t}(x_{2}),\phi_{t}(x_{3})]\right]\le C\left(\frac{\delta}{x_{4}-x_{3}}\right)^{\frac{1}{2}}.\]
This implies that there exists $\epsilon_{0}>0$ such that 
\[\PP^{\phi_{t}(w)}\left[\text{BM hits } \partial\U\, \text{in } [\phi_{t}(x_{2}),\phi_{t}(x_{3})]\right]\le 1-\epsilon_{0}.\]
Thus, by the same method as in the proof of~\eqref{eqn::bounded}, we have
\[\PP^{\phi_{t}(w)}\left[\text{BM hits } \partial\U\, \text{in } [\phi_{t}(x_{2}),\phi_{t}(x_{3})]\right]\asymp \frac{(1-|\phi_{t}(w)|^{2})|\phi_{t}(x_{3})-\phi_{t}(x_{2})|}{|1-\overline{\phi_{t}(w)}\phi_{t}(x_{3})||1-\overline{\phi_{t}(w)}\phi_{t}(x_{2})|}.\] 
Moreover,
\begin{align*}
\PP^{\phi_{t}(w)}\left[\text{BM hits } \partial\U\, \text{in } [\phi_{t}(\eta(t)),\phi_{t}(x_{2})]\right]&=\frac{1}{2\pi}\left(\arg \frac{\phi_{t}(x_{2})-\phi_{t}(w)}{1-\overline{\phi_{t}(w)}\phi_{t}(x_{2})}-\arg \frac{\phi_{t}(\eta(t))-\phi_{t}(w)}{1-\overline{\phi_{t}(w)}\phi_{t}(\eta(t))}\right)\\
&\ge\frac{1}{2\pi}\left|\frac{\phi_{t}(x_{2})-\phi_{t}(w)}{1-\overline{\phi_{t}(w)}\phi_{t}(x_{2})}-\frac{\phi_{t}(\eta(t))-\phi_{t}(w)}{1-\overline{\phi_{t}(w)}\phi_{t}(\eta(t))}\right|\\
&=\frac{1}{2\pi}\frac{(1-|\phi_{t}(w)|^{2})|\phi_{t}(\eta(t))-\phi_{t}(x_{2})|}{|1-\overline{\phi_{t}(w)}\phi_{t}(\eta(t))||1-\overline{\phi_{t}(w)}\phi_{t}(x_{2})|}.
\end{align*}
Combining these two together, there exists $C>0$ such that
\[\frac{\PP^{\phi_{t}(w)}\left[\text{BM hits } \partial\U\, \text{in } [\phi_{t}(\eta(t)),\phi_{t}(x_{2})]\right]}{\PP^{\phi_{t}(w)}\left[\text{BM hits } \partial\U\, \text{in } [\phi_{t}(x_{2}),\phi_{t}(x_{3})]\right]}\ge C\frac{|\phi_{t}(\eta(t))-\phi_{t}(x_{2})|}{|\phi_{t}(x_{3})-\phi_{t}(x_{2})|}\frac{|\phi_{t}(x_{3})-\phi_{t}(w)|}{|\phi_{t}(\eta(t))-\phi_{t}(w)|}.\]
We denote the connected component of $\HH\setminus (\eta[0,t]\cup C_{\delta})$ which contains $\infty$ by $A$. By the relation between diameter and harmonic measure, there exists $C_{1}>0$, such that
\begin{align*}
\frac{1}{C_{1}}\diam\left(\phi_{t}(A)\right)\le \PP^{0}[\text{BM hits }\phi_{t}(C_{\delta})\, \text{before } \partial\U]&=\PP^{z}[\text{BM hits } C_{\delta}\, \text{before } \partial(\HH\setminus\eta[0,t])]\\&\le\PP^{z}[\text{BM hits } C_{\delta}\, \text{before } \partial\HH]\\&\le C_{1}\delta.
\end{align*}
Thus, we have 
\[|\phi_{t}(\eta(t))-\phi_{t}(w)|\le C_{1}^{2}\delta.\]
For $|\phi_{t}(\eta(t))-\phi_{t}(x_{2})|$ and $|\phi_{t}(x_{3})-\phi_{t}(w)|$, we have
\begin{align*}
|\phi_{t}(\eta(t))-\phi_{t}(x_{2})|&\ge \min\left\{\left|\phi_{t}\left(\frac{x_{4}+x_{3}}{2}\right)-\phi_{t}(x_{2})\right|,\left|\phi_{t}\left(\frac{x_{1}+x_{0}}{2}\right)-\phi_{t}(x_{2})\right|\right\},\\
|\phi_{t}(w)-\phi_{t}(x_{3})|&\ge |\phi_{t}(\eta(t))-\phi_{t}(x_{3})|-|\phi_{t}(\eta(t))-\phi_{t}(w)|\\ &\ge \min\left\{\left|\phi_{t}\left(\frac{x_{4}+x_{3}}{2}\right)-\phi_{t}(x_{3})\right|,\left|\phi_{t}\left(\frac{x_{1}+x_{0}}{2}\right)-\phi_{t}(x_{3})\right|\right\}-C_{1}^{2}\delta.
\end{align*}
Thus, by the uniform convergence, there exists $C_{2}>0$ such that 
\[\frac{|\phi_{t}(\eta(t))-\phi_{t}(x_{2})||\phi_{t}(x_{3})-\phi_{t}(w)|}{|\phi_{t}(x_{3})-\phi_{t}(x_{2})|}\ge C_{2}.\]
This implies that there exists $C>0$ such that
\[\frac{\PP^{\phi_{t}(w)}\left[\text{BM hits } \partial\U\, \text{in } [\phi_{t}(\eta(t)),\phi_{t}(x_{2})]\right]}{\PP^{\phi_{t}(w)}\left[\text{BM hits } \partial\U\, \text{in } [\phi_{t}(x_{2}),\phi_{t}(x_{3})]\right]}\ge C\frac{1}{\delta}.\]
Therefore, we have~\eqref{eqn::1}. Combining~\eqref{eqn::bounded} and~\eqref{eqn::1}, we obtain the conclusion. 
\end{proof}

We set $\LB_{\emptyset}=1$, and for $\alpha, \beta\in\Pair_N$ and $x_1<\cdots<x_{2N}$, we define
\begin{equation*}
\LB_{\beta}(x_1, \ldots, x_{2N}):=\prod_{\{a,b\}\in\beta}|x_a-x_b|^{-1}, \qquad F_{\alpha, \beta}(x_1, \ldots, x_{2N}):=\frac{\LB_{\beta}(x_1, \ldots, x_{2N})}{\LU_{\alpha}(x_1, \ldots, x_{2N})^2}. 
\end{equation*}

\begin{lemma}\label{lem::valuemart}
Fix $\alpha, \beta\in\Pair_N$ such that $\alpha\KWle \beta$. Fix $j\in\{1, 2, \ldots, 2N-1\}$, we assume that $\wedge_j\in\alpha$ and $\wedge_j\in\beta$. 
Fix $n\in \{1, \ldots, j-1, j+2, \ldots, 2N\}$ such that $\alpha(n-1)=\alpha(j)$ and $\alpha(n)=\alpha(j-1)$. Fix $x_1<\cdots<x_{2N}$. 
Suppose $\eta$ is a continuous simple curve in $\HH$ starting from $x_j$ and terminating at $x_n$ at time $T$. Assume $\eta$ hits $\R$ only at its two end points. 
Let $(W_t, 0\le t\le T)$ be its driving function and $(g_t, 0\le t\le T)$ be the corresponding family of conformal maps. Then 
\[\lim_{t\to T}F_{\alpha, \beta}(g_t(x_1), \ldots, g_t(x_{j-1}), W_t, g_t(x_{j+1}), g_t(x_{2N}))=0.\]
\end{lemma}

\begin{proof}
We may assume $j+1<n$. The other case can be proved similarly. 
By definition, we have
\[F_{\alpha,\beta}(x_{1},\ldots,x_{2N})=\prod_{\substack{1\le i\le 2N\\ i\neq j,j+1}}\left|\frac{x_{i}-x_{j+1}}{x_{i}-x_{j}}\right|^{\vartheta_{\alpha}(i,j)}F_{\alpha/\wedge_{j},\beta/\wedge_{j}}(x_{1},\ldots,x_{j-1},x_{j+2},\ldots,x_{2N}).\]
To get the conclusion, we will prove the following two estimates: 
\begin{align}\label{eqn::prodzero}
\lim_{t\to T}\prod_{\substack{1\le i\le 2N\\ i\neq j,j+1}}\left|\frac{g_{t}(x_{i})-g_{t}(x_{j+1})}{g_{t}(x_{i})-W_{t}}\right|^{\vartheta_{\alpha}(i,j)}=0, 
\end{align}
and 
\begin{align}\label{eqn::prodfinite}
\sup_{0\le t\le T}F_{\alpha/\wedge_{j},\beta/\wedge_{j}}(g_{t}(x_{1}),\ldots,g_{t}(x_{j-1}),g_{t}(x_{j+2}),\ldots,g_{t}(x_{2N}))<\infty. 
\end{align}

Suppose $\alpha=\{\{a_1, b_1\}, \ldots, \{a_N, b_N\}\}$ is ordered as in~\eqref{eqn::LPtoDP_order}. The number of elements in two sets of indexes $A=\{i:j+1<i\le n, i\in\{a_{1},\ldots,a_{N}\}\}$ and $B=\{i:j+1<i\le n, i\in\{b_{1},\ldots,b_{N}\}\}$ are equal. Note that $n\in B$. We choose the increasing bijection $\xi: A\to B$ and suppose $\xi(i_{0})=n$.

We first show~\eqref{eqn::prodzero}. We write
\begin{align*}
\prod_{\substack{1\le i\le 2N\\ i\neq j,j+1}}\left|\frac{g_{t}(x_{i})-g_{t}(x_{j+1})}{g_{t}(x_{i})-W_{t}}\right|^{\vartheta_{\alpha}(i,j)}=&\prod_{i<j\,\text{or }i>n}\left|\frac{g_{t}(x_{i})-g_{t}(x_{j+1})}{g_{t}(x_{i})-W_{t}}\right|^{\vartheta_{\alpha}(i,j)}\\
&\times\prod_{\substack{i\in A\\i\neq i_{0}}}\left(\left|\frac{g_{t}(x_{i})-g_{t}(x_{j+1})}{g_{t}(x_{i})-W_{t}}\right|\left|\frac{g_{t}(x_{\xi(i)})-W_{t}}{g_{t}(x_{\xi(i)})-g_{t}(x_{j+1})}\right|\right)\\ 
&\times\left|\frac{g_{t}(x_{i_{0}})-g_{t}(x_{j+1})}{g_{t}(x_{i_{0}})-W_{t}}\right|\left|\frac{g_{t}(x_{n})-W_{t}}{g_{t}(x_{n})-g_{t}(x_{j+1})}\right|
\end{align*}
By Lemma~\ref{lem::valuemartaux}, we have 
\begin{align}\label{eqn::vanishterm}
\lim_{t\to T}\left|\frac{g_{t}(x_{i_{0}})-g_{t}(x_{j+1})}{g_{t}(x_{i_{0}})-W_{t}}\right|\left|\frac{g_{t}(x_{n})-W_{t}}{g_{t}(x_{n})-g_{t}(x_{j+1})}\right|= 0.
\end{align}
By Lemma~\ref{lem::crossbound}, there exist $C_{1}$,$C_{2}>0$,  which only depend on $\eta[0,T]$, such that for any $i\in A$ with $\xi(i)\neq n$, we have for all $t\in [0,T]$, 
\begin{align}\label{eqn::crossingratio}
C_{1}\le\left|\frac{g_{t}(x_{i})-g_{t}(x_{j+1})}{g_{t}(x_{i})-W_{t}}\right|\left|\frac{g_{t}(x_{\xi(i)})-W_{t}}{g_{t}(x_{\xi(i)})-g_{t}(x_{j+1})}\right|\le C_{2}.
\end{align}
For $i\notin A\cup B$, 
\[\lim_{t\to T}\frac{g_{t}(x_{i})-g_{t}(x_{j+1})}{g_{t}(x_{i})-W_{t}}= \frac{g_{T}(x_{i})-g_{T}(x_{j+1})}{g_{T}(x_{i})-W_{T}}.\]
Combining with~\eqref{eqn::crossingratio} and~\eqref{eqn::vanishterm}, we obtain~\eqref{eqn::prodzero}.

Next, we prove~\eqref{eqn::prodfinite}. We write
\begin{align*}
F_{\alpha/\wedge_{j},\beta/\wedge_{j}}(g_{t}(x_{1}),\ldots,g_{t}(x_{j-1}),g_{t}(x_{j+2}),\ldots,g_{t}(x_{2N}))=\frac{\prod_{\substack{\{a_{i},b_{i}\}\in\beta/\wedge_{j} \\ a_{i}\notin A\ \text{or}\ b_{i}\notin B}}\left(g_{t}(x_{b_{i}})-g_{t}(x_{a_{i}})\right)^{-1}}{\prod_{\substack{i\notin A\cup B\\ \text{or}\ k\notin A\cup B}}|g_{t}(x_{k})-g_{t}(x_{i})|^{\vartheta_{\alpha}(i,k)}}\times S_{t},
\end{align*}
where 
\[S_{t}=\frac{\prod_{\substack{\{a_{i},b_{i}\}\in\beta/\wedge_{j} \\ a_{i}\in A\ \text{and}\ b_{i}\in B}}\left(g_{t}(x_{b_{i}})-g_{t}(x_{a_{i}})\right)^{-1}}{\prod_{\substack{i\in A\cup B\setminus\{n\}\\ \text{and}\ k\in A\cup B\setminus\{n\}}}|g_{t}(x_{k})-g_{t}(x_{i})|^{\vartheta_{\alpha}(i,k)}\prod\limits_{i\in A\cup B\setminus\{n\}}|g_{t}(x_{n})-g_{t}(x_{i})|^{\vartheta_{\alpha}(i,n)}}.\]
In this decomposition, we have 
\[\frac{\prod_{\substack{\{a_{i},b_{i}\}\in\beta/\wedge_{j} \\ a_{i}\notin A\ \text{or}\ b_{i}\notin B}}\left(g_{t}(x_{b_{i}})-g_{t}(x_{a_{i}})\right)^{-1}}{\prod_{\substack{i\notin A\cup B\\ \text{or}\ k\notin A\cup B}}|g_{t}(x_{k})-g_{t}(x_{i})|^{\vartheta_{\alpha}(i,k)}}\asymp 1,\]
because both the numerator and the denominator converge to a bounded and nonzero quantity as $t\to T$. 
Here the notation $\asymp$ is defined in the same way as in the proof of Lemma~\ref{lem::crossbound}. By~\eqref{eqn::bounded}, for distinct $i, k\in A\cup B\setminus\{n\}$, we have
\[|g_{t}(x_{k})-g_{t}(x_{i})|\asymp g_{t}(x_{j+2})-g_{t}(x_{j+1}).\]
By the same method as in the proof of~\eqref{eqn::1}, for $i\in A\cup B\setminus\{n\}$, we have
\[\lim_{t\to T}\frac{g_{t}(x_{n})-g_{t}(x_{i})}{g_{t}(x_{n})-g_{t}(x_{n-1})}=1.\]
Thus we have 
\[S_{t}\asymp \prod_{\substack{\{a_{i},b_{i}\}\in\beta/\wedge_{j} \\ a_{i}\in A\ \text{and}\ b_{i}\in B}}\left(g_{t}(x_{b_{i}})-g_{t}(x_{a_{i}})\right)^{-1}(g_{t}(x_{j+2})-g_{t}(x_{j+1}))^{\#A-1}(g_{t}(x_{n})-g_{t}(x_{n-1})).\]
When there is $a\in A$ such that $\{a,n\}\in\beta/\wedge_{j}$, and $\#A-1$ pairs $\{a_{i},b_{i}\}\in\beta/\wedge_{j}$ such that $a_{i}\in A\ \text{and}\ b_{i}\in B$, we have $S_{t}\asymp 1$. Otherwise, we have $\lim_{t\to T}S_{t}=0$.
This gives~\eqref{eqn::prodfinite} and completes the proof. 
\end{proof}

%% file: tex/mgff.tex
In this section, we first introduce discrete GFF and metric graph GFF in Section~\ref{subsec::dGFFmGFF}, and then we introduce first passage set in Section~\ref{subsec::fps}. In Section~\ref{subsec::mgffcvg}, we show that the crossing probabilities in metric graph GFF converges to the probability of certain connection probabilities in continuum GFF, see Proposition~\ref{prop::mgffconvergence}. This gives the first half of the proof of Theorem~\ref{thm::main}. In order to calculate the desired connection probabilities in continuum GFF, we use Theorem~\ref{thm::GFFconnectionproba} and a result about asymptotics of pure partition functions---Proposition~\ref{prop::purepartitionfusionall}. Section~\ref{subsec::mgffasy} proves Proposition~\ref{prop::purepartitionfusionall} and Proposition~\ref{prop::mGFFpdecov}, and it is quite independent of the rest of the section. Finally, we complete the proof of Theorem~\ref{thm::main} in Section~\ref{subsec::mgfffinal}.

\subsection{Discrete GFF and metric graph GFF}
\label{subsec::dGFFmGFF}

In this section, we review basic definition and properties of discerte GFF and metric graph GFF. We refer to~\cite{SchrammSheffieldDiscreteGFF, AruLupuSepulvedaFPSGFFCVGISO} for details. 
Suppose $\LG=(V,E)$ is a connected planar graph, and $\partial\LG$ is a given subset of $V$ which we call the boundary of $\LG$. 
We equip each edge $e=\{x,y\}$ with conductance $C(e)=C(x,y)>0$. 
Let $\Delta$ be the discrete Laplacian on $\LG$:
\[(\Delta f)(x)=\sum_{y\sim x}C(x,y)(f(y)-f(x)),\, \forall x\in V\setminus \partial \LG.\]
The discrete Green's function $G_{\LG}$ is the inverse of $-\Delta$ with zero-boundary condition on $\partial \LG$. The discrete GFF is the centered Gaussian process $\left(\Gamma^{\LG}(v): v\in V\right)$ with covariance given by Green's function: 
\[\E\left[\Gamma^{\LG}(x)\Gamma^{\LG}(y)\right]=G_{\LG}(x,y),\quad \forall x,y\in V.\]

\medbreak
Suppose $\LG=(V,E)$ is a connected planar graph with boundary $\partial\LG$ and conductance $(C(e), e\in E)$. For each $e\in E$, we view it as a line segment in the plane, and for every $x',y'\in e$, we define\footnote{Here we use the normalization in~\cite{AruLupuSepulvedaFPSGFFCVGISO} which is distinct from the one in~\cite{LupuFreeField}.} 
\[m([x',y'])=\frac{1}{C(e)}\frac{|x'-y'|}{|x-y|}.\] 
This defines a length measure $dm$ on $\LG$. We call $(\LG,dm)$ metric graph of $\LG$ and we denote it by $\tilde{\LG}$. 

The metric graph GFF $\Gamma^{\tilde\LG}$ can be constructed as follows, see~\cite{LupuFreeField}.
First, we sample the discrete GFF $\left(\Gamma^{\LG}(v): v\in V\right)$. Then, conditional on $\left(\Gamma^{\LG}(v): v\in V\right)$, for each $e=\{x,y\}\in E$, we sample an independent Brownian bridge with length $m([x,y])$ and two terminal values $\Gamma^{\LG}(x)$ and $\Gamma^{\LG}(y)$. This defines the metric graph GFF with zero-boundary condition and we denote it by $\left(\Gamma^{\tilde\LG}(z): z\in \tilde\LG\right)$. Given a function $u:\partial \LG\to \R$, we choose the discrete harmonic extension of $u$ to $V\setminus \partial \LG$ and then extend it inside each edge by linear interpolation. We still denote this function by $u$ and view it as the harmonic function on the metric graph. We call $\Gamma^{\tilde\LG}+u$ the metric graph GFF with boundary data $u$.

\subsection{First passage sets}
\label{subsec::fps}

In this section, we introduce first passage sets for metric graph GFF. Suppose $\Gamma^{\tilde\LG}+u$ is the metric graph GFF with boundary data $u$. For every $a\in \R$, the first passage set above $-a$ is defined by 
\[\tilde{\A}^{u}_{-a}:=\{x\in \tilde\LG\,|\,\exists \text{ a continous path } \gamma \text{ from } x \text{ to }\partial\LG \text{ in }\tilde\LG \text{ such that } \Gamma^{\tilde\LG}+u\ge -a \text{ along }\gamma\}.\]
Note that, conditional on $\tilde{\A}^{u}_{-a}$, the closure of $\tilde\LG\setminus\tilde{\A}^{u}_{-a}$ is also a metric graph with length measure inherited from $\tilde\LG$. According to~\cite[Proposition 2.1]{AruLupuSepulvedaFPSGFFCVGISO}, metric graph GFF satisfies the following space Markov property: 
\[\Gamma^{\tilde\LG}=\Gamma^{\tilde\LG}_{\tilde{\A}^{u}_{-a}}+\Gamma^{{\tilde\LG},\tilde{\A}^{u}_{-a}},\]
where $\Gamma^{{\tilde\LG},\tilde{\A}^{u}_{-a}}$ is the metric graph GFF with zero-boundary condition on the closure of $\tilde\LG\setminus\tilde{\A}^{u}_{-a}$ conditional on $\tilde{\A}^{u}_{-a}$, and $\Gamma^{\tilde\LG}_{\tilde{\A}^{u}_{-a}}$ is defined as follows: it is $\Gamma^{\tilde\LG}$ on $\tilde{\A}^{u}_{-a}$ and it is the harmonic function with boundary value given by $\Gamma^{\tilde\LG}$ on $\tilde\LG\setminus\tilde{\A}^{u}_{-a}$.

We also need the following description of first passage set by clusters of loops and excursions. The Brownian loop measure and Brownian excursion measure are conformally invariant measures on Brownian paths in the plane. In this article, we do not need the precise definition of these measures, so we content ourselves with referring their definition to~\cite[Section~2.2]{AruLupuSepulvedaFPSGFFCVGISO}. 
We denote by $\mu^{\tilde\LG}_{\text{loop}}$ the Brownian loop measure on $\tilde{\LG}$. 
Suppose $u$ is non-negative, and we denote by $\mu_{\text{exc}}^{\tilde\LG,u}$ the Brownian excursion measure on $\tilde{\LG}$ with boundary data $u$. 
We sample Poisson point process with intensity measure $\frac{1}{2}\mu^{\tilde\LG}_{\text{loop}}$ , and denote it by $\LL^{\tilde\LG}_{1/2}$. We sample an independent Poisson point process with intensity measure $\mu_{\text{exc}}^{\tilde\LG,u}$ and denote it by $\Xi^{\tilde\LG}_{u}$. We denote by $\tilde\LA(\LL^{\tilde\LG}_{1/2}, \Xi^{\tilde\LG}_{u})$ the closure of union of clusters formed by loops and excursions that contain at least one excursion connected to $\partial\LG$. As shown in~\cite[Proposition~2.5]{AruLupuSepulvedaFPSGFFCVGISO}, the set $\tilde\LA(\LL^{\tilde\LG}_{1/2}, \Xi^{\tilde\LG}_{u})$ has the same law as the first passage set $\tilde{\A}^{u}_{0}$.

\medbreak

Next, we introduce the first passage set for continuum GFF. To this end, we first introduce local set. Suppose $\Omega\subsetneq\C$ is a simply connected domain and let $\Gamma$ be a continuum GFF on $\Omega$ with zero-boundary condition. We call a random closed set $A\subset \overline\Omega$ is a local set of $\Gamma$, if $\Gamma=\Gamma_{A}+\Gamma^{A}$, where $\Gamma_{A}$ and $\Gamma^{A}$ are two random distributions such that $\Gamma_{A}$ is harmonic in $\Omega\setminus A$ and, conditional on $(A, \Gamma_{A})$, the function $\Gamma^{A}$ is the GFF with zero-boundary condition in $\Omega\setminus A$. Suppose $h_{A}$ is defined as follows: it is $\Gamma_{A}$ on $\Omega\setminus A$ and it is $0$ on $A$.  Then we have the following description of the first passage set.
\begin{theorem}\label{thm::description} 
Suppose $\Omega\subsetneq\C$ is a simply connected domain and let $\Gamma$ be a continuum GFF on $\Omega$ with zero-boundary condition. Suppose $u$ is a bounded harmonic function with piecewise constant boundary data.\footnote{Throughout the article, by piecewise constant boundary data, we mean that the boundary data is piecewise constant and it changes only finitely many times.}
The first passage set $\A_{-a}^{u}$ is the local set of $\Gamma$ containing $\partial\Omega$ with the following two properties:
\begin{itemize}
\item 
The function $h_{\A^{u}_{-a}}+u$ is harmonic in $\Omega\setminus \A^{u}_{-a}$ such that it equals $-a$ on $\partial\A^{u}_{-a}\setminus \partial\Omega$ and it equals $u$ on $\partial(\Omega\setminus\A^{u}_{-a})\cap\partial\Omega$. Moreover, $h_{\A^{u}_{-a}}+u\le -a$.
\item
We have $\Gamma_{\A^{u}_{-a}}-h_{\A^{u}_{-a}}\ge0$. I.e. for any positive smooth function $f$ with compact support, we have $(\Gamma_{\A^{u}_{-a}}-h_{\A^{u}_{-a}},f)\ge0$.
\end{itemize}
For all $a\ge0$, the first passage set $\A_{-a}^{u}$ exists. Moreover, the set $\A_{-a}^{u}$ is the unique local set which satisfies the above two properties and is measurable with respect to $\Gamma$. 
\end{theorem}
\begin{proof}
See~\cite[Theorem~3.5]{AruLupuSepulvedaFPSGFFCVGISO}.
\end{proof}

\medbreak
Now, we are ready to state the convergence of the first passage set of the metric graph GFF to the first passage set of the continuum GFF.
Fix a bounded simply connected  domain $\Omega$ such that $\Omega\subset [-C, C]^2$ for some $C>0$. Suppose $\{\Omega^{\delta}\}_{\delta>0}$ is a sequence of simply connected domains such that $\Omega^{\delta}\subset [-C, C]^2$ for all $\delta>0$. Suppose $\Omega^{\delta}$ converges to $\Omega$ as $\delta\to 0$ in the following sense :
\[[-C, C]^{2}\setminus\Omega^{\delta} \text{ converges to } [-C, C]^{2}\setminus\Omega \text{ in Hausdorff metric}.\]
We equip the edge $e$ of the graph $\delta\Z^{2}$ with $C(e)=1$ and denote by $\tilde{\delta\Z^{2}}$ the corresponding metric graph. We define $\tilde\Omega^{\delta}$ to be the closure of $\Omega^{\delta}\cap\tilde{\delta\Z^{2}}$. It is also a metric graph with metric inherited from $\tilde{\delta\Z^{2}}$. We define its boundary by $\partial\tilde\Omega^{\delta}:=\tilde\Omega^{\delta}\cap\partial\Omega^{\delta}$. We have the following setup for the convergence. 
\begin{itemize}
\item Suppose $\Gamma$ is the continuum GFF on $\Omega$ and $\tilde\Gamma^{\delta}$ is the metric grpah GFF on $\tilde\Omega^{\delta}$ with zero-boundary condition. We extend $\Gamma$ to $(-C, C)^2$ such that it is zero outside $\Omega$, and we still denote the extension by $\Gamma$. 
We define $\hat{\Gamma}^{\delta}$ on $(-C, C)^2$ as follows: it equals $\tilde\Gamma^{\delta}$ on $\tilde\Omega^{\delta}$ and it is harmonic in $(-C, C)^2\setminus\tilde{\Omega}^{\delta}$ which equals zero along $\partial (-C, C)^2$. 
\item Suppose $u$ is a harmonic function on $\Omega$ with piecewise constant boundary data and $u^{\delta}$ is a harmonic function on $\tilde{\Omega}^\delta$ for every $\delta>0$ such that $u^{\delta}$ converges to $u$ uniformly as $\delta\to 0$.
\item For $a\in \R$, suppose $\A^{u}_{-a}$ is the first passage set of $\Gamma$ on $\Omega$ and $\tilde{\A}^{u^{\delta}}_{-a}$ is the first passage set of $\tilde\Gamma^{\delta}$ on $\tilde\Omega^{\delta}$. We extend $\Gamma_{\A^u_{-a}}$ to $(-C, C)^2$ such that it is zero outside $\Omega$, and we still denote the extension by $\Gamma_{\A^u_{-a}}$. We define $\hat{\Gamma}^{\delta}_{\tilde{\A}_{-a}^{u^{\delta}}}$ on $(-C, C)^2$ as follows: it equals $\tilde{\Gamma}^{\delta}_{\tilde\A^{u^{\delta}}_{-a}}$ on $\tilde\Omega^{\delta}$ and it is harmonic in $(-C, C)^2\setminus\tilde{\Omega}^{\delta}$ which equals zero along $\partial (-C, C)^2$. 
\end{itemize}

\begin{proposition}\label{prop::mgffconvergence}
We have the following convergence in law : 
\[\left(\hat\Gamma^{\delta},\hat{\Gamma}^{\delta}_{\tilde\A^{u^{\delta}}_{-a}}, \tilde\A^{u^{\delta}}_{-a}\right)\to \left(\Gamma, \Gamma_{\A^{u}_{-a}}, \overline{\A^{u}_{-a}\cap\Omega}\right), \quad \text{as }\delta\to 0. \]
Furthermore, if we couple $\{\hat\Gamma^{\delta}\}_{\delta>0}$ and $\Gamma$ together such that $\hat\Gamma^{\delta}\to\Gamma$ in probability as distributions on $[-C, C]^{2}$, then $(\hat\Gamma^{\delta}, \tilde\A^{u^{\delta}}_{-a})\to (\Gamma,\overline{\A^{u}_{-a}\cap\Omega})$ in probability. 
\end{proposition}
\begin{proof}
See~\cite[Proposition~4.7 and Lemma~4.9]{AruLupuSepulvedaFPSGFFCVGISO}. 
\end{proof}

\subsection{Convergence of the connection probability}
\label{subsec::mgffcvg}
Fix a bounded polygon $(\Omega; y_1, \ldots, y_{2N})$ and suppose $\left(\Omega^{\delta};y_{1}^{\delta},\ldots,y_{2N}^{\delta}\right)$ converges to $\left(\Omega;y_{1},\ldots,y_{2N}\right)$ as $\delta\to 0$ in the sense of~\eqref{eqn::topology}. We have the following setup. 
\begin{itemize}
\item Suppose $\tilde{\Gamma}^{\delta}$ is the zero-boundary metric graph GFF on $\tilde{\Omega}^{\delta}$ and let $u^{\delta}$ be the harmonic function with boundary data~\eqref{eqn::mgff_boundaryconditions}. Suppose $\Gamma$ is zero-boundary GFF on $\Omega$ and let $u$ be the harmonic function with the same boundary data. 
\item We call the first passage set above $0$ of $\tilde\Gamma^{\delta}+u^{\delta}$ the positive first passage set and we denote it by $\tilde{\A}^{\delta}$. We call the first passage set above $0$ of $-(\tilde\Gamma^{\delta}+u^{\delta})$ the negative first passage set and we denote it by $\tilde{\AB}^{\delta}$. Similarly, we can also define the positive first passage set and the negative first passage set for the continuum GFF in $\Omega$, and we denote them by $\A$ and $\AB$ respectively. 
\end{itemize}
Note that the frontier of these first passage sets is a collection of $2N$ curves connecting the $2N$ boundary points so that their end points form a planar $2N$-link pattern of $2N$ points with index valences $\varsigma=(2, \ldots, 2)$, see Figure~\ref{fig::mgfflinkpatterns}. We denote the link pattern by $\LA^{\delta}$ for metric graph GFF and by $\LA$ for cotinuum GFF. The goal of this section is to prove the following convergence. 
\begin{proposition}\label{prop::convergence}
Fix $N\ge 1$ and $\varsigma=(2, \ldots, 2)$ of length $2N$. 
For all $\hat{\alpha}\in\LP_{\varsigma}$, we have 
\[\lim_{\delta\to 0}\PP[\LA^{\delta}=\hat{\alpha}]=\PP[\LA=\hat{\alpha}].\]
\end{proposition}
To prove Proposition~\ref{prop::convergence}, we will give an explicit construction of $\A$ and $\AB$ in Lemma~\ref{lem::construction}. This construction indicates that the frontier of $\A$ and of $\AB$ forms a planar link pattern in $\LP_{\varsigma}$. Then, we prove Lemma~\ref{lem::convergence} which indicates that for any subsequence $\delta_n\to 0$ as $n\to\infty$, there exists a coupling such that the frontier of $\tilde\A^{\delta_{n}}$ and of $\tilde\AB^{\delta_{n}}$ converges to the frontier of $\A$ and of $\AB$ almost surely in Hausdorff metric. This indicates the proposition.

\begin{lemma}\label{lem::construction}
The frontier of $\A$ is the union of level lines of the continuum GFF $\Gamma+u$ starting from $y_{2j-1}$ with height $\lambda$ for $1\le j\le N$, the frontier of $\AB$ is the union of level lines of $\Gamma+u$ starting from $y_{2j-1}$ with height $-\lambda$ for $1\le j\le N$. 
\end{lemma}
\begin{proof}
We will prove the conclusion for $\A$, and the proof for $\AB$ is similar. By the conformal invariance of GFF, we may assume $\Omega=\HH$ and $y_1<\cdots<y_{2N}$. 
We will argue that the first passage set $\A$ can be constructed as follows. Let $\eta_j$ be the level line of $\Gamma+u$ starting from $y_{2j-1}$ with height $\lambda$ for $1\le j\le N$. 
Suppose $S_{1},\ldots, S_{r}$ are the different connected components of $\HH\setminus\ (\cup_{1\le j\le N}\eta_{j})$ which have $(y_{2j-1},y_{2j})$ on their boundary for some $1\le j\le N$. Note that $(\Gamma+u)|_{S_i}$ has boundary data $2\lambda$ along $\partial S_i$. Conditional on $\cup_{1\le j\le N}\eta_{j}$, we sample the first passage set above zero of $(\Gamma+u)|_{S_{i}}$ in each $S_{i}$, and we denote it by $\A_{0,i}^{2\lambda}$ for $1\le i\le r$. We will show that the union $A:=(\cup_{1\le j\le N}\eta_{j})\cup(\cup_{1\le i\le r}\A_{0,i}^{2\lambda})\cup \partial\HH$ has the same law as $\A$. 

First, we prove that $A$ is a local set. By construction, 
\begin{align*}
\Gamma&= \Gamma_{\cup_{1\le j\le N}\eta_{j}}+\Gamma^{\cup_{1\le j\le N}\eta_{j}}\\
&= \Gamma_{\cup_{1\le j\le N}\eta_{j}}+\sum_{i=1}^{r}\Gamma^{\cup_{1\le j\le N}\eta_{j}}|_{S_{i}}\\
&= \Gamma_{\cup_{1\le j\le N}\eta_{j}}+\sum_{i=1}^{r}\left(\Gamma^{\cup_{1\le j\le N}\eta_{j}}|_{S_{i}}\right)_{\A_{0,i}^{2\lambda}}+\sum_{i=1}^{r}\left(\Gamma^{\cup_{1\le j\le N}\eta_{j}}|_{S_{i}}\right)^{\A_{0,i}^{2\lambda}}.
\end{align*}
Note that $\Gamma_{A}:=\Gamma_{\cup_{1\le j\le N}\eta_{j}}+\sum_{i=1}^{r}\left(\Gamma^{\cup_{1\le j\le N}\eta_{j}}|_{S_{i}}\right)_{\A_{0,i}^{2\lambda}}$ is harmonic in $\HH\setminus A$. Conditional on $(A, \Gamma_{A})$, the function $\Gamma^{A}=\sum_{i=1}^{r}\left(\Gamma^{\cup_{1\le j\le N}\eta_{j}}|_{S_{i}}\right)^{\A_{0,i}^{2\lambda}}$ is the continuum GFF with zero-boundary condition in $\HH\setminus A$. This implies that $A$ is a local set.

Next, we check the two properties in Theorem~\ref{thm::description}. The first one is obvious by construction. For the second one, suppose $f$ is a positive smooth function with compact support in $\HH$, it suffices to prove
\begin{align}\label{eqn::thin}
(\Gamma_{\cup_{1\le j\le N}\eta_{j}},f)=(h_{\cup_{1\le j\le N}\eta_{j}}, f) 
\end{align}
and
\begin{align}\label{eqn::passage}
\left(\left(\Gamma^{\cup_{1\le j\le N}\eta_{j}}|_{S_{i}}\right)_{\A_{0,i}^{2\lambda}}+2\lambda \one_{S_{i}},f\right)\ge 0.
\end{align}
Eq.~\eqref{eqn::thin} is a consequence of properties of level lines of GFF. 
%{\color{blue}For.~\eqref{eqn::thin}, conditional on $\cup_{k\neq i}\eta_{k}$, we denote by $D$ the connected component of $\HH\setminus(\cup_{k\neq i}\eta_{k})$ which contains $y_{2i-1}$ on its boundary. Suppose $\eta_{i}$ ends at $y_{2j}$ for some $j$. Then $\eta_{i}$ has the law of a $\SLE_{4}(2)$ curve with force point the prime end connecting to the left of $y_{2i-1}$. Thus combining the conformal invariance of continuum $\GFF$ and Lemma~\ref{lem::thinlocalset}, we get~\eqref{eqn::thin}.}
For~\eqref{eqn::passage}, consider the metric graph $\tilde S_{i}^{\delta}=S_{i}\cap\delta\tilde\Z^{2}$, we denote  by $\tilde\Gamma^{\delta}|_{\tilde S_{i}^{\delta}}$ the metric graph GFF with zero-boundary condition on $\tilde S_{i}^{\delta}$. Then by Proposition~\ref{prop::mgffconvergence}, we can couple $\left\{\left(\hat\Gamma^{\delta}|_{\tilde S_{i}^{\delta}}\right)_{\tilde\A_{0}^{2\lambda}}\right\}_{\delta>0}$ and $\left(\Gamma^{\cup_{1\le j\le N}\eta_{j}}|_{S_{i}}\right)_{\A_{0,i}^{2\lambda}}$ together such that 
\begin{align*}
\left(\left(\Gamma^{\cup_{1\le j\le N}\eta_{j}}|_{S_{i}}\right)_{\A_{0,i}^{2\lambda}}+2\lambda \one_{S_{i}},f\right)&=\left(\left(\Gamma^{\cup_{1\le j\le N}\eta_{j}}|_{S_{i}}\right)_{\A_{0,i}^{2\lambda}}+2\lambda \one_{S_{i}},f\one_{S_{i}}\right)\\&=\lim_{\delta\to 0}\left(\left(\hat{\Gamma}^{\delta}|_{\tilde S_{i}^{\delta}}\right)_{\tilde\A_{0}^{2\lambda}}+2\lambda\one_{S_{i}},f\one_{S_{i}}\right)\ge 0. 
\end{align*}
This gives~\eqref{eqn::passage}. Combining with~\eqref{eqn::thin} and Theorem~\ref{thm::description}, we see that $A$ has the same law as $\A$, and this completes the proof. 
\end{proof}

\begin{lemma}\label{lem::convergence} Suppose $\delta_{n}\to 0$ as $n\to\infty$. 
\begin{itemize}
\item Suppose $\left(\A_{1},\ldots,\A_{r}\right)$ are  different connected components of $\overline{\A\cap\Omega}$ and $\left(y_{i,1}\ldots,y_{i,k_{i}}\right)$ are the marked points on the boundary of $\A_{i}$ for each $1\le i\le r$; and suppose $\left(\AB_{1},\ldots,\AB_{s}\right)$ are  different connected components of $\overline{\AB\cap\Omega}$ and $\left(y_{j,1},\ldots,y_{j,l_{j}}\right)$ are the marked points on the boundary of $\AB_{j}$ for each $1\le j\le s$. 
\item Suppose $\left(\tilde\A^{\delta_{n}}_{1},\ldots,\tilde\A^{\delta_{n}}_{r_{n}}\right)$ are the connected components of $\tilde\A^{\delta_{n}}\cup\left(\cup_{1\le l\le N}\left(y^{\delta_{n}}_{2l-1},y^{\delta_{n}}_{2l}\right)\right)$; and suppose $\left(\tilde\AB^{\delta_{n}}_{1},\ldots,\tilde\AB^{\delta_{n}}_{s_{n}}\right)$ are the connected components of $\tilde\AB^{\delta_{n}}\cup\left(\cup_{1\le l\le N}\left(y^{\delta_{n}}_{2l},y^{\delta_{n}}_{2l+1}\right)\right)$.
\end{itemize}
Then there exists a coupling of $\{(\tilde\A^{\delta_{n}}, \tilde\AB^{\delta_{n}})\}_{n\ge 1}$ and $(\A,\AB)$ such that the following holds almost surely. 
\begin{itemize}
\item For $n$ large enough,  we have $r_{n}=r$ and $s_{n}=s$.
\item Moreover, we can reorder $\left(\tilde\A^{\delta_{n}}_{1},\ldots,\tilde\A^{\delta_{n}}_{r}\right)$ and we still denote them by $\left(\tilde\A^{\delta_{n}}_{1},\ldots,\tilde\A^{\delta_{n}}_{r}\right)$, such that $\left(y^{\delta_{n}}_{i,1},\ldots,y^{\delta_{n}}_{i,k_{i}}\right)$ are the marked points on the boundary of $\tilde\A^{\delta_{n}}_{i}$ for $1\le i\le r$. Similarly, we can reorder $\left(\tilde\AB^{\delta_{n}}_{1},\ldots,\tilde\AB^{\delta_{n}}_{s}\right)$ and we still denote them by $\left(\tilde\AB^{\delta_{n}}_{1},\ldots,\tilde\AB^{\delta_{n}}_{s}\right)$ such that $\left(y^{\delta_{n}}_{j,1},\ldots,y^{\delta_{n}}_{j,l_{j}}\right)$ are the marked points on the boundary of $\tilde\AB^{\delta_{n}}_{j}$ for $1\le j\le s$.
\item Furthermore, we have that $\tilde\A^{\delta_{n}}_{i}$ converges to $\A_{i}$ for each $1\le i\le r$ and $\tilde\AB^{\delta_{n}}_{j}$ converges to $\AB_{j}$ for each $1\le j\le s$ in Hausdorff metric.
\end{itemize}
  
\end{lemma} 
\begin{proof}
We denote by $\tilde R^{\delta_{n}}_{i}$ the connected component of $\tilde\A^{\delta_{n}}\cup\left(\cup_{1\le l\le N}\left(y^{\delta_{n}}_{2l-1},y^{\delta_{n}}_{2l}\right)\right)$ which contains $\left(y^{\delta_{n}}_{2i-1},y^{\delta_{n}}_{2i}\right)$ on its boundary for $1\le i\le N$; and we denote by $\tilde S^{\delta_{n}}_{i}$ the connected component of $\tilde\AB^{\delta_{n}}\cup\left(\cup_{1\le l\le N}\left(y^{\delta_{n}}_{2l},y^{\delta_{n}}_{2l+1}\right)\right)$ which contains $(y^{\delta_{n}}_{2i},y^{\delta_{n}}_{2i+1})$ on its boundary for $1\le i\le N$. 
Note that the sequence
\[\left\{\left(\tilde R^{\delta_{n}}_{1},\ldots,\tilde R^{\delta_{n}}_{N},\tilde\A^{\delta_{n}},\tilde S^{\delta_{n}}_{1},\ldots,\tilde S^{\delta_{n}}_{N}, \tilde\AB^{\delta_{n}}\right)\right\}_{n\ge 1}\] 
is tight. Thus, it suffices to prove that Lemma~\ref{lem::convergence} holds for any convergent subsequence. Given any convergent subsequence, we still denote it by $\left\{\left(\tilde R^{\delta_{n}}_{1},\ldots,\tilde R^{\delta_{n}}_{N},\tilde\A^{\delta_{n}},\tilde S^{\delta_{n}}_{1},\ldots,\tilde S^{\delta_{n}}_{N}, \tilde\AB^{\delta_{n}}\right)\right\}_{n\ge 1}$. By Skorokhod representation theorem, we can couple them on the same probability space such that there is almost sure convergence. We denote the probability measure of this coupling by $\PP$, and we denote its limit by 
$(R_{1},\ldots,R_{N},R,S_{1},\ldots,S_{N},S)$.

By Proposition~\ref{prop::mgffconvergence}, we know that $R=\overline{\A\cap\Omega}$ and $S=\overline{\AB\cap\Omega}$. We will prove that $R_{i}$ is the connected component of $\A$ which contains $(y_{2i-1},y_{2i})$ on its boundary and $S_{i}$ is the connected component of $\AB$ which contains $(y_{2i},y_{2i+1})$ on its boundary. Moreover, we will show that Lemma~\ref{lem::convergence} holds in this coupling. We only give proof for the positive first passage set, as the proof for the negative first passage set is similar. The proof is divided into two steps. First, suppose $\left(y_{i_{1}},\ldots,y_{i_k}\right)$ are the marked points on the boundary of $R_{i}$. We will prove that $\left(y^{\delta_{n}}_{i_{1}},\ldots,y^{\delta_{n}}_{i_{k}}\right)$ are the marked points on the boundary of $\tilde R^{\delta_{n}}_{i}$ for $n$ large enough. Then we will prove that $R_{i}$ is the connected component of $R$ which contains $(y_{2i-1},y_{2i})$ on its boundary for each $1\le i\le N$.

For the first step, it is clear that if $y_{j}\notin R_{i}$ for some $1\le j\le 2N$, we have $y_{j}^{\delta_{n}}\notin \tilde R^{\delta_{n}}_{i}$ for $n$ large enough by the almost sure convergence. We define the event $F_{j}:=\left\{y_{j}\in R_{i}, \text{ but } y^{\delta_{n}}_{j}\notin \tilde R^{\delta_{n}}_{i}\text{ except for finite } n\right\}$ 
for $1\le j\le 2N$. It suffices to prove $\PP[F_{j}]=0$.
By Lemma~\ref{lem::construction}, we have $S\cap (y_{2i-1},y_{2i})=\emptyset$ for $1\le i\le N$. We denote by $D_{i}$ the connected component of $\Omega\setminus S$ which contains $(y_{2i-1},y_{2i})$ and we denote by $D^{\delta_{n}}_{i}$ the connected component of $\Omega^{\delta_{n}}\setminus \tilde\AB^{\delta_{n}}$ which contains $\left(y^{\delta_{n}}_{2i-1},y^{\delta_{n}}_{2i}\right)$. By Carath\'{e}odory kernel theorem, the domain $D^{\delta_{n}}_{i}$ converges to $D_{i}$ in Carath\'{e}odory topology as $n\to\infty$. Note that $\tilde \A^{\delta_{n}}\cap \tilde D^{\delta_{n}}_{i}$ is the first passage set $\tilde\A_{0}^{v^{\delta_{n}}}$ of the metric graph GFF on $\tilde D^{\delta_{n}}_{i}$ with boundary data given by $v^{\delta_{n}}$ which is defined as follows: $v^{\delta_{n}}$ equals $0$ on $\tilde \AB^{\delta_{n}} \cap\partial\tilde D^{\delta_{n}}_{i}$ and $v^{\delta_{n}}$ equals $2\lambda$ on $\partial\tilde\Omega^{\delta_{n}}\cap\partial\tilde D^{\delta_{n}}_{i}$. We may assume $j$ is odd. On the event $F_{j}$, we have $(y_{j},y_{j+1})\subset \partial D_{i}$. Thus for $n$ large enough, we have $(y^{\delta_{n}}_{j},y^{\delta_{n}}_{j+1})\subset \partial D^{\delta_{n}}_{i}$. In such case, we define the harmonic function $v_{1}^{\delta_{n}}$ on $\tilde D_{i}^{\delta_{n}}$ as follows: it equals $0$ on $\left(y^{\delta_{n}}_{j},y^{\delta_{n}}_{j+1}\right)$ and it equals $v^{\delta_{n}}$ on $\partial \tilde D_{i}^{\delta_{n}}\setminus\left(y^{\delta_{n}}_{j},y^{\delta_{n}}_{j+1}\right)$. Then, in the construction of $\tilde\A_{0}^{v^{\delta_{n}}}$ by loops and excursions, we can divide the excursions into two independent parts: the excursions connecting to $\left(y^{\delta_{n}}_{j},y^{\delta_{n}}_{j+1}\right)$ and the excursions which do not intersect $\left(y^{\delta_{n}}_{j},y^{\delta_{n}}_{j+1}\right)$.  Note that the excursions connecting to $\left(y^{\delta_{n}}_{j},y^{\delta_{n}}_{j+1}\right)$ correspond to the Poisson point process with intensity measure $\mu_{\text{exc}}^{\tilde D_{i}^{\delta_{n}},v^{\delta_{n}}}-\mu_{\text{exc}}^{\tilde D_{i}^{\delta_{n}},v_{1}^{\delta_{n}}}$ and that the excursions which do not intersect $\left(y^{\delta_{n}}_{j},y^{\delta_{n}}_{j+1}\right)$ correspond to the Poisson point process with intensity measure $\mu_{\text{exc}}^{\tilde D_{i}^{\delta_{n}},v_{1}^{\delta_{n}}}$. Thus, we have $\tilde R^{\delta_{n}}_{i}\subset\tilde\LA\left(\LL^{\tilde D_{i}^{\delta_{n}}}_{1/2}, \Xi^{\tilde D_{i}^{\delta_{n}}}_{v_{1}^{\delta_{n}}}\right)$ if $y^{\delta_{n}}_{j}\notin \tilde R^{\delta_{n}}_{i}$. Note that $\tilde\A_{0}^{v_{1}^{\delta_{n}}}$ has the same law as $\tilde\LA\left(\LL^{\tilde D_{i}^{\delta_{n}}}_{1/2}, \Xi^{\tilde D_{i}^{\delta_{n}}}_{v_{1}^{\delta_{n}}}\right)$. According to~\cite[Corollary~4.12]{AruLupuSepulvedaFPSGFFCVGISO}, the limit of $\overline{\tilde\A_{0}^{v_{1}^{\delta_{n}}}\cap D_{i}}$ does not intersect $(y_{j},y_{j+1})$ almost surely. This implies $\PP[F_{j}]=0$.

For the second step, we define the event $F_{i,k}:=\{R_{i}\neq R_{k},\text{ but } R_{i}\cap R_{k}\neq\emptyset\}$ for $1\le i< k\le N$. It suffices to prove that $\PP[F_{i,k}]=0$. Note that on the event $F_{i,k}$, we have $(y_{2k-1},y_{2k})\subset D_{i,k}$. This implies $(y^{\delta_{n}}_{2k-1},y^{\delta_{n}}_{2k})\subset D^{\delta_{n}}_{i,k}$ for $n$ large enough. Moreover, we have $\tilde R^{\delta_{n}}_{i}\cap \tilde R^{\delta_{n}}_{k}=\emptyset$ for $n$ large enough. This implies $\tilde R^{\delta_{n}}_{i}\cap (y^{\delta_{n}}_{2k-1},y^{\delta_{n}}_{2k})=\emptyset$. We denote by $D_{i,k}$ the connected component of $D_{i}\setminus R_{i}$ with $(y_{2k-1},y_{2k})$ on its boundary and we denote by $D_{i,k}^{\delta_{n}}$ the connected component of $D_{i}^{\delta_{n}}\setminus \tilde R_{i}^{\delta_{n}}$ with $\left(y_{2k-1}^{\delta_{n}},y_{2k}^{\delta_{n}}\right)$ on its boundary. Then by Carath\'{e}odory kernel theorem, the domain $D_{i,k}^{\delta_{n}}$ converges to $D_{i,k}$ as $n\to\infty$ in Carath\'{e}odory topology. Note that $\tilde\A^{\delta_{n}}\cap \tilde D_{i,k}^{\delta_{n}}$ is the first passage set of the metric graph GFF with boundary data $w^{\delta_{n}}$ on $\tilde D_{i,k}^{\delta_{n}}$, where $w^{\delta_{n}}$ is defined as follows: $w^{\delta_{n}}$ equals $0$ on $\left(\tilde \AB^{\delta_{n}}\cup R_{i}^{\delta_{n}}\right)\cap\partial\tilde D_{i,k}^{\delta_{n}}$ and $w^{\delta_{n}}$ equals $2\lambda$ on $\partial\tilde\Omega^{\delta_{n}}\cap\ \partial\tilde D^{\delta_{n}}_{i,k}$. According to~\cite[Corollary~4.12]{AruLupuSepulvedaFPSGFFCVGISO}, the limit of $\overline{\tilde\A_{0}^{w^{\delta_{n}}}\cap D_{i,k}}$ does not intersect $(y_{2i-2},y_{2i-1})\cup \partial R_{i}\cup(y_{2i},y_{2i+1})$. This implies $\PP[F_{i,k}]=0$. It completes the proof.
\end{proof}

\subsection{Asymptotics of partition functions and proof of Proposition~\ref{prop::mGFFpdecov}}
\label{subsec::mgffasy}
The goal of this subsection is the following asymptotics of pure partition functions. The purpose of this proposition will be clear in the proof of Theorem~\ref{thm::main}. This subsection is independent of the rest of Section~\ref{sec::mgfffps}, and we suggest readers to first read Section~\ref{subsec::mgfffinal} and then come back to this subsection. 

\begin{proposition}\label{prop::purepartitionfusionall}
Fix $\kappa=4$. Fix $N\ge 1$ and the index valences $\varsigma=(2, \ldots, 2)$ of length $2N$. 
For each $\hat{\alpha}\in\LP_{\varsigma}$, let $\alpha:=\tau(\hat{\alpha})\in\Pair_{2N}$ be the associated planar pair partition as defined in Section~\ref{subsec::linkpatterntopairpartition}, and $\PartF_{\alpha}$ be the pure partition function associated to $\alpha$. Then, the following limit exists: for $y_1<\cdots<y_{2N}$ and $x_1<\cdots<x_{4N}$, 
\begin{equation}\label{eqn::purepartitionfusionall}
\PartF_{\hat{\alpha}}(y_1, \ldots, y_{2N}):=\lim_{\substack{x_{2j-1}, x_{2j}\to y_j, \\\small{\forall 1\le j\le 2N}}}\frac{\PartF_{\alpha}(x_1, \ldots, x_{4N})}{\prod_{j=1}^{2N}(x_{2j}-x_{2j-1})^{1/2}}. 
\end{equation}
\end{proposition}

We will show Proposition~\ref{prop::purepartitionfusionall} by the explicit expression for $\PartF_{\alpha}$ from~\eqref{eqn::purepartitionvsconformalblock}: 
\[\PartF_{\alpha}(x_1, \ldots, x_{4N})=\sum_{\beta\in\DP_{2N}}\LM_{\alpha, \beta}^{-1}\LU_{\beta}(x_1, \ldots, x_{4N}). \]
For $\beta\in\DP_{2N}$ such that $\times_{2j-1}\in\beta$ for all $1\le j\le 2N$, it is easy to see that $\LU_{\beta}$ admits a limit when normalized by $\prod_{j}(x_{2j}-x_{2j-1})^{1/2}$, see Lemma~\ref{lem::conformalblockfusionall}. However, for other $\beta$, the conformal block $\LU_{\beta}$ explodes when normalized by $\prod_{j}(x_{2j}-x_{2j-1})^{1/2}$. In order to derive the existence of the limit, we need to group distinct $\beta$'s properly so that the explosion cancels. 
The proof of Proposition~\ref{prop::purepartitionfusionall} involves heavy notation which we find unavoidable. We suggest readers to first read the proof of Corollary~\ref{cor::crossingproba_mgff} where we give the proof for Proposition~\ref{prop::purepartitionfusionall} when $N=2$. 

\smallbreak
\begin{lemma}\label{lem::conformalblockfusionall}
Fix $\kappa=4$. Fix $N\ge 1$. Given a Dyck path $\beta\in\DP_{2N}$ of length $4N$ such that $\times_{2j-1}\in\beta$ for all $1\le j\le 2N$, define $(\beta)_2\in\DP_N$ by 
\begin{equation*}
(\beta)_2(k)=\frac{1}{2}\beta(2k),\quad 1\le k\le 2N.
\end{equation*}
One may check that this is a well-defined Dyck path of length $2N$.   
Then,  for $y_1<\cdots<y_{2N}$ and $x_1<\cdots<x_{4N}$, we have 
\begin{equation}
\lim_{\substack{x_{2j-1}, x_{2j}\to y_j, \\\small{\forall 1\le j\le 2N}}}\frac{\LU_{\beta}(x_1, \ldots, x_{4N})}{\prod_{j=1}^{2N}(x_{2j}-x_{2j-1})^{1/2}}=\LU_{(\beta)_2}^4(y_1, \ldots, y_{2N}). 
\end{equation}
\end{lemma}
\begin{proof}
By the definition~\eqref{eqn::conformalblock_def}, we have 
\[\frac{\LU_{\beta}(x_1, \ldots, x_{4N})}{\prod_{j=1}^{2N}(x_{2j}-x_{2j-1})^{1/2}}=\prod_{1\le s<t\le 2N}\left((x_{2t}-x_{2s})(x_{2t}-x_{2s-1})(x_{2t-1}-x_{2s})(x_{2t-1}-x_{2s-1})\right)^{\frac{1}{2}\vartheta_{\beta}(2s,2t)}.\]
Thus 
\[\lim_{\substack{x_{2j-1}, x_{2j}\to y_j, \\\small{\forall 1\le j\le 2N}}}\frac{\LU_{\beta}(x_1, \ldots, x_{4N})}{\prod_{j=1}^{2N}(x_{2j}-x_{2j-1})^{1/2}}=\prod_{1\le s<t\le 2N}(y_{t}-y_{s})^{2\vartheta_{\beta}(2t,2s)}=\LU_{(\beta)_2}^4(y_1, \ldots, y_{2N}).\]
\end{proof}

\begin{proof}[Proof of Proposition~\ref{prop::purepartitionfusionall}]
From~\eqref{eqn::purepartitionvsconformalblock}, we have
\[\PartF_{\alpha}(x_1, \ldots, x_{4N})=\sum_{\beta\in\DP_{2N}}\LM_{\alpha, \beta}^{-1}\LU_{\beta}(x_1, \ldots, x_{4N}). \]
For $\beta\in\DP_{2N}$, there exists $J\subset\{1, 2, \ldots, 2N\}$ such that 
\[\lozenge_{2j-1}\in\beta\text{ for all }j\in J, \text{ and }\times_{2j-1}\in\beta\text{ for all }j\in \{1, \ldots, 2N\}\setminus J.\]
Then, from the definition~\eqref{eqn::conformalblock_def}, the following limit exists:  
\[\lim_{\substack{x_{2j-1}, x_{2j}\to y_j, \\\small{\forall j\not\in J}}}\frac{\LU_{\beta}(x_1, \ldots, x_{4N})}{\prod_{j\not\in J}(x_{2j}-x_{2j-1})^{1/2}}. 
\]To obtain the desired limit, we need to group distinct $\beta$'s according to the location of their local extremes. 

\begin{landscape}
Let $J$ be any subset of $\{1, 2, \ldots, 2N\}$, and define 
\[\mathcal{P}_J^{\alpha}=\left\{\beta\in\DP_{2N}: \beta\succeq\alpha, \text{ and }\lozenge_{2j-1}\in\beta\text{ for all }j\in J, \text{ and }\times_{2j-1}\in\beta\text{ for all }j\in \{1, 2, \ldots, 2N\}\setminus J\right\}.\]
It suffices to show that the following limit exists for all possible $J$: 
\begin{equation*}
\lim_{\substack{x_{2j-1}, x_{2j}\to y_j, \\\small{\forall j\in J}}}\frac{\sum_{\beta\in\mathcal{P}_J^{\alpha}}\LM^{-1}_{\alpha,\beta}\LU_{\beta}(x_1, \ldots, x_{4N})}{\prod_{j\in J}(x_{2j}-x_{2j-1})^{1/2}}.
\end{equation*} 
Suppose $n=\#J\ge 1$. 
For some $\beta_0\in\mathcal{P}_J^{\alpha}$ such that $\vee_{2j-1}\in\beta_0$ for all $j\in J$, we define
\[\mathcal{P}_{J}^{\alpha, \beta_0}=\{\beta\in\DP_{2N}: \exists \{i_1, \ldots, i_k\}\subset J \text{ such that }\beta=\beta_0\uparrow\lozenge_{2i_1-1}\cdots\uparrow\lozenge_{2i_k-1}\}.\]
It is clear that $\#\mathcal{P}_{J}^{\alpha,\beta_0}=2^n$. Furthermore, for distinct $\beta_0,\beta_0'\in\mathcal{P}_J^{\alpha}$ such that $\vee_{2j-1}\in\beta_0$ and $\vee_{2j-1}\in\beta_0'$ for all $j\in J$, we see that $\mathcal{P}_{J}^{\alpha, \beta_0}\cap \mathcal{P}_{J}^{\alpha,\beta_0'}=\emptyset$. Thus $\{\mathcal{P}_{J}^{\alpha, \beta_0}: \beta_0\in\mathcal{P}_J^{\alpha} \text{ with } \vee_{2j-1}\in\beta_0 \,\forall\, j\in J\}$ gives a disjoint partition of $\mathcal{P}_J^{\alpha}$. Therefore, it suffices to show that the following limit exists for each such $\beta_0$: 
\begin{equation}\label{eqn::purepartitionfusionallaux1}
\lim_{\substack{x_{2j-1}, x_{2j}\to y_j, \\\forall j\in J}}\frac{\sum_{\beta\in\mathcal{P}_{J}^{\alpha, \beta_0}}\LM^{-1}_{\alpha,\beta}\LU_{\beta}(x_1, \ldots, x_{4N})}{\prod_{j\in J}(x_{2j}-x_{2j-1})^{1/2}}.
\end{equation}

To derive~\eqref{eqn::purepartitionfusionallaux1}, we will show a more general conclusion. Suppose $K\subset J$ and suppose $\gamma_0\in\mathcal{P}_J^{\alpha}$ such that $\vee_{2j-1}\in \gamma_0$ for all $j\in K$. We define
\[\mathcal{P}^{\alpha, \gamma_0}_{J, K}=\{\gamma\in\DP_{2N}: \exists\{i_1, \ldots, i_k\}\subset K \text{ such that }\gamma=\gamma_0\uparrow\lozenge_{2i_1-1}\cdots\uparrow\lozenge_{2i_k-1}\}.\] 
We denote $R_{K}:=\{2j-1,2j:j\in K\}$, and we denote 
\[Z^{\gamma}_{K,j}:=\sum_{l\notin R_{K}}\frac{\vartheta_{\gamma}(l,2j-1)}{x_{l}-y_{j}}.\]
We denote by $\mathfrak{S}_n$ the set of permutations of $\{1,2, \ldots, n\}$. Suppose $K=\{j_1, \ldots, j_n\}$. For any $\gamma_0\in\mathcal{P}_{J}^{\alpha}$ such that $\vee_{2j-1}\in\gamma_0$ for all $j\in K$, we claim that
\begin{align}\label{eqn::purepartitionfusionallaux2}
\lim_{\substack{x_{2j-1}, x_{2j}\to y_j, \\\forall j\in K}}\frac{\sum_{\gamma\in\mathcal{P}_{J, K}^{\alpha, \gamma_0}}\LM^{-1}_{\alpha,\gamma}\LU_{\gamma}(x_1, \ldots, x_{4N})}{\prod_{j\in K}(x_{2j}-x_{2j-1})^{1/2}}
=&\LM^{-1}_{\alpha,\gamma_0}\prod_{\substack{1\le t<s\le 4N\\ t,s\notin R_{K}}}(x_s-x_t)^{\frac{1}{2}\vartheta_{\gamma_0}(t,s)}\times\left(\sum_{m=0}^{\lfloor\frac{n}{2}\rfloor}\sum_{\sigma\in \mathfrak{J}_n^m}
\frac{2^{m}Z^{\gamma_0}_{K,j_{\sigma_{2m+1}}}\cdots Z^{\gamma_0}_{K,j_{\sigma_{n}}}}{(y_{j_{\sigma_{1}}}-y_{j_{\sigma_{2}}})^{2}\times\cdots\times (y_{j_{\sigma_{2m-1}}}-y_{j_{\sigma_{2m}}})^{2}}
\right),
\end{align}
where $\mathfrak{J}_n^m$ is a subset of $\mathfrak{S}_n$: 
\[\mathfrak{J}_n^m=\{\sigma\in \mathfrak{S}_{n}: \\ \sigma_1<\sigma_3<\cdots<\sigma_{2m-1}, \text{ and }\sigma_{2j-1}<\sigma_{2j}\text{ for } j\le m, \text{ and }\sigma_{2m+1}<\cdots<\sigma_{n}\}. \]

Fix $\alpha$ and $J$, we will show~\eqref{eqn::purepartitionfusionallaux2} by induction on $n=\#K$. It is true for $K=\emptyset$ as it is the same as the definition of $\LU_{\gamma_0}$. Suppose~\eqref{eqn::purepartitionfusionallaux2} holds for $\#K=i$. We need to show it for $\#K=i+1$. 
Suppose $K=\{j_{1},\ldots,j_{i+1}\}$. 
We will take the limit in the left hand side of~\eqref{eqn::purepartitionfusionallaux2} in a particular order: we first let $x_{2j-1}, x_{2j}\to y_j$ with $j\in K\setminus\{j_{i+1}\}$ and then let $x_{2j_{i+1}-1}, x_{2j_{i+1}}\to y_{j_{i+1}}$. 
It will be clear from the calculation that the limit in~\eqref{eqn::purepartitionfusionallaux2} for $\#K=i+1$ does not depend on the order of taking limits. 

For any $\gamma_0\in\mathcal{P}_{J}^{\alpha}$ such that $\vee_{2j_{1}-1},\ldots,\vee_{2j_{i+1}-1}\in \gamma_0$, we have the decomposition \[\mathcal{P}_{J, K}^{\alpha, \gamma_0}=\mathcal{P}_{J, K\setminus\{j_{i+1}\}}^{\alpha, \gamma_0}\bigsqcup \mathcal{P}_{J, K\setminus\{j_{i+1}\}}^{\alpha, \gamma_0\uparrow\lozenge_{2j_{i+1}-1}}.\]
Denote by $\gamma_1=\gamma_0\uparrow\lozenge_{2j_{i+1}-1}$ and $K_1=\{j_1, \ldots, j_i\}=K\setminus\{j_{i+1}\}$. Then we have
\begin{align}\label{eqn::purepartitionfusionallinduction}
&\lim_{\substack{x_{2j-1}, x_{2j}\to y_j, \\\forall j\in K}}\frac{\sum_{\gamma\in\mathcal{P}_{J, K}^{\alpha, \gamma_0}}\LM^{-1}_{\alpha,\gamma}\LU_{\gamma}(x_1, \ldots, x_{4N})}{\prod_{j\in K}(x_{2j}-x_{2j-1})^{1/2}}\\
=&\lim_{x_{2j_{i+1}-1},x_{2j_{i+1}}\to y_{j_{{i+1}}}}\frac{1}{(x_{2j_{i+1}}-x_{2j_{i+1}-1})^{1/2}}\notag\times \left(\lim_{\substack{x_{2j-1}, x_{2j}\to y_j\\ \forall j\in K_1}}\frac{\sum_{\gamma\in\mathcal{P}_{J, K_1}^{\alpha, \gamma_0}}\LM^{-1}_{\alpha,\gamma}\LU_{\gamma}(x_1, \ldots, x_{4N})}{\prod_{j\in K_1}(x_{2j}-x_{2j-1})^{1/2}}+\lim_{\substack{x_{2j-1}, x_{2j}\to y_j\\ \forall j\in K_1}}\frac{\sum_{\gamma\in\mathcal{P}_{J, K_1}^{\alpha, \gamma_1}}\LM^{-1}_{\alpha,\gamma}\LU_{\gamma}(x_1, \ldots, x_{4N})}{\prod_{j\in K_1}(x_{2j}-x_{2j-1})^{1/2}}\right).\notag
\end{align}
By the induction hypothesis, we have 
\begin{align*}
\lim_{\substack{x_{2j-1}, x_{2j}\to y_j\\ \forall j\in K_1}}\frac{\sum_{\gamma\in\mathcal{P}_{J, K_1}^{\alpha, \gamma_0}}\LM^{-1}_{\alpha,\gamma}\LU_{\gamma}(x_1, \ldots, x_{4N})}{\prod_{j\in K_1}(x_{2j}-x_{2j-1})^{1/2}}
=&\LM^{-1}_{\alpha, \gamma_0}\prod_{\substack{1\le t<s\le 4N\\ t,s\not\in R_{K_1}}}(x_s-x_t)^{\frac{1}{2}\vartheta_{\gamma_0}(t,s)}\times S_0,\\
\lim_{\substack{x_{2j-1}, x_{2j}\to y_j\\ \forall j\in K_1}}\frac{\sum_{\gamma\in\mathcal{P}_{J, K_1}^{\alpha, \gamma_1}}\LM^{-1}_{\alpha,\gamma}\LU_{\gamma}(x_1, \ldots, x_{4N})}{\prod_{j\in K_1}(x_{2j}-x_{2j-1})^{1/2}}
=&\LM^{-1}_{\alpha,\gamma_1}\prod_{\substack{1\le t<s\le 4N\\ t,s \not\in R_{K_1}}}(x_s-x_t)^{\frac{1}{2}\vartheta_{\gamma_1}(t,s)}\times S_1, 
\end{align*}
where 
\[S_u=\sum_{m=0}^{\lfloor\frac{i}{2}\rfloor}\sum_{\sigma\in \mathfrak{J}_i^m}
\frac{2^{m}Z^{\gamma_u}_{K_1,j_{\sigma_{2m+1}}}\cdots Z^{\gamma_u}_{K_1,j_{\sigma_{i}}}}{(y_{j_{\sigma_{1}}}-y_{j_{\sigma_{2}}})^{2}\times\cdots\times (y_{j_{\sigma_{2m-1}}}-y_{j_{\sigma_{2m}}})^{2}},\quad\text{for }u=0, 1. \]
Comparing the two expressions in the right hand side, we have $\LM^{-1}_{\alpha,\gamma_1}=-\LM^{-1}_{\alpha,\gamma_0}$, and
\[\prod_{\substack{1\le t<s\le 4N\\ t,s\not\in R_{K_1}}}(x_s-x_t)^{\frac{1}{2}\vartheta_{\gamma_0}(t,s)}
=\prod_{\substack{1\le t<s\le 4N\\ t,s\not\in R_K}}(x_s-x_t)^{\frac{1}{2}\vartheta_{\gamma_0}(t,s)}
\times\prod_{\substack{1\le n\le 4N\\ n\not\in R_K}}\left|\frac{x_n-x_{2j_{i+1}-1}}{x_n-x_{2j_{i+1}}}\right|^{\frac{1}{2}\vartheta_{\gamma_0}(n, 2j_{i+1}-1)}\times(x_{2j_{i+1}}-x_{2j_{i+1}-1})^{-\frac{1}{2}}.\]
\[\prod_{\substack{1\le t<s\le 4N\\ t,s \not\in R_{K_1}}}(x_s-x_t)^{\frac{1}{2}\vartheta_{\gamma_1}(t,s)}
=\prod_{\substack{1\le t<s\le 4N\\ t,s\not\in R_K}}(x_s-x_t)^{\frac{1}{2}\vartheta_{\gamma_0}(t,s)}\times \prod_{\substack{1\le n\le 4N\\ n\not\in R_K}}\left|\frac{x_n-x_{2j_{i+1}-1}}{x_n-x_{2j_{i+1}}}\right|^{-\frac{1}{2}\vartheta_{\gamma_0}(n, 2j_{i+1}-1)}\times(x_{2j_{i+1}}-x_{2j_{i+1}-1})^{-\frac{1}{2}}.\]
Plugging into~\eqref{eqn::purepartitionfusionallinduction} and denoting $\delta=x_{2j_{i+1}}-x_{2j_{i+1}-1}$, we have
\begin{align}\label{eqn::purepartitionfusionallinduction2}
&\lim_{\substack{x_{2j-1}, x_{2j}\to y_j, \\\forall j\in K}}\frac{\sum_{\gamma\in\mathcal{P}_{J, K}^{\alpha, \gamma_0}}\LM^{-1}_{\alpha,\gamma}\LU_{\gamma}(x_1, \ldots, x_{4N})}{\prod_{j\in K}(x_{2j}-x_{2j-1})^{1/2}}\\
=&\LM_{\alpha,\gamma_0}^{-1}\prod_{\substack{1\le t<s\le 4N\\ t,s\not\in R_K}}(x_s-x_t)^{\frac{1}{2}\vartheta_{\gamma_0}(t,s)}\times\lim_{\substack{\delta\to 0\\x_{2j_{i+1}-1}\to y_{j_{{i+1}}}}}\frac{1}{\delta}\left(\left(\prod_{\substack{1\le n\le 4N\\ n\not\in R_K}}\left|\frac{x_n-x_{2j_{i+1}-1}}{x_n-x_{2j_{i+1}}}\right|^{\vartheta_{\gamma_0}(n, 2j_{i+1}-1)}-1\right)\times S_0+S_0-S_1\right). \notag
\end{align}
Note that
\begin{align*}
Z^{\gamma_0}_{K_1,j_{\sigma_{2m+1}}}\cdots Z^{\gamma_0}_{K_1,j_{\sigma_{i}}}=&\left(Z^{\gamma_0}_{K,j_{\sigma_{2m+1}}}+\frac{\delta}{\left(x_{2j_{i+1}-1}-y_{j_{\sigma_{2m+1}}}\right)\left(x_{2j_{i+1}}-y_{j_{\sigma_{2m+1}}}\right)}\right)\cdots \left(Z^{\gamma_0}_{K,j_{\sigma_{i}}}+\frac{\delta}{\left(x_{2j_{i+1}-1}-y_{j_{i}}\right)\left(x_{2j_{i+1}}-y_{j_{i}}\right)}\right).
\end{align*}
\begin{align*}
Z^{\gamma_1}_{K_1,j_{\sigma_{2m+1}}}\cdots Z^{\gamma_1}_{K_1,j_{\sigma_{i}}}=&\left(Z^{\gamma_0}_{K,j_{\sigma_{2m+1}}}-\frac{\delta}{\left(x_{2j_{i+1}-1}-y_{j_{\sigma_{2m+1}}}\right)\left(x_{2j_{i+1}}-y_{j_{\sigma_{2m+1}}}\right)}\right)\cdots \left(Z^{\gamma_0}_{K,j_{\sigma_{i}}}-\frac{\delta}{\left(x_{2j_{i+1}-1}-y_{j_{i}}\right)\left(x_{2j_{i+1}}-y_{j_{i}}\right)}\right).
\end{align*}
Plugging into $S_0$ and $S_1$, we have
\begin{align*}
&\lim_{\substack{\delta\to 0\\x_{2j_{i+1}-1}\to y_{j_{{i+1}}}}}\frac{1}{\delta}\left(\prod_{\substack{1\le n\le 4N\\ n\not\in R_K}}\left|\frac{x_n-x_{2j_{i+1}-1}}{x_n-x_{2j_{i+1}}}\right|^{\vartheta_{\gamma_0}(n, 2j_{i+1}-1)}-1\right)\times S_0=Z_{K, j_{i+1}}^{\gamma_0}\left(\sum_{m=0}^{\lfloor\frac{i}{2}\rfloor}\sum_{\sigma\in\mathfrak{J}_i^m}
\frac{2^{m}Z^{\gamma_0}_{K,j_{\sigma_{2m+1}}}\cdots Z^{\gamma_0}_{K,j_{\sigma_{i}}}}{(y_{j_{\sigma_{1}}}-y_{j_{\sigma_{2}}})^{2}\times\cdots\times (y_{j_{\sigma_{2m-1}}}-y_{j_{\sigma_{2m}}})^{2}}
\right)
\end{align*}
\begin{align*}
\lim_{\substack{\delta\to 0\\x_{2j_{i+1}-1}\to y_{j_{{i+1}}}}}\frac{1}{\delta}(S_0-S_1)=\sum_{m=0}^{\lfloor\frac{i}{2}\rfloor}\sum_{\sigma\in\mathfrak{J}_i^m}
\sum_{r=2m+1}^i
\frac{2^{m+1}Z^{\gamma_0}_{K, j_{\sigma_{2m+1}}}\cdots Z^{\gamma_0}_{K, j_{\sigma_{r-1}}}Z^{\gamma_0}_{K, j_{\sigma_{r+1}}}\cdots Z^{\gamma_0}_{K, j_{\sigma_{i}}}}{(y_{j_{\sigma_{1}}}-y_{j_{\sigma_{2}}})^{2}\times\cdots\times (y_{j_{\sigma_{2m-1}}}-y_{j_{\sigma_{2m}}})^{2}\times(y_{j_{i+1}}-y_{j_{\sigma_r}})^2}
\end{align*}
Plugging into~\eqref{eqn::purepartitionfusionallinduction2}, we see that it remains to show 
\begin{align}\label{eqn::purepartitionfusionallaux3}
&\sum_{m=0}^{\lfloor\frac{i+1}{2}\rfloor}\sum_{\tau\in\mathfrak{J}_{i+1}^m}
\frac{2^{m}Z^{\gamma_0}_{K,j_{\tau_{2m+1}}}\cdots Z^{\gamma_0}_{K,j_{\tau_{i+1}}}}{(y_{j_{\tau_{1}}}-y_{j_{\tau_{2}}})^{2}\times\cdots\times (y_{j_{\tau_{2m-1}}}-y_{j_{\tau_{2m}}})^{2}}\\
&=\sum_{m=0}^{\lfloor\frac{i}{2}\rfloor}\sum_{\sigma\in\mathfrak{J}_i^m}
\left(\sum_{r=2m+1}^i
\frac{2^{m+1}Z^{\gamma_0}_{K, j_{\sigma_{2m+1}}}\cdots Z^{\gamma_0}_{K, j_{\sigma_{r-1}}}Z^{\gamma_0}_{K, j_{\sigma_{r+1}}}\cdots Z^{\gamma_0}_{K, j_{\sigma_{i}}}}{(y_{j_{\sigma_{1}}}-y_{j_{\sigma_{2}}})^{2}\times\cdots\times (y_{j_{\sigma_{2m-1}}}-y_{j_{\sigma_{2m}}})^{2}\times(y_{j_{i+1}}-y_{j_{\sigma_r}})^2}+\frac{2^{m}Z^{\gamma_0}_{K,j_{\sigma_{2m+1}}}\cdots Z^{\gamma_0}_{K,j_{\sigma_{i}}}Z^{\gamma_0}_{K,j_{i+1}}}{(y_{j_{\sigma_{1}}}-y_{j_{\sigma_{2}}})^{2}\times\cdots\times (y_{j_{\sigma_{2m-1}}}-y_{j_{\sigma_{2m}}})^{2}}\right).\notag
\end{align}

For $\tau\in\mathfrak{J}_{i+1}^m$, let us consider the location of $i+1$ in $\tau$. If $\tau_{i+1}=i+1$, we define $\sigma_j=\tau_j$ for $1\le j\le i$, then $\sigma\in \mathfrak{J}_i^m$. Thus, for $0\le m\le \lfloor\frac{i}{2}\rfloor$, we have
\begin{align}\label{eqn::purepartitionfusionallaux4}
\sum_{\substack{\tau\in \mathfrak{J}_{i+1}^m: \\ \tau_{i+1}=i+1}}
\frac{2^{m}Z^{\gamma_0}_{K,j_{\tau_{2m+1}}}\cdots Z^{\gamma_0}_{K,j_{\tau_{i+1}}}}{(y_{j_{\tau_{1}}}-y_{j_{\tau_{2}}})^{2}\times\cdots\times (y_{j_{\tau_{2m-1}}}-y_{j_{\tau_{2m}}})^{2}}
=\sum_{\sigma\in\mathfrak{J}_i^m}
\frac{2^{m}Z^{\gamma_0}_{K,j_{\sigma_{2m+1}}}\cdots Z^{\gamma_0}_{K,j_{\sigma_{i}}}Z_{K, j_{i+1}}^{\gamma_0}}{(y_{j_{\sigma_{1}}}-y_{j_{\sigma_{2}}})^{2}\times\cdots\times (y_{j_{\sigma_{2m-1}}}-y_{j_{\sigma_{2m}}})^{2}}.
\end{align}
If $\tau_{i+1}<i+1$, we define a mapping for each $m\in \{1, 2, \ldots, \lfloor\frac{i}{2}\rfloor\}$: 
\begin{align*}
T_{m}\, :\, \{\tau\in \mathfrak{J}_{i+1}^m: \tau_{i+1}<i+1\} \longrightarrow \mathfrak{J}_i^{m-1}\times\{2m-1,\ldots,i\}, \quad \tau\mapsto (\sigma, r)
\end{align*}
in the following way. For $\tau\in\mathfrak{J}_{i+1}^m$ and $\tau_{i+1}<i+1$, we must have $\tau_{2k}=i+1$ for some $1\le k\le m$. We set $\sigma_{j}=\tau_{j}$, for $j\le 2k-2$; we set $\sigma_{j}=\tau_{j+2}$, for $2k-1\le j\le 2m-2$; and we set $\{\sigma_{2m-1},\ldots,\sigma_{i}\}=\{\tau_{2k-1},\tau_{2m+1},\ldots,\tau_{i+1}\}$ such that $\sigma_{2m-1}<\ldots<\sigma_{i}$. Suppose $\sigma_{r}=\tau_{2k-1}$ for some $r\in\{2m-1,\ldots,i\}$. This defines the map $T_{m}(\tau)=(\sigma, r)$. We argue that $T_m$ is a bijection. For any $(\sigma,r)\in\frak{J}_i^{m-1}\times\{2m-1,\ldots,i\}$, we can define $\tau$ as follows: $\{\{\tau_{1},\tau_{2}\},\ldots,\{\tau_{2m-1},\tau_{2m}\}\}=\{\{\sigma_{1},\sigma_{2}\},\ldots,\{\sigma_{2m-3},\sigma_{2m-2}\},\{\sigma_{r},i+1\}\}$ and $\{\tau_{2m+1},\ldots,\tau_{i+1}\}=\{\sigma_{2m-1},\ldots,\sigma_{r-1},\sigma_{r+1},\ldots,\sigma_{i}\}$, such that $\tau_1<\tau_3<\cdots<\tau_{2m-1}$, $\tau_{2j-1}<\tau_{2j}$ for $j\le m$ and $\tau_{2m+1}<\cdots<\tau_{i+1}$. Then we have $\tau\in \mathfrak{J}_{i+1}^m$ and $\tau_{i+1}<i+1$. This implies $T_{m}$ is a bijection. Thus, we have 
\begin{align}\label{eqn::purepartitionfusionallaux5}
\sum_{m=1}^{\lfloor\frac{i}{2}\rfloor}\sum_{\substack{\tau\in \mathfrak{J}_{i+1}^m: \\ \tau_{i+1}<i+1}}
\frac{2^{m}Z^{\gamma_0}_{K,j_{\tau_{2m+1}}}\cdots Z^{\gamma_0}_{K,j_{\tau_{i+1}}}}{(y_{j_{\tau_{1}}}-y_{j_{\tau_{2}}})^{2}\times\cdots\times (y_{j_{\tau_{2m-1}}}-y_{j_{\tau_{2m}}})^{2}}
&=\sum_{m=1}^{\lfloor\frac{i}{2}\rfloor}\sum_{\sigma\in\mathfrak{J}_i^{m-1}}
\sum_{r=2m-1}^i
\frac{2^{m}Z^{\gamma_0}_{K, j_{\sigma_{2m-1}}}\cdots Z^{\gamma_0}_{K, j_{\sigma_{r-1}}}Z^{\gamma_0}_{K, j_{\sigma_{r+1}}}\cdots Z^{\gamma_0}_{K, j_{\sigma_{i}}}}{(y_{j_{\sigma_{1}}}-y_{j_{\sigma_{2}}})^{2}\times\cdots\times (y_{j_{\sigma_{2m-3}}}-y_{j_{\sigma_{2m-2}}})^{2}\times(y_{j_{i+1}}-y_{j_{\sigma_r}})^2}\notag\\
&=\sum_{m=0}^{\lfloor\frac{i}{2}\rfloor-1}\sum_{\sigma\in\mathfrak{J}_i^m}
\sum_{r=2m+1}^i
\frac{2^{m+1}Z^{\gamma_0}_{K, j_{\sigma_{2m+1}}}\cdots Z^{\gamma_0}_{K, j_{\sigma_{r-1}}}Z^{\gamma_0}_{K, j_{\sigma_{r+1}}}\cdots Z^{\gamma_0}_{K, j_{\sigma_{i}}}}{(y_{j_{\sigma_{1}}}-y_{j_{\sigma_{2}}})^{2}\times\cdots\times (y_{j_{\sigma_{2m-1}}}-y_{j_{\sigma_{2m}}})^{2}\times(y_{j_{i+1}}-y_{j_{\sigma_r}})^2}
\end{align}

Combining~\eqref{eqn::purepartitionfusionallaux4} and~\eqref{eqn::purepartitionfusionallaux5}, we obtain~\eqref{eqn::purepartitionfusionallaux3} for even $i$. 
Next, suppose $i$ is odd and denote $\ell=\frac{i+1}{2}$. By~\eqref{eqn::purepartitionfusionallaux4} and~\eqref{eqn::purepartitionfusionallaux5}, we have
\begin{align*}
&\sum_{m=0}^{\ell}\sum_{\substack{\tau\in\mathfrak{J}_{i+1}^m: \\\tau_{i+1}=i+1}}
\frac{2^{m}Z^{\gamma_0}_{K,j_{\tau_{2m+1}}}\cdots Z^{\gamma_0}_{K,j_{\tau_{i+1}}}}{(y_{j_{\tau_{1}}}-y_{j_{\tau_{2}}})^{2}\times\cdots\times (y_{j_{\tau_{2m-1}}}-y_{j_{\tau_{2m}}})^{2}}\\
&=\sum_{m=0}^{\lfloor\frac{i}{2}\rfloor}\sum_{\sigma\in\mathfrak{J}_i^m}
\frac{2^{m}Z^{\gamma_0}_{K,j_{\sigma_{2m+1}}}\cdots Z^{\gamma_0}_{K,j_{\sigma_{i}}}Z_{K, j_{i+1}}^{\gamma_0}}{(y_{j_{\sigma_{1}}}-y_{j_{\sigma_{2}}})^{2}\times\cdots\times (y_{j_{\sigma_{2m-1}}}-y_{j_{\sigma_{2m}}})^{2}}
+\sum_{\substack{\tau\in \mathfrak{J}_{i+1}^{\ell}: \\\tau_{i+1}=i+1}}
\frac{2^{\ell}}{(y_{j_{\tau_{1}}}-y_{j_{\tau_{2}}})^{2}\times\cdots\times (y_{j_{\tau_{i}}}-y_{j_{\tau_{i+1}}})^{2}}. \\
&\sum_{m=0}^{\ell}\sum_{\substack{\tau\in\mathfrak{J}_{i+1}^m: \\\tau_{i+1}<i+1}}
\frac{2^{m}Z^{\gamma_0}_{K,j_{\tau_{2m+1}}}\cdots Z^{\gamma_0}_{K,j_{\tau_{i+1}}}}{(y_{j_{\tau_{1}}}-y_{j_{\tau_{2}}})^{2}\times\cdots\times (y_{j_{\tau_{2m-1}}}-y_{j_{\tau_{2m}}})^{2}}\\
&=\sum_{m=1}^{\lfloor\frac{i}{2}\rfloor}\sum_{\substack{\tau\in \mathfrak{J}_{i+1}^m: \\ \tau_{i+1}<i+1}}
\frac{2^{m}Z^{\gamma_0}_{K,j_{\tau_{2m+1}}}\cdots Z^{\gamma_0}_{K,j_{\tau_{i+1}}}}{(y_{j_{\tau_{1}}}-y_{j_{\tau_{2}}})^{2}\times\cdots\times (y_{j_{\tau_{2m-1}}}-y_{j_{\tau_{2m}}})^{2}}+
\sum_{\substack{\tau\in \mathfrak{J}_{i+1}^{\ell}: \\\tau_{i+1}<i+1}}
\frac{2^{\ell}}{(y_{j_{\tau_{1}}}-y_{j_{\tau_{2}}})^{2}\times\cdots\times (y_{j_{\tau_{i}}}-y_{j_{\tau_{i+1}}})^{2}}\\
&=\sum_{m=0}^{\lfloor\frac{i}{2}\rfloor-1}\sum_{\sigma\in\mathfrak{J}_i^m}
\sum_{r=2m+1}^i
\frac{2^{m+1}Z^{\gamma_0}_{K, j_{\sigma_{2m+1}}}\cdots Z^{\gamma_0}_{K, j_{\sigma_{r-1}}}Z^{\gamma_0}_{K, j_{\sigma_{r+1}}}\cdots Z^{\gamma_0}_{K, j_{\sigma_{i}}}}{(y_{j_{\sigma_{1}}}-y_{j_{\sigma_{2}}})^{2}\times\cdots\times (y_{j_{\sigma_{2m-1}}}-y_{j_{\sigma_{2m}}})^{2}\times(y_{j_{i+1}}-y_{j_{\sigma_r}})^2}+
\sum_{\substack{\tau\in \mathfrak{J}_{i+1}^{\ell}: \\\tau_{i+1}<i+1}}
\frac{2^{\ell}}{(y_{j_{\tau_{1}}}-y_{j_{\tau_{2}}})^{2}\times\cdots\times (y_{j_{\tau_{i}}}-y_{j_{\tau_{i+1}}})^{2}}\\
\end{align*}
Combining these two, in order to get~\eqref{eqn::purepartitionfusionallaux3}, it remains to show 
\begin{equation}\label{eqn::purepartitionfusionallaux6}
\sum_{\tau\in \mathfrak{J}_{i+1}^{\ell}}
\frac{2^{\ell}}{(y_{j_{\tau_{1}}}-y_{j_{\tau_{2}}})^{2}\times\cdots\times (y_{j_{\tau_{i}}}-y_{j_{\tau_{i+1}}})^{2}}
=\sum_{\sigma\in\mathfrak{J}_i^{\ell-1}}
\frac{2^{\ell}}{(y_{j_{\sigma_1}}-y_{j_{\sigma_2}})^2\times\cdots\times(y_{j_{\sigma_{i-2}}}-y_{j_{\sigma_{i-1}}})^2\times (y_{j_{i+1}}-y_{j_{\sigma_i}})^2}.
\end{equation}

To derive~\eqref{eqn::purepartitionfusionallaux6}, we define
$T\, :\, \mathfrak{J}_{i+1}^{\ell}\longrightarrow \mathfrak{J}_i^{\ell-1}$
in the following way. For $\tau\in\mathfrak{J}_{i+1}^{\ell}$, we must have $\tau_{2k}=i+1$ for some $1\le k\le \ell$. We set $\sigma_{j}=\tau_{j}$, for $j\le 2k-2$; we set $\sigma_{j}=\tau_{j+2}$, for $2k-1\le j\le 2\ell-2$; and we set $\sigma_{i}=\tau_{2k-1}$.  This defines the map $T(\tau)=\sigma$. We argue that $T$ is a bijection. For any $\sigma \in\frak{J}_i^{\ell-1}$, we can define $\tau$ as follows: $\{\{\tau_{1},\tau_{2}\},\ldots,\{\tau_{i},\tau_{i+1}\}\}=\{\{\sigma_{1},\sigma_{2}\},\ldots,\{\sigma_{i-2},\sigma_{i-1}\},\{\sigma_{i},i+1\}\}$ such that $\tau_1<\tau_3<\cdots<\tau_{i}$, $\tau_{2j-1}<\tau_{2j}$ for $j\le \ell$ . Then we have $\tau\in \mathfrak{J}_{i+1}^{\ell}$. This implies $T$ is a bijection, and gives~\eqref{eqn::purepartitionfusionallaux6}. Hence, we complete the proof of~\eqref{eqn::purepartitionfusionallaux3} for odd $i$, and complete the proof of~\eqref{eqn::purepartitionfusionallaux3}. 

Eq.~\eqref{eqn::purepartitionfusionallaux3} gives~\eqref{eqn::purepartitionfusionallaux2} for $\#K=i+1$ and completes the induction. Hence, Eq.~\eqref{eqn::purepartitionfusionallaux2} holds for all $K\subset J$. 
Taking $K=J$ in~\eqref{eqn::purepartitionfusionallaux2}, we obtain~\eqref{eqn::purepartitionfusionallaux1}. This completes the proof. 
\end{landscape}
\end{proof}

\begin{proof}[Proof of Proposition~\ref{prop::mGFFpdecov}]
The conformal covariance~\eqref{eqn::mGFFcov} is a consequence of~\eqref{eqn::multipleSLEsCOV} and the existence of the limit~\eqref{eqn::purepartitionfusionall}. It remains to show the third order PDE~\eqref{eqn::mGFFpde3rd}. We will show~\eqref{eqn::mGFFpde3rd} with $j=1$, and the other cases can be proved similarly. For $y_1<\cdots<y_{2N}$, we denote $y_{kl}=y_k-y_l$ for $k\neq l$. For $x_1<\cdots<x_{4N}$, we denote $x_{kl}=x_k-x_l$ for $k\neq l$. 

Fix $N\ge 1$ and the index valences $\varsigma=(2, \ldots, 2)$ of length $2N$. 
Fix $\hat{\alpha}\in\LP_{\varsigma}$ and let $\alpha=\tau(\hat{\alpha})\in\Pair_{2N}$ be the associated planar pair partition as defined in Section~\ref{subsec::linkpatterntopairpartition}. 
We set $F_0(x_1, \ldots, x_{4N})=\PartF_{\alpha}(x_1, \ldots, x_{4N})$ for $x_1<\cdots<x_{4N}$. We define $F_j$ by induction on $j$. 
Fix $j\in\{1,2, \ldots, 2N\}$ and suppose $F_{j-1}$ is defined. For $y_1<\cdots<y_j<x_{2j+1}<\cdots<x_{4N}$ and $y_{j-1}<x_{2j-1}<x_{2j}<x_{2j+1}$, we define
\begin{equation*}
F_j(y_1, \ldots, y_j, x_{2j+1}, \ldots, x_{4N}):=\lim_{x_{2j-1}, x_{2j}\to y_j}\frac{F_{j-1}(y_1, \ldots, y_{j-1}, x_{2j-1}, x_{2j}, \ldots, x_{4N})}{(x_{2j}-x_{2j-1})^{1/2}}. 
\end{equation*}
From Proposition~\ref{prop::purepartitionfusionall}, we see that $F_1, \ldots, F_{2N}$ are well-defined and $F_{2N}=\PartF_{\hat{\alpha}}$. We will show the following PDE by induction on $j\in\{1,2,\ldots, 2N\}$: 
\begin{align}\label{eqn::pde3rdintermediate}
&\LD_j F_j(y_1, \ldots, y_j, x_{2j+1}, \ldots, x_{4N})=0, 
\end{align}where 
\begin{align*}
\LD_j=\frac{\partial^3}{\partial y_1^3}
&-4\left(\sum_{2\le i\le j}\left(\frac{1}{y_{i1}^2}-\frac{1}{y_{i1}}\frac{\partial}{\partial y_i}\right)+\sum_{2j+1\le i\le 4N}\left(\frac{1/4}{(x_i-y_1)^2}-\frac{1}{(x_i-y_1)}\frac{\partial}{\partial x_i}\right)\right)\frac{\partial}{\partial y_1}\\
&+2\left(\sum_{2\le i\le j}\left(\frac{2}{y_{i1}^3}-\frac{1}{y_{i1}^2}\frac{\partial}{\partial y_i}\right)+\sum_{2j+1\le i\le 4N}\left(\frac{1/2}{(x_i-y_1)^3}-\frac{1}{(x_i-y_1)^2}\frac{\partial}{\partial x_i}\right)\right). 
\end{align*}

When $j=1$, PDE~\eqref{eqn::pde3rdintermediate} holds due to~\eqref{eqn::pdefusion3rd} with $j=1$. For $j\ge 2$, suppose~\eqref{eqn::pde3rdintermediate} holds for $j-1$, and we will show it for $j$. Comparing the two operators $\LD_{j-1}$ and $\LD_{j}$, we denote their overlap by
\begin{align*}
\LO_{j-1}=\frac{\partial^3}{\partial y_1^3}
&-4\left(\sum_{2\le i\le j-1}\left(\frac{1}{y_{i1}^2}-\frac{1}{y_{i1}}\frac{\partial}{\partial y_i}\right)+\sum_{2j+1\le i\le 4N}\left(\frac{1/4}{(x_i-y_1)^2}-\frac{1}{(x_i-y_1)}\frac{\partial}{\partial x_i}\right)\right)\frac{\partial}{\partial y_1}\\
&+2\left(\sum_{2\le i\le j-1}\left(\frac{2}{y_{i1}^3}-\frac{1}{y_{i1}^2}\frac{\partial}{\partial y_i}\right)+\sum_{2j+1\le i\le 4N}\left(\frac{1/2}{(x_i-y_1)^3}-\frac{1}{(x_i-y_1)^2}\frac{\partial}{\partial x_i}\right)\right). 
\end{align*}
Then, we have
\begin{align*}
\LD_{j-1}=\LO_{j-1}&+\left(\frac{-1}{(x_{2j-1}-y_1)^2}+\frac{4}{(x_{2j-1}-y_1)}\frac{\partial}{\partial x_{2j-1}}\right)\frac{\partial}{\partial y_1}
+\left(\frac{-1}{(x_{2j}-y_1)^2}+\frac{4}{(x_{2j}-y_1)}\frac{\partial}{\partial x_{2j}}\right)\frac{\partial}{\partial y_1}\\
&+\left(\frac{1}{(x_{2j-1}-y_1)^3}+\frac{-2}{(x_{2j-1}-y_1)^2}\frac{\partial}{\partial x_{2j-1}}\right)
+\left(\frac{1}{(x_{2j}-y_1)^3}+\frac{-2}{(x_{2j}-y_1)^2}\frac{\partial}{\partial x_{2j}}\right)\\
\LD_j=\LO_{j-1}&+\left(\frac{-4}{y_{j1}^2}+\frac{4}{y_{j1}}\frac{\partial}{\partial y_j}\right)\frac{\partial}{\partial y_1}
+\left(\frac{4}{y_{j1}^3}+\frac{-2}{y_{j1}^2}\frac{\partial}{\partial y_j}\right). 
\end{align*}
We set $G_{j-1}=(x_{2j}-x_{2j-1})^{-1/2}F_{j-1}$. From $\LD_{j-1}F_{j-1}=0$, we have
%0=&\LO_{j-1}G_{j-1}
%+\left(\frac{-1}{(x_{2j-1}-y_1)^2}+\frac{-1}{(x_{2j}-y_1)^2}\right)\frac{\partial}{\partial y_1}G_{j-1}+\left(\frac{1}{(x_{2j-1}-y_1)^3}+\frac{1}{(x_{2j}-y_1)^3}\right)G_{j-1}\\
%&+\left(\frac{4}{(x_{2j-1}-y_1)}\frac{\partial}{\partial y_1}+\frac{-2}{(x_{2j-1}-y_1)^2}\right)\left(\frac{-1}{2\eps}G_{j-1}+\frac{\partial}{\partial x_{2j-1}}G_{j-1}\right)\\
%&+\left(\frac{4}{(x_{2j}-y_1)}\frac{\partial}{\partial y_1}+\frac{-2}{(x_{2j}-y_1)^2}\right)\left(\frac{1}{2\eps}G_{j-1}+\frac{\partial}{\partial x_{2j}}G_{j-1}\right)\\
\begin{align}\label{eqn::3rdpdeaux1}
0=&\LO_{j-1}G_{j-1}
+\left(\frac{-1}{(x_{2j-1}-y_1)^2}+\frac{-1}{(x_{2j}-y_1)^2}+\frac{-2}{(x_{2j-1}-y_1)(x_{2j}-y_1)}\right)\frac{\partial}{\partial y_1}G_{j-1}\notag\\
&+\left(\frac{1}{(x_{2j-1}-y_1)^3}+\frac{1}{(x_{2j}-y_1)^3}+\frac{(x_{2j}-x_{2j-1})+2(x_{2j-1}-y_1)}{(x_{2j-1}-y_1)^2(x_{2j}-y_1)^2}\right)G_{j-1}\notag\\
&+\left(\frac{4}{(x_{2j-1}-y_1)}\frac{\partial}{\partial y_1}+\frac{-2}{(x_{2j-1}-y_1)^2}\right)\frac{\partial}{\partial x_{2j-1}}G_{j-1}+\left(\frac{4}{(x_{2j}-y_1)}\frac{\partial}{\partial y_1}+\frac{-2}{(x_{2j}-y_1)^2}\right)\frac{\partial}{\partial x_{2j}}G_{j-1}.
\end{align}
We will argue that
\begin{align}\label{eqn::3rdpdeaux2}
\LK G_{j-1}(y_1, \ldots, y_{j-1}, x_{2j-1}, \ldots, x_{4N})\to \LK F_j(y_1, \ldots, y_j, x_{2j+1}, \ldots, x_{4N}),\quad\text{as }x_{2j-1}, x_{2j}\to y_j, 
\end{align}
where 
\[\LK\in\left\{1, \, \frac{\partial}{\partial y_1},\,\frac{\partial^3}{\partial y_1^3}\right\}\cup\left\{\frac{\partial}{\partial y_i},\,\frac{\partial^2}{\partial y_i\partial y_1}: 2\le i\le j-1\right\}\cup\left\{\frac{\partial}{\partial x_n}, \,\frac{\partial^2}{\partial x_n\partial y_1}: 2j+1\le n\le 4N \right\};\]
and that
\begin{align}
&\left(\frac{1}{(x_{2j-1}-y_1)^2}\frac{\partial}{\partial x_{2j-1}}+\frac{1}{(x_{2j}-y_1)^2}\frac{\partial}{\partial x_{2j}}\right)G_{j-1}(y_1, \ldots, y_{j-1}, x_{2j-1}, \ldots, x_{4N})\notag\\
&\to \frac{1}{y_{j1}^2}\frac{\partial}{\partial y_j}F_j(y_1, \ldots, y_j, x_{2j+1}, \ldots, x_{4N}),\quad\text{as }x_{2j-1}, x_{2j}\to y_j,\label{eqn::3rdpdeaux3}\\
&\left(\frac{1}{(x_{2j-1}-y_1)}\frac{\partial}{\partial x_{2j-1}}+\frac{1}{(x_{2j}-y_1)}\frac{\partial}{\partial x_{2j}}\right)\frac{\partial}{\partial y_1}G_{j-1}(y_1, \ldots, y_{j-1}, x_{2j-1}, \ldots, x_{4N})\notag\\
&\to \frac{1}{y_{j1}}\frac{\partial^2}{\partial y_j\partial y_1}F_j(y_1, \ldots, y_j, x_{2j+1}, \ldots, x_{4N}),\quad\text{as }x_{2j-1}, x_{2j}\to y_j. \label{eqn::3rdpdeaux4}
\end{align}

From the proof of Proposition~\ref{prop::purepartitionfusionall}, Eq.~\eqref{eqn::3rdpdeaux2} holds for $\LK=1$. Furthermore, as $G_{j-1}$ is a finite linear combination of terms of the form 
$\prod (x_a-x_b)^{\pm 1/2}\times \prod (y_k-y_l)^{\pm 2}\times \prod (x_n-y_m)^{\pm 1}$, we may view $G_{j-1}$ as a function on distinct complex variables $(y_1, \ldots, y_{j-1}, x_{2j-1}, x_{2j}, \ldots, x_{4N})$. We fix arbitrarily distinct complex points $(y_1, \ldots, y_{j-1}, x_{2j+1}, \ldots, x_{4N})$  and denote $\boldsymbol{y}=(y_1, \ldots, y_{j-1})$ and $\boldsymbol{x}=(x_{2j+1}, \ldots, x_{4N})$. Then $G_{j-1}$ is a meromorphic function of $\eps=x_{2j}-x_{2j-1}$  and its Laurent series can be written as: 
\begin{align*}
H_{j-1}(\boldsymbol{y},x_{2j-1},\epsilon,\boldsymbol{x}):=&G_{j-1}(\boldsymbol{y}, x_{2j-1},x_{2j-1}+\epsilon,\boldsymbol{x})\\
=&F_j(\boldsymbol{y}, x_{2j-1}, \boldsymbol{x})+\sum_{n\ge 1}K_n(\boldsymbol{y}, x_{2j-1}, \boldsymbol{x})\eps^n,
\end{align*}
where $K_n$ is a finite linear combination of terms of the form $\prod (x_a-x_b)^{p/2}\times \prod (y_k-y_l)^{\pm 2}\times \prod (x_n-y_m)^{q}$ with $p,q\in\Z$. 

Now, we prove~\eqref{eqn::3rdpdeaux3}. We have
\begin{align*}
&\left(\frac{1}{(x_{2j-1}-y_1)^2}\frac{\partial}{\partial x_{2j-1}}+\frac{1}{(x_{2j}-y_1)^2}\frac{\partial}{\partial x_{2j}}\right)G_{j-1}(\boldsymbol{y}, x_{2j-1}, x_{2j}, \boldsymbol{x})\notag\\
&=\left(\frac{1}{(x_{2j-1}-y_1)^2}\frac{\partial}{\partial x_{2j-1}}+\left(\frac{1}{(x_{2j}-y_1)^2}-\frac{1}{(x_{2j-1}-y_1)^{2}}\right)\frac{\partial}{\partial \epsilon}\right)H_{j-1}(\boldsymbol{y},x_{2j-1},\epsilon,\boldsymbol{x}).
\end{align*}
Thus, it suffices to prove, as $x_{2j-1}\to y_j, \eps\to 0$, 
\begin{align}\label{partialx}
\frac{\partial}{\partial x_{2j-1}}H_{j-1}(\boldsymbol{y},x_{2j-1},\epsilon,\boldsymbol{x})\to \frac{\partial}{\partial y_j}F_j(\boldsymbol{y}, y_j, \boldsymbol{x}),\quad\frac{\partial}{\partial\epsilon} H_{j-1}(\boldsymbol{y},x_{2j-1},\epsilon,\boldsymbol{x})\to K_{1}(\boldsymbol{y},y_{j},\boldsymbol{x}). 
\end{align}
%and
%\begin{align}\label{partialeps}
%&\frac{\partial}{\partial\epsilon} H_{j-1}(\boldsymbol{y},x_{2j-1},\epsilon,\boldsymbol{x})\to K_{1}(\boldsymbol{y},y_{j},\boldsymbol{x}),\quad \text{as }x_{2j-1}\to y_j, \eps\to 0. 
%\end{align}

We define $\delta_{0}=\min\{\frac{x_{2j+1}-y_{j}}{4},\frac{y_{j}-y_{j-1}}{4}\}$ and 
$H_{j-1}(\boldsymbol{y},x_{2j-1},0,\boldsymbol{x}):=F_j(\boldsymbol{y},x_{2j-1}, \boldsymbol{x})$.
Then, the function $(x_{2j-1}, \eps)\mapsto H_{j-1}(\boldsymbol{y},x_{2j-1},\epsilon,\boldsymbol{x})$ is continuous on $[y_j-\delta_{0},y_j+\delta_{0}]\times B(0,\delta_{0})$ where $B(0,\delta_{0}):=\{z\in\C:d(z,0)<\delta_{0}\}$. Moreover, for every $x_{2j-1}\in [y_j-\delta_{0},y_j+\delta_{0}]$, the function $\eps\mapsto H_{j-1}(\boldsymbol{y},x_{2j-1},\epsilon,\boldsymbol{x})$ is holomorphic in $B(0,\delta_{0})\setminus\{0\}$. Thus, the function $\eps\mapsto H_{j-1}(\boldsymbol{y},x_{2j-1},\epsilon,\boldsymbol{x})$ is holomorphic in $B(0,\delta_{0})$. Then, we have
\[K_n(\boldsymbol{y}, x_{2j-1}, \boldsymbol{x})=\frac{1}{2\pi i}\int_{\partial B(0,\delta_{0})}\frac{H_{j-1}(\boldsymbol{y},x_{2j-1},z,\boldsymbol{x})}{z^{n+1}}dz.\]

Note that, there exists $M=M(\boldsymbol{y}, \boldsymbol{x}, y_j)>0$ such that, for all $x_{2j-1}\in [y_{j}-\delta_{0},y_{j}+\delta_{0}]$ and $z\in B(0,2\delta_{0})\setminus B(0,\frac{\delta_{0}}{2})$, 
\[|H_{j-1}(\boldsymbol{y},x_{2j-1},z,\boldsymbol{x})|\le M, \quad\text{and}\quad \left|\frac{\partial}{\partial x_{2j-1}}H_{j-1}(\boldsymbol{y},x_{2j-1},z,\boldsymbol{x})\right|\le M. \]
Thus, we have
\[|K_n(\boldsymbol{y}, x_{2j-1}, \boldsymbol{x})|\le \frac{M}{\delta_{0}^{n}}, \quad\text{and}\quad
\left|\frac{\partial}{\partial x_{2j-1}}K_n(\boldsymbol{y}, x_{2j-1}, \boldsymbol{x})\right|=\left|\frac{1}{2\pi i}\int_{\partial B(0,\delta_{0})}\frac{\frac{\partial}{\partial x_{2j-1}}H_{j-1}(\boldsymbol{y},x_{2j-1},z,\boldsymbol{x})}{z^{n+1}}dz\right|\le \frac{M}{\delta_{0}^{n}}.
\]
These imply that, for every $x_{2j-1}\in [y_{j}-\delta_{0},y_{j}+\delta_{0}]$ and $\eps\in B(0,\frac{\delta_{0}}{2})$, 
\begin{align*}
\frac{\partial}{\partial x_{2j-1}}H_{j-1}(\boldsymbol{y}, x_{2j-1},\epsilon, \boldsymbol{x})=&\frac{\partial}{\partial x_{2j-1}}F_j(\boldsymbol{y}, x_{2j-1}, \boldsymbol{x})+\sum_{n\ge 1}\frac{\partial}{\partial x_{2j-1}}K_n(\boldsymbol{y}, x_{2j-1}, \boldsymbol{x})\eps^n,
\end{align*}
and
\begin{align*}
\frac{\partial}{\partial \epsilon}H_{j-1}(\boldsymbol{y}, x_{2j-1},\epsilon, x_{2j+1} \ldots, x_{4N})=\sum_{n\ge 1}nK_n(\boldsymbol{y}, x_{2j-1}, \boldsymbol{x})\eps^{n-1}.
\end{align*}
These give~\eqref{partialx}, and complete the proof of~\eqref{eqn::3rdpdeaux3}. Eq.~\eqref{eqn::3rdpdeaux2} and~\eqref{eqn::3rdpdeaux4} can be proved in a similar way. 

Plugging~\eqref{eqn::3rdpdeaux2}-\eqref{eqn::3rdpdeaux4} into~\eqref{eqn::3rdpdeaux1}, and letting $x_{2j-1}, x_{2j}\to y_j$, we obtain $\LD_j F_j=0$. This completes the proof of~\eqref{eqn::pde3rdintermediate}. Taking $j=2N$ in~\eqref{eqn::pde3rdintermediate}, we obtain the third order PDE~\eqref{eqn::mGFFpde3rd} as desired. This completes the proof. 
\end{proof}

We end this subsection by a discussion on the solutions to the system of $2N$ PDEs in Proposition~\ref{prop::mGFFpdecov}. From the proof, the collection $\{\PartF_{\hat{\alpha}}: \hat{\alpha}\in\LP_{\varsigma}\}$ are solutions for the PDE system. From Lemma~\ref{lem::conformalblock3rdpde}, the collection $\{\LU_{\beta}^4: \beta\in\DP_N\}$ are also solutions for the PDE system. It is an interesting question to figure out the dimension of the solution space, and to find a basis for it. 

\subsection{Proof of Theorem~\ref{thm::main}}
\label{subsec::mgfffinal}

In this section, we complete the proof of Theorem~\ref{thm::main}. Before that, we first address the coefficient $\LM_{\omega,\tau(\hat{\alpha})}$ in the theorem. 
\begin{lemma}\label{lem::thmmaincoefficient} Fix $N\ge 1$ and the index valences $\varsigma=(2, \ldots, 2)$ of length $2N$. 
Define $\omega\in\DP_{2N}$ to be:
\[\omega(4j)=0, \quad \omega(4j+1)=1,\quad\omega(4j+2)=2,\quad\omega(4j+3)=1,\quad\omega(4j+4)=0,\quad\forall j\in\{0,1,\ldots, N-1\}.\]
For any $\hat{\alpha}\in\LP_{\varsigma}$, let $\tau(\hat{\alpha})\in \Pair_{2N}$ be the associated planar pair partition as defined in Section~\ref{subsec::linkpatterntopairpartition}. Recall the definition of the incidence matrix $\LM$ from~\eqref{eqn::KWleincidencematrix}. 
Then 
\begin{equation}\label{eqn::planarkwrelation}
\LM_{\omega, \beta}=1\quad\text{ implies }\quad \beta=\tau(\hat{\alpha}) \text{ for some }\hat{\alpha}\in\LP_{\varsigma}. 
\end{equation}
However, the converse does not hold in general. 
\end{lemma}
\begin{proof}
It suffices to prove that $\wedge_{2j-1}\notin\beta$ for every $1\le j\le 2N$.
By definition, 
\[\omega=\{\{4j+1,4j+4\},\{4j+2,4j+3\}:1\le j\le N-1\}.\]
Note that $\LM_{\omega, \beta}=1$ implies there exists a $\sigma$ which is a permutation of $\{4j+3,4j+4:0\le j\le N-1\}$ such that 
\[ \beta = \{ \{4j+1, \sigma(4j+4)\}, \{4j+2, \sigma(4j+3)\}: 0\le j\le N-1 \}. \]
This implies $\wedge_{2j-1}\notin\beta$ for every $1\le j\le 2N$.
\end{proof}

\begin{figure}[ht!]
\begin{center}
\includegraphics[width=0.24\textwidth]{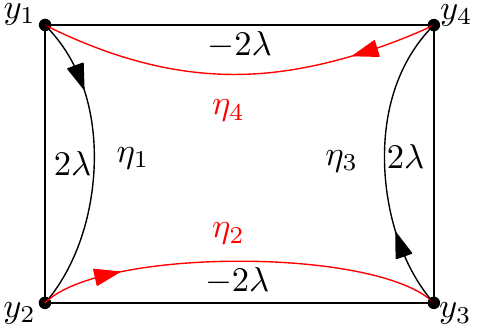}$\quad$
\includegraphics[width=0.24\textwidth]{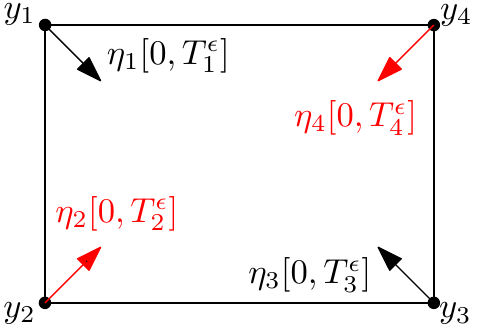}$\quad$
\includegraphics[width=0.24\textwidth]{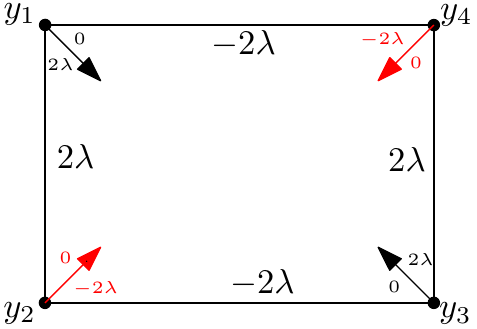}
\end{center}
\caption{\label{fig::mgffproof} Consider continuum GFF $\Gamma+u$ in rectangle with alternating boundary data. In the left panel, we have four level lines: Let $\eta_1$ (resp. $\eta_3$) be the level line of $\Gamma+u$ starting from $y_1$ (resp. from $y_3$) with height $\lambda$. These two curves are in black.Let $\eta_2$ (resp. $\eta_4$) be the level line of $-(\Gamma+u)$ starting from $y_2$ (resp. from $y_4$) with height $\lambda$. These two curves are in red. The middle panel indicates the domain $H_{\eps}$ which is obtained by removing from $\HH$ the four pieces $\eta_i[0,T_i^{\eps}]$ with $i=1,2, 3, 4$. 
In the right panel, we see that the boundary data of $\Gamma+u$ in $H_{\eps}$ is piecewise constant: $0, 2\lambda, 0, -2\lambda, 0, 2\lambda, 0, -2\lambda$.}
\end{figure}

\begin{proof}[Proof of Theorem~\ref{thm::main}]
We use the same notations as in Section~\ref{subsec::mgffcvg}.
By conformal invariance, we may assume $\Omega=\HH$ and $y_1<\cdots<y_{2N}$. Suppose $\Gamma$ is zero-boundary GFF on $\HH$ and let $u$ be the harmonic function with the boundary data~\eqref{eqn::mgff_boundaryconditions}.  
From Proposition~\ref{prop::convergence}, we have
\[\lim_{\delta\to 0}\PP[\LA^{\delta}=\hat{\alpha}]=\PP[\LA=\hat{\alpha}].\]
Let $\eta_{2j-1}$ be the level line of the continuum $\GFF$ $\Gamma+u$ starting from $y_{2j-1}$ with height $\lambda$ for $1\le j\le N$; and let $\eta_{2j}$ be the level line of $-(\Gamma+u)$ starting from $y_{2j}$ with height $\lambda$ for $1\le j\le N$. Note that the collection $\{\eta_2,\eta_4, \ldots, \eta_{2N}\}$ coincides with the collection of level lines $\Gamma+u$ starting from $y_{2j-1}$ with height $-\lambda$ for $1\le j\le N$.  See Figure~\ref{fig::mgffproof}. 
From Lemma~\ref{lem::construction}, the frontier of $\A$ and of $\AB$ has the same law as $\cup_{1\le j\le 2N}\eta_{j}$.  It suffices to prove 
\[\PP\left[\{\eta_1, \ldots, \eta_{2N}\}\text{ forms the planar link pattern }\hat\alpha\right]=\LM_{\omega, \tau(\hat{\alpha})}\frac{\PartF_{\hat{\alpha}}(y_1, \ldots, y_{2N})}{\PartF_{\mgff}^{(N)}(y_1, \ldots, y_{2N})},\]
where $\omega$ and $\LM_{\omega, \tau(\hat{\alpha})}$ are defined in Lemma~\ref{lem::thmmaincoefficient}. 

For $1\le j\le 2N$ and $\eps>0$ small, we denote 
$T_{j}^{\epsilon}=\inf\{t>0: d(\eta_{j}(t),y_{j})=\epsilon\}$.
We take $\phi^{\epsilon}$ to be the conformal map from 
\[H_{\eps}:=\HH\setminus \left(\cup_{1\le j\le 2N}\eta_{j}[0,T_{j}^{\epsilon}]\right)\] onto $\HH$ normalized at $\infty$. 
Then, we see that, given $H_{\eps}$, the event \[\{\{\eta_1, \ldots, \eta_{2N}\}\text{ forms the planar link pattern }\hat\alpha\}\] is the same as 
\[\{\{\phi^{\epsilon}(\eta_{1}), \ldots, \phi^{\epsilon}(\eta_{2N})\}\text{ forms the planar link pattern }\tau(\hat\alpha)\}\]
where $\tau(\hat\alpha)$ is defined in Section~\ref{subsec::linkpatterntopairpartition}. 

Now, let us consider the collection $\{\phi^{\epsilon}(\eta_{1}), \ldots, \phi^{\epsilon}(\eta_{2N})\}$. The conditional law of $\Gamma+u$ given $H_{\eps}$ is a GFF in $H_{\eps}$ with the following boundary data: for $1\le j\le 2N$, 
\begin{align*}
&2\lambda\text{ on }(y_{2j-1}^+, y_{2j}^-), \quad 0 \text{ along the left side of }\eta_{2j}[0, T_{2j}^{\eps}], \quad -2\lambda \text{ along the right side of }\eta_{2j}[0, T_{2j}^{\eps}],\\
-&2\lambda\text{ on }(y_{2j}^+, y_{2j+1}^-), \quad 0\text{ along the left side of }\eta_{2j+1}[0,T_{2j+1}^{\eps}],\quad 2\lambda\text{ along the right side of }\eta_{2j+1}[0,T_{2j+1}^{\eps}].
\end{align*}
See Figure~\ref{fig::mgffproof}.  
Then, we have 
\begin{align*}
&\PP\left[\{\eta_1, \ldots, \eta_{2N}\}\text{ forms the planar link pattern }\hat\alpha\right]\\
=&\E\left[\PP\left[\{\eta_1, \ldots, \eta_{2N}\}\text{ forms the planar link pattern }\hat\alpha\cond H_{\eps}\right]\right]\\
=&\E\left[\PP\left[\{\phi^{\epsilon}(\eta_{1}), \ldots, \phi^{\epsilon}(\eta_{2N})\}\text{ forms the planar link pattern }\tau(\hat\alpha)\cond H_{\eps}\right]\right]\\
=&\E\left[\LM_{\omega, \tau(\hat{\alpha})}\frac{\PartF_{\tau(\hat{\alpha})}(\phi^{\eps}(y_1^-), \phi^{\eps}(\eta_1(T_1^{\eps})), \phi^{\eps}(y_2^-), \phi^{\eps}(\eta_2(T_2^{\eps})), \ldots, \phi^{\eps}(y_{2N}^+), \phi^{\eps}(\eta_{2N}(T_{2N}^{\eps})))}{\LU_{\omega}(\phi^{\eps}(y_1^-), \phi^{\eps}(\eta_1(T_1^{\eps})), \phi^{\eps}(y_2^-), \phi^{\eps}(\eta_2(T_2^{\eps})), \ldots, \phi^{\eps}(y_{2N}^+), \phi^{\eps}(\eta_{2N}(T_{2N}^{\eps})))}\right]\\
=&\LM_{\omega, \tau(\hat{\alpha})}\frac{\PartF_{\hat{\alpha}}(y_1, \ldots, y_{2N})}{\LU_{(\omega)_2}^4(y_1, \ldots, y_{2N})}, 
\end{align*}
where $(\omega)_2$ is defined as in Lemma~\ref{lem::conformalblockfusionall}. 
In the second last equal sign, we use Theorem~\ref{thm::GFFconnectionproba}: consider the GFF in $H_{\eps}$, the collection $\{\phi^{\eps}(\eta_2), \phi^{\eps}(\eta_4), \ldots, \phi^{\eps}(\eta_{2N})\}$ coincides with the collection of level lines starting from $y_{2j-1}$ with height $-\lambda$. Therefore, the connection probability is given by $\LM_{\omega, \tau(\hat{\alpha})}\PartF_{\tau(\hat{\alpha})}/\LU_{\omega}$.
In the last equal sign, we let $\eps\to 0$. Combining Proposition~\ref{prop::purepartitionfusionall}, Lemma~\ref{lem::conformalblockfusionall}, and dominated convergence theorem, we obtain the conclusion. 

Finally, from the above analysis, we have 
\[\lim_{\delta\to 0}\PP[\LA^{\delta}=\hat{\alpha}]=\PP[\LA=\hat{\alpha}]=\LM_{\omega, \tau(\hat{\alpha})}\frac{\PartF_{\hat{\alpha}}(y_1, \ldots, y_{2N})}{\LU_{(\omega)_2}^4(y_1, \ldots, y_{2N})},\quad \text{for all }\hat{\alpha}\in\LP_{\varsigma}. \]
Furthermore, from~\eqref{eqn::purepartitionvsconformalblock} and~\eqref{eqn::planarkwrelation}, we have 
\[\sum_{\hat{\alpha}\in\LP_{\varsigma}}\LM_{\omega, \tau(\hat{\alpha})}\frac{\PartF_{\hat{\alpha}}(y_1, \ldots, y_{2N})}{\LU_{(\omega)_2}^4(y_1, \ldots, y_{2N})}=1. \]
Thus
\[\PartF_{\mgff}^{(N)}(y_1, \ldots, y_{2N}):=\sum_{\hat{\alpha}\in\LP_{\varsigma}}\LM_{\omega, \tau(\hat{\alpha})}\PartF_{\hat{\alpha}}(y_1, \ldots, y_{2N})=\LU_{(\omega)_2}^4(y_1, \ldots, y_{2N}). \]
This completes the proof of~\eqref{eqn::thmmainZtotal}. 
\end{proof}

\begin{figure}[ht!]
\begin{center}
\includegraphics[width=0.24\textwidth]{figures/gffboundarydataDP}$\quad$
\includegraphics[width=0.24\textwidth]{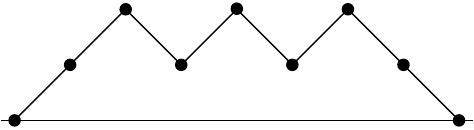}$\quad$
\includegraphics[width=0.24\textwidth]{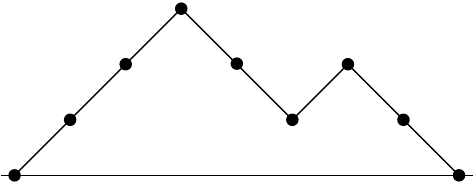}\\
\includegraphics[width=0.24\textwidth]{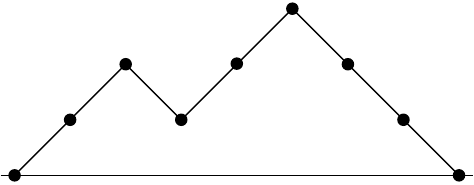}$\quad$
\includegraphics[width=0.24\textwidth]{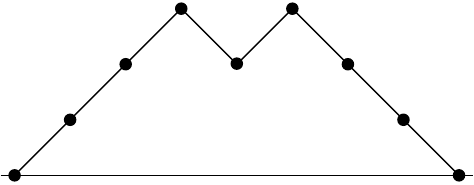}$\quad$
\includegraphics[width=0.24\textwidth]{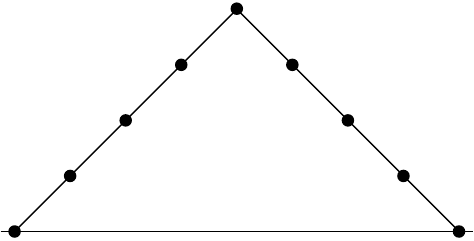}
\end{center}
\caption{\label{fig::gffboundarydata8solutions} There are six Dyck paths in this figure: in the first row, from left to right, we denote them by $\alpha_1, \alpha_2, \alpha_3$ respectively; in the second row, from left to right, we denote them by $\alpha_4, \alpha_5, \alpha_6$ respectively. We see that $\alpha_1\preceq\alpha_2\preceq\alpha_3, \alpha_4\preceq\alpha_5\preceq\alpha_6$. }
\end{figure}

\begin{corollary}\label{cor::crossingproba_mgff}
The conclusion in~\eqref{eqn::crossingproba_mgff} holds. 
\end{corollary}

\begin{proof}
We define $\beta_1, \beta_2, \beta_3$ as in Figure~\ref{fig::gfflevellines8} and we define $\alpha_1, \ldots, \alpha_6$ as in Figure~\ref{fig::gffboundarydata8solutions}. From~\eqref{eqn::purepartitionvsconformalblock}, we have
\begin{align}\label{eqn::eightpointslineartransform}
\begin{split}
\PartF_{\beta_1}&=\LU_{\alpha_2}-\LU_{\alpha_3}-\LU_{\alpha_4}+\LU_{\alpha_5}-2\LU_{\alpha_6},\\
\PartF_{\beta_2}&=\LU_{\alpha_6},\\
\PartF_{\beta_3}&=\LU_{\alpha_1}-\LU_{\alpha_2}+\LU_{\alpha_3}+\LU_{\alpha_4}-\LU_{\alpha_5}+\LU_{\alpha_6}. 
\end{split}
\end{align}
Suppose $y_1<y_2<y_3<y_4$ and we need to derive the limits as $x_1, x_2\to y_1$, $x_3, x_4\to y_2$, $x_5, x_6\to y_3$ and $x_7, x_8\to y_4$. 
We denote $x_{ji}=x_j-x_i$ for $1\le i<j\le 8$ and $y_{ji}=y_j-y_i$ for $1\le i<j\le 4$. Furthermore, we denote the cross-ratio by $q=(y_{21}y_{43})/(y_{31}y_{42})$. 

First, for $\alpha_1$ and $\alpha_6$, we have
\begin{align}
\lim_{\substack{x_1, x_2\to y_1; x_3, x_4\to y_2;\\ x_5, x_6\to y_3; x_7, x_8\to y_4}}\frac{\LU_{\alpha_1}(x_1, \ldots, x_8)}{\sqrt{x_{21}x_{43}x_{65}x_{87}}}&=\left(\frac{y_{31}y_{42}}{y_{21}y_{41}y_{32}y_{43}}\right)^2=\frac{1}{q^2(1-q)^2y_{31}^2y_{42}^2},\label{eqn::fusionalphaone}\\
\lim_{\substack{x_1, x_2\to y_1; x_3, x_4\to y_2;\\ x_5, x_6\to y_3; x_7, x_8\to y_4}}\frac{\LU_{\alpha_6}(x_1, \ldots, x_8)}{\sqrt{x_{21}x_{43}x_{65}x_{87}}}&=
\left(\frac{y_{21}y_{43}}{y_{31}y_{41}y_{32}y_{42}}\right)^2=\frac{q^2}{(1-q)^2y_{31}^2y_{42}^2}.\label{eqn::fusionalphasix}
\end{align}

Second, for $n=2,3,4,5$, we have
\begin{align*}
\lim_{\substack{x_1, x_2\to y_1;\\ x_7, x_8\to y_4}}\frac{\LU_{\alpha_n}(x_1, \ldots, x_8)}{\sqrt{x_{21}x_{87}}}=y_{41}^{-2}\times \prod_{3\le i\le 6}(x_i-y_1)^{\vartheta_{\alpha_n}(i,1)}(y_4-x_i)^{\vartheta_{\alpha_n}(i,7)}\times\prod_{3\le i<j\le 6}x_{ji}^{\frac{1}{2}\vartheta_{\alpha_n}(i,j)}. 
\end{align*}
Taking the difference between $\LU_{\alpha_2}$ and $\LU_{\alpha_3}$ and the difference between $\LU_{\alpha_4}$ and $\LU_{\alpha_5}$, we have
\begin{align*}
\lim_{\substack{x_1, x_2\to y_1; x_3, x_4\to y_2;\\ x_7, x_8\to y_4}}\frac{(\LU_{\alpha_2}-\LU_{\alpha_3})(x_1, \ldots, x_8)}{\sqrt{x_{21}x_{43}x_{87}}}
&=y_{41}^{-2}\times\frac{(x_6-y_1)}{(x_5-y_1)}\frac{(y_4-x_5)}{(y_4-x_6)}\times\left( \frac{\sqrt{x_{65}}}{(x_5-y_2)(x_6-y_2)}+\frac{2y_{41}}{y_{21}y_{42}\sqrt{x_{65}}}\right),\\
\lim_{\substack{x_1, x_2\to y_1; x_3, x_4\to y_2;\\ x_7, x_8\to y_4}}\frac{(\LU_{\alpha_4}-\LU_{\alpha_5})(x_1, \ldots, x_8)}{\sqrt{x_{21}x_{43}x_{87}}}
&=y_{41}^{-2}\times \frac{(x_5-y_1)(y_4-x_6)}{(x_6-y_1)(y_4-x_5)}\times\left(\frac{-\sqrt{x_{65}}}{(x_5-y_2)(x_6-y_2)}+\frac{2y_{41}}{y_{21}y_{42}\sqrt{x_{65}}}\right).
\end{align*}
Taking the difference between these two, we have
\begin{align}\label{eqn::fusionalpharest}
\lim_{\substack{x_1, x_2\to y_1; x_3, x_4\to y_2;\\ x_5, x_6\to y_3; x_7, x_8\to y_4}}\frac{(\LU_{\alpha_2}-\LU_{\alpha_3}-\LU_{\alpha_4}+\LU_{\alpha_5})(x_1, \ldots, x_8)}{\sqrt{x_{21}x_{43}x_{65}x_{87}}}
&=\frac{2}{y_{41}^2y_{32}^2}+\frac{4}{y_{21}y_{31}y_{42}y_{43}}=\left(\frac{2}{(1-q)^2}+\frac{4}{q}\right)\frac{1}{y_{31}^2y_{42}^2}. 
\end{align}

Plugging~\eqref{eqn::fusionalphaone}, \eqref{eqn::fusionalphasix} and~\eqref{eqn::fusionalpharest} into~\eqref{eqn::eightpointslineartransform}, we have
\begin{align*}
\lim_{\substack{x_1, x_2\to y_1; x_3, x_4\to y_2;\\ x_5, x_6\to y_3; x_7, x_8\to y_4}}\frac{\PartF_{\beta_1}(x_1, \ldots, x_8)}{\LU_{\alpha_1}(x_1,\ldots, x_8)}
&=2q(1-q)(2-q+q^2),\\
\lim_{\substack{x_1, x_2\to y_1; x_3, x_4\to y_2;\\ x_5, x_6\to y_3; x_7, x_8\to y_4}}\frac{\PartF_{\beta_2}(x_1, \ldots, x_8)}{\LU_{\alpha_1}(x_1,\ldots, x_8)}
&=q^4,\\
\lim_{\substack{x_1, x_2\to y_1; x_3, x_4\to y_2;\\ x_5, x_6\to y_3; x_7, x_8\to y_4}}\frac{\PartF_{\beta_3}(x_1, \ldots, x_8)}{\LU_{\alpha_1}(x_1,\ldots, x_8)}
&=(1-q)^4. 
\end{align*}
The scaling limit of the crossing probability in~\eqref{eqn::crossingproba_mgff} corresponds to the limit of $\PartF_{\beta_2}/\LU_{\alpha_1}$, see Figure~\ref{fig::mgfflinkpatterns} and Figure~\ref{fig::gfflevellines8}. This completes the proof. 
\end{proof}